\documentclass[11pt]{article}
\usepackage{geometry}
\usepackage{amssymb,amsmath}
 \usepackage{pstricks, pst-coil, pst-node, pst-tree, multido}
 \usepackage[english]{babel}

\usepackage{graphics}
\usepackage{graphicx}
 \newtheorem{DE}{Definition}[section]

\newcommand {\sm} {\setminus}

\usepackage{latexsym} 
\usepackage{theorem} 
 \newcommand{\qed}{\relax\ifmmode\hskip2em\Box\else\unskip\nobreak\hfill$\Box$\fi}

\newtheorem{defeng}[DE]{Definition}
\newtheorem{theorem}[DE]{Theorem}
\newtheorem{lemma}[DE]{Lemma}

\newtheorem{corollary}[DE]{Corollary}
\newtheorem{remark}[DE]{Remark}
{\theoremstyle{break}\theorembodyfont{\rmfamily}}
{\theoremstyle{break}\theorembodyfont{\rmfamily}}

\newcounter{claim}
\newenvironment{proof}[1][]%
	{\noindent {\setcounter{claim}{0}\sc proof --- }{#1}{}}{\qed\vspace{2ex}}
	{\refstepcounter{claim}\vspace{1ex}\noindent {(\it\arabic{claim}) {#1}{}}\it}{\vspace{1ex}}
	{\noindent {}{#1}{}}{ This proves~(\arabic{claim}).\vspace{1ex}}

\begin{document}

 \title{The (theta, wheel)-free graphs\\ Part IV: induced paths and cycles}
\author{Marko Radovanovi\'c\thanks{University
    of Belgrade, Faculty of Mathematics, Belgrade, Serbia. Partially
    supported by Serbian Ministry of Education, Science and
    Technological Development project 174033. E-mail:
    markor@matf.bg.ac.rs}~, Nicolas Trotignon\thanks{Univ Lyon, EnsL, UCBL, CNRS, LIP,
    F-69342, LYON Cedex 07, France. Partially
    supported by the LABEX MILYON ANR-10-LABX-0070. E-mail: nicolas.trotignon@ens-lyon.fr}~, Kristina Vu\v
  skovi\'c\thanks{School of Computing, University of Leeds, Leeds, UK and
    Faculty of Computer Science (RAF), Union University, Belgrade,
    Serbia.  Partially supported by EPSRC grant EP/N019660/1, and
    Serbian Ministry of Education, Science and
    Technological Development projects 174033 and
    III44006. E-mail: k.vuskovic@leeds.ac.uk}}

\maketitle

\begin{abstract}
  A hole in a graph is a chordless cycle of length at least~4. A theta
  is a graph formed by three internally vertex-disjoint paths of
  length at least~2 between the same pair of distinct vertices. A
  wheel is a graph formed by a hole and a node that has at least 3
  neighbors in the hole.  In this series of papers we study the class
  of graphs that do not contain as an induced subgraph a theta nor a
  wheel. In Part II of the series we prove a decomposition theorem for
  this class, that uses clique cutsets and 2-joins. In this paper we
  use this decomposition theorem to solve several problems related to
  finding induced paths and cycles in our class.
 \end{abstract}

 \section{Introduction}

 In this paper all graphs are finite and simple.  We say that a
 graph $G$ \emph{contains} a graph $H$ if $H$ is isomorphic to an
 induced subgraph of $G$, and that $G$ is \emph{$H$-free} if it does
 not contain $H$.  For a family of graphs ${\mathcal H}$, $G$ is
 \emph{${\mathcal H}$-free} if for every $H\in {\mathcal H}$, $G$ is $H$-free.

 A \emph{hole} in a graph is a chordless cycle of length at least 4.
 A \emph{theta} is a graph formed by three paths between the same pair
 of distinct vertices so that the union of any two of the paths
 induces a hole. This implies that each of the three paths has length
 at least~2 (a path of length~1 would form a chord of the cycle
 induced by the two other paths). A \emph{wheel} is a graph formed by
 a hole and a node that has at least 3 neighbors in the hole.

 In this series of papers we study the class of (theta, wheel)-free
 graphs, that we denote by $\mathcal C$ throughout the paper.  This
 project is motivated and explained in more detail in Part I of the
 series~\cite{twf-p1}, where two subclasses of $\mathcal C$ are
 studied.  In Part II of the series~\cite{twf-p2}, we prove a
 decomposition theorem for graphs in $\mathcal C$ that uses clique
 cutsets and 2-joins, and use it to obtain a polynomial time
 recognition algorithm for the class.  In Part III of the
 series~\cite{twf-p3} we use the decomposition theorem
 from~\cite{twf-p2} to obtain further properties of graphs in the
 class and to construct polynomial time algorithms for maximum weight
 clique, maximum weight stable set, and coloring problems.  In this
 part we use the decomposition theorem from~\cite{twf-p2} to obtain
 polynomial time algorithms for several problems related to finding
 induced paths and cycles.

 The \textsc{Disjoint Paths} problem is to test whether a graph $G$
 with $k$ pairs of specified vertices $(s_1,t_1), \ldots ,(s_k,t_k)$
 contains vertex-disjoint paths $P_1,\ldots ,P_k$ such that, for
 $i=1, \ldots ,k$, $P_i$ is a path from $s_i$ to $t_i$. When $k$ is
 part of the input, this problem is $\mathrm{NP}$-complete
 \cite{Karp-DisjointNP}.  If $k$ is any fixed integer (i.e.\ not part
 of the input) then the problem is called $k$-\textsc{Disjoint Paths}
 and can be solved in $\mathcal O(n^3)$ time as shown by the linkage
 algorithm of Robertson and Seymour~\cite{RS:GraphMinor13}.
 (Throughout the paper $n$ will denote the number of vertices and $m$
 the number of edges of an input graph.)  In this paper we consider
 the induced variant of this problem. We note that there are slight
 differences in the definition of this problem (for example in
 \cite{IDP-clawFPT}; see also Subsection \ref{sub:related}).

  \vspace{2ex}

  \noindent
  \textsc{Induced Disjoint Paths} problem $(G,\mathcal W)$
  \\
  {\bf Instance:} A graph $G$ and a
    set $\mathcal W=\{(s_1,t_1),(s_2,t_2),\ldots,(s_k,t_k)\}$ of pairs of
    vertices of $G$ such that all $2k$ vertices are distinct and the only
    edges between these vertices are of the form $s_it_i$, for some
    $1\leqslant i\leqslant k$.
  \\
  {\bf Question:} Does $G$ contain $k$ vertex-disjoint paths
  $P_i=s_i\ldots t_i$, $i\in \{ 1, \ldots ,k\}$, such that no vertex
  of $P_i$ is adjacent to a vertex of $P_j$, for every $i\neq j$?

  \vspace{2ex}

  For chordal
  graphs~\cite{DBLP:journals/algorithmica/BelmonteGHHKP14} this
  problem can be solved in time $\mathcal O(kn^3)$.  When $k$ is fixed
  (i.e.\ not part of the input) then the problem is called
  $k$-\textsc{Induced Disjoint Paths}.  The $k$-\textsc{Induced
    Disjoint Paths} is $\mathrm{NP}$-complete whenever $k\geq 2$ as
  proved by Bienstock~\cite{bienstock}, even when restricted to
  several classes of graphs (see \cite{leveque.lmt:detect}). It has
  been studied in several classes of graphs (but not many if compared
  to other problems such as graph coloring or maximum stable set for
  instance).  It can be solved for line graphs as a simple corollary
  of the linkage algorithm mentioned above. For any fixed $k$, it can
  be solved in linear time for planar
  graphs~\cite{DBLP:journals/jcss/KawarabayashiK12} and circular-arc
  graphs~\cite{DBLP:journals/tcs/GolovachPL16}, and in polynomial time for
  AT-free graphs~\cite{DBLP:conf/swat/GolovachPL12} and claw-free
  graphs~\cite{IDP-clawPoly}.

  Golovach, Paulusma and van Leeuewen \cite{IDP-clawFPT} proved
  that for claw-free graphs the $k$-\textsc{Induced Disjoint Paths} is
  fixed-parameter tractable when parameterized by $k$, meaning that it
  can be solved for any fixed $k$ in time $f(k) n^c$, where $c$ is a
  constant that does not depends on $k$, and $f$ is a computable
  function that depends only on $k$.

  Here, we prove that the \textsc{Induced Disjoint Paths} problem
  remains $\mathrm{NP}$-complete on $\mathcal C$, and we give a
  polynomial time algorithm for the $k$-\textsc{Induced Disjoint
    Paths} problem on $\mathcal C$.  We also consider a number of
  related problems, namely $k$-{\sc in-a-Path}, $k$-{\sc in-a-Tree}
  and $H$-{\sc Induced-Topological Minor}, to be defined in
  Section~\ref{sec:paths}.  These problems were already studied in
  differents settings and classes of graphs, see~\cite{chudnovsky.seymour:theta, nicolas.d.p:fourTree, DBLP:journals/algorithmica/FellowsKMP95,nicolas.wei:kTree}.

  \subsection*{Outline of the paper}

  In Section~\ref{sec:decTh} we state several results for the class
  ${\mathcal C}$ that are proved in previous parts of this series and that
  will be needed here. Most importantly, we give the decomposition
  theorem for ${\mathcal C}$ from~\cite{twf-p2} and recall some results
  that will help us manage the cutsets that appear in this
  decomposition theorem. In fact, fundamental for our algorithms are
  the 2-join decomposition techniques developed in
  \cite{nicolas.kristina:two}, which we also describe here.

  In Section \ref{sec:hole} we prove that for a graph $G\in\mathcal C$
  the only obstruction for the existence of an induced cycle through
  the given two non-adjacent vertices is a clique cutset of $G$ that
  separates these vertices.  Note that a similar result holds for
  claw-free graphs as proved by Bruhn and Saito~\cite{BruhnS12}.
  Using our result we derive an $\mathcal O(nm)$-time algorithm for
  the problem of finding and chordless cycle through two prescribed
  vertices of a graph in $\mathcal C$. Also, we show that $k$-\textsc{in-a-Cycle}
  problem is fixed-parameter tractable (when parameterized by $k$) for graphs
  in $\mathcal C$.

  In Section \ref{sec:paths} we give an $\mathcal O(n^{2k+6})$-time
  algorithm for the $k$-\textsc{Induced Disjoint Paths} for graphs in
  $\mathcal C$. As an intermediate step we show that the
  \textsc{Induced Disjoint Paths} problem is fixed-parameter tractable
  (when parameterized by $k$) for the class of graphs of $\mathcal C$
  that do not have clique cutsets. We also show that if $\mathcal G$
  is any hereditary class of graphs such that there exists an
  $\mathcal O(n^c)$-time algorithm (where $c$ is a constant that does
  not depend on $k$) for the $k$-\textsc{Induced Disjoint Paths} on
  the graphs of $\mathcal G$ that do not have clique cutsets, then the
  $k$-\textsc{Induced Disjoint Paths} problem can be solved on
  $\mathcal G$ in $\mathcal O(n^{2k+c})$-time.

  In Section \ref{sec:paths} we show that the generalization of the
  $k$-\textsc{Induced Disjoint Paths} problem where the terminals
  (i.e.\ vertices $s_1,\ldots ,s_k,t_1,\ldots ,t_k$) are not
  necessarily distinct and where there could possibly be edges between
  the terminal vertices, can be solved in $\mathcal O(n^{4k+6})$-time.
  Consequently we obtain polynomial-time algorithms for a number of
  related problems (on $\mathcal C$) such as $k$-\textsc{in-a-Path},
  $k$-\textsc{in-a-Tree} and
  $H$-\textsc{Induced Topological Minor}.

  Finally, in Section \ref{sec:NPC} we give several
  $\mathrm{NP}$-completeness results.

\subsection*{Terminology and notation}

A {\em clique} in a graph is a (possibly empty) set of pairwise
adjacent vertices.   A {\em stable set} in a graph is a (possibly empty) set
of pairwise nonadjacent vertices.  A {\it diamond} is a graph obtained
from a complete graph on 4 vertices by deleting an edge. A {\it claw}
is a graph induced by nodes $u,v_1,v_2,v_3$ and edges
$uv_1,uv_2,uv_3$.

A {\em path} $P$ is a sequence of distinct vertices
$p_1p_2\ldots p_k$, $k\geq 1$, such that $p_ip_{i+1}$ is an edge for
all $1\leqslant i <k$.  Edges $p_ip_{i+1}$, for $1\leqslant i <k$, are called
the {\em edges of $P$}.  Vertices $p_1$ and $p_k$ are the {\em endnodes}
of $P$.  A {\em cycle} $C$ is a sequence of vertices $p_1p_2\ldots p_kp_1$,
$k \geq 3$, such that $p_1\ldots p_k$ is a path and $p_1p_k$ is an
edge.  Edges $p_ip_{i+1}$, for $1\leqslant i <k$, and edge $p_1p_k$ are
called the {\em edges of $C$}.  Let $Q$ be a path or a cycle.  The
vertex set of $Q$ is denoted by $V(Q)$.  The {\em length} of $Q$ is
the number of its edges.  An edge $e=uv$ is a {\em chord} of $Q$ if
$u,v\in V(Q)$, but $uv$ is not an edge of $Q$. A path or a cycle $Q$
in a graph $G$ is {\em chordless} if no edge of $G$ is a chord of
$Q$.

Let $G$ be a graph. We denote the vertex set of $G$ by $V(G)$. For
$v\in V(G)$ and $T\subseteq V(G)$, $N_T(v)$
is the set of neighbors of $v$ in $T$.  Also, if $T=V(G)$, then we use
$N(v)$ to denote $N_T(v)$. For $S\subseteq V(G)$, $N_G(S)$ (or simply $N(S)$
when clear from context) denotes the set of vertices $u\in V(G)\setminus S$
such that for some $v\in S$, $uv$ is an edge of $G$. Also, for $S\subseteq V(G)$, $G[S]$ denotes the subgraph of $G$ induced by $S$ and $G\setminus S$ the
subgraph of $G$ induced by $V(G)\setminus S$.  For disjoint subsets $A$ and $B$ of
$V(G)$, we say that $A$ is {\em complete} (resp. {\em anticomplete})
to $B$ if every vertex of $A$ is adjacent (resp. nonadjacent) to every
vertex of $B$.

When clear from the context, we will sometimes write $G$ instead of
$V(G)$.

\section{Decomposition of (theta, wheel)-free graphs}
\label{sec:decTh}

To state the decomposition theorem for graphs in ${\mathcal C}$ we first
define the basic classes involved and then the cutsets used.

\subsection*{Basic classes}

If $R$ is a graph, then the {\em line graph} of $R$, denoted by
$L(R)$, is the graph whose vertices are the edges of $R$, and such
that two vertices of $L(R)$ are adjacent if and only if the
corresponding edges are adjacent in $R$. A graph $G$ is {\em
  chordless} if no cycle of $G$ has a chord.

In the decomposition theorem for (theta, wheel)-free graphs obtained
in \cite{twf-p2} we have two types of {\em basic graphs}: line graphs
of triangle-free chordless graphs and P-graphs. Their union is denoted
by $\mathcal B$. The definition of a P-graph is a bit technical (see
\cite{twf-p2}), and in this paper we do not need it in full.  We
therefore just list the properties that we need. Fortunately, these
properties are common to both types of basic graphs. Here they are:

\begin{itemize}
\item every graph $G\in\mathcal B$ is induced by $V(L(R))\cup K$, where
$V(L(R))$ and $K$ are disjoint;
\item $R$ is triangle-free and chordless and $K$ is a clique (it is
  possible that $K=\emptyset$);
\item each vertex of $L(R)$, that corresponds to an edge of $R$
  incident with a degree 1 vertex (in $R$), is adjacent to at most one
  vertex of $K$, and these are the only edges between $L(R)$ and $K$;
\item if a vertex $v$ of $L(R)$, that is of degree at least 2 in
  $L(R)$, is adjacent to a vertex $w$ of $K$, then $w$ has no other
  neighbors in $L(R)$ (note that this implies: if a vertex $v$ of
  $L(R)$ is adjacent to a vertex $w$ of $K$ that has some other
  neighbor in $L(R)$, then the degree of $v$ in $G$ is 2).
\end{itemize}

We say that $R$ is the skeleton and $K$ the special clique of $G$.
Clearly every line graph of a triangle-free chordless graph satisfies
the above properties. The fact that P-graphs satisfy them follows from
conditions (i), (vi) and (viii) in the definition of P-graphs in \cite{twf-p2}.
We also observe that all induced subgraphs of graphs in $\mathcal B$ satisfy
these properties. Note that each $G\in\mathcal B$ is
diamond-free (since $L(R)$ does not contain a diamond, each vertex of
$L(R)$ has at most one neighbor in $K$ and no vertex of $K$ is
adjacent to more than 1 vertex that is of degree at least 2 in
$L(R)$), and center of every claw of $G$ is contained in
$K$.

\subsection*{Cutsets}

In a graph $G$, a
subset $S$ of nodes and/or edges is a {\em cutset} if its removal yields
a disconnected graph.  A node cutset $S$ is a {\em clique cutset} if
$S$ is a clique.  Note that every disconnected graph has a clique
cutset: the empty set.

An {\em almost 2-join} in a graph $G$ is a pair $(X_1,X_2)$ that is a
partition of $V(G)$, and such that:

\begin{itemize}
\item For $i\in\{1,2\}$, $X_i$ contains disjoint nonempty sets $A_i$ and
  $B_i$, such that  $A_1$ is complete to $A_2$, $B_1$ is complete to
  $B_2$, and there are no other adjacencies between $X_1$ and $X_2$.
\item For $i\in\{1,2\}$, $|X_i|\geq 3$.
\end{itemize}

An almost 2-join $(X_1, X_2)$ is a \emph{2-join} when for $i\in\{1,2\}$, $X_i$
contains at least one path from $A_i$ to $B_i$, and if $|A_i|=|B_i|=1$
then $G[X_i]$ is not a chordless path.

We say that $(X_1,X_2,A_1,A_2,B_1,B_2)$ is a {\em split} of this (almost)
2-join, and the sets $A_1,A_2,B_1,B_2$ are the {\em special sets} of
this (almost) 2-join.  We often use the following notation:
$C_i = X_i\sm (A_i \cup B_i)$ (possibly, $C_i = \emptyset$).

We are ready to state the decomposition theorem from~\cite{twf-p2}.

\begin{theorem}[\cite{twf-p2}]\label{decomposeTW}
  If $G$ is (theta, wheel)-free, then $G$ is a line graph of a
  triangle-free chordless graph or a P-graph, or $G$ has a clique
  cutset or a 2-join.
\end{theorem}

It trivially follows that if $G\in \mathcal C$ then $G\in \mathcal B$
or $G$ has a clique cutset or a 2-join.  We now describe how we
decompose a graph from ${\mathcal C}$ into basic graphs using the cutsets
in the above theorem.

\subsection*{Decomposing with clique cutsets}

If a graph $G$ has a clique cutset $K$, then
its node set can be partitioned into sets $(A,K,B)$, where $A$ and $B$
are nonempty and   anticomplete.  We say that $(A,K,B)$
is a \emph{split} for the clique cutset $K$.  When $(A, K, B)$ is a
split for a clique cutset of a graph $G$, the {\em blocks of decomposition}
of $G$ with respect to $(A, K, B)$ are the graphs $G_A=
G[A\cup K]$ and $G_B= G[K \cup B]$.

A \emph{clique cutset decomposition tree of depth $p$} for a graph $G$ is a rooted
tree $T$ defined as follows.

\begin{itemize}
\item[(i)] The root of $T$ is $G_0=G$.
\item[(ii)] The non-leaf nodes of $T$ are $G_0,G_1,\ldots,G_{p-1}$. Each non-leaf node $G_i$ has two children: one is $G_{i+1}$ and the other one is $G_{i+1}^{B}$.

    The leaf-nodes of $T$ are graphs $G_1^B,G_2^B,\ldots,G_p^{B}$ and $G_p$.
    Graphs $G_1^B,G_2^B,\ldots,G_p^B,G_p$ have no clique cutset.

\item[(iii)] For $i\in\{0,1,\ldots,p-1\}$, $G_i$ has a clique cutset split $(A_i,K_i,B_i)$ and graphs $G_{i+1}=G[A_i\cup K_i]$ and $G_{i+1}^{B}=G[B_i\cup K_i]$ are blocks of decomposition of $G_i$ w.r.t.\ this clique cutset split.
\end{itemize}

Also, we set $B_p=A_{p-1}$, $K_p=K_{p-1}$ and $G_{p+1}^B=G_p$. So,
the leaves of $T$ are the graphs $G_{i+1}^B = G[K_i \cup B_i]$, for $i\in\{0,1,\dots, p\}$.

\begin{theorem}[\cite{tarjan}]
  \label{th:tarjan}
  A clique cutset decomposition tree of an input graph $G$ can be
  computed in time $\mathcal O(nm)$ and has $\mathcal O(n)$ nodes.
\end{theorem}

Note that for a non-leaf node $G_i$ of $T$ the corresponding
clique cutset $K_i$ is also a clique cutset of $G$.
The following lemmas proved in \cite{twf-p1} will also be needed.

\begin{lemma}[Lemma 3.2 in \cite{twf-p1}]\label{diamondCliqueCut}
  If $G$ is a wheel-free graph that contains a diamond, then $G$ has a
  clique cutset.
\end{lemma}

A {\em star cutset} in a graph is a node cutset $S$ that contains a
node (called a {\em center}) adjacent to all other nodes of $S$. Note
that a nonempty clique cutset is a star cutset.

\begin{lemma}[Lemma 3.3 in \cite{twf-p1}]\label{Star=Clique}
  If $G\in\mathcal C$ has a star cutset, then $G$ has a clique cutset.
\end{lemma}

\subsection*{Decomposing with 2-joins}

We first state some properties of 2-joins in graphs with no clique cutset.
Let $\mathcal D$ be the class of all graphs from $\mathcal C$ that do
not have a clique cutset. By Lemma~\ref{Star=Clique} no graph from
$\mathcal D$ has a star cutset and by Lemma \ref{diamondCliqueCut} no
graph from $\mathcal D$ contains a diamond.

An almost 2-join with a split $(X_1, X_2, A_1, A_2, B_1, B_2)$ in a
graph $G$ is \emph{consistent} if the following statements hold for
$i=1, 2$:

\begin{enumerate}
\item Every component of $G[X_i]$ meets both $A_i$, $B_i$.
\item Every node of $A_i$ has a non-neighbor in $B_i$.
\item Every node of $B_i$ has a non-neighbor in $A_i$.
\item Either both $A_1$, $A_2$ are cliques, or one of $A_1$ or $A_2$ is
  a single node, and the other one is a disjoint union of cliques.
\item Either both $B_1$, $B_2$ are cliques, or one of $B_1$, $B_2$ is
  a single node, and the other one is a disjoint union of cliques.
\item $G[X_i]$ is connected.
\item For every node $v$  in $X_i$, there exists a path in $G[X_i]$
  from $v$ to some node of $B_i$ with no internal node in $A_i$.
\item For every node $v$  in $X_i$, there exists a path in $G[X_i]$
  from $v$ to some node of $A_i$ with no internal node in $B_i$.
\end{enumerate}

Note that the definition contains redundant statements (for instance, (vi)
implies (i)), but it is convenient to list properties separately as above.

\begin{lemma}[Lemma 6.1 in \cite{twf-p1}]
  \label{l:consistent}
  If $G\in\mathcal D$, then
  every almost 2-join of $G$ is consistent.
\end{lemma}


We now define the blocks of decomposition of a graph with respect to a
2-join.  Let $G$ be a graph and $(X_1, X_2,A_1,A_2,B_1,B_2)$ a split of a 2-join of $G$.
Let $k_1$ and $k_2$ be positive integers. The
\emph{blocks of decomposition} of $G$ with respect to $(X_1, X_2)$ are
the two graphs $G_1^{k_1}$ and $G_2^{k_2}$ that we describe now.  We obtain $G_1^{k_1}$
from $G$ by replacing $X_2$ by a \emph{marker path} $P^2= a_2 \ldots
b_2$ of length $k_1$, where $a_2$ is a node complete to $A_1$, $b_2$ is a node complete
to $B_1$, and  $V(P_2)\setminus \{ a_2,b_2\}$ is anticomplete to $X_1$.  The
block $G_2^{k_2}$ is obtained similarly by replacing $X_1$ by a marker path
$P^1 = a_1\ldots b_1$ of length $k_2$.

In \cite{twf-p2} the blocks of decomposition w.r.t.\ a 2-join that we used in construction of a
recognition algorithm had marker paths of length 2. In this paper we will use blocks whose
marker paths are of length 5. So, unless otherwise stated, when we say that $G_1$ and
$G_2$ are blocks of decomposition w.r.t.\ a 2-join we will mean that their marker paths are of length~5. This will be discussed in more details in Section \ref{sec:paths}.


\begin{lemma}[Lemma 2.10 in \cite{twf-p3}]
  \label{new2}
  Let $G$ be a graph from ${\mathcal D}$. Let $(X_1, X_2)$ be a 2-join of $G$, and $G_1$,
  $G_2$ the blocks of decomposition with respect to this 2-join whose marker
  paths are of length at least~2. Then
  $G_1$ and $G_2$ are in ${\mathcal D}$ and they do not have star cutsets.
\end{lemma}

A 2-join $(X_1,X_2)$ of $G$ is a {\em minimally-sided 2-join} if for some $i\in \{ 1,2 \}$ the following holds: for every 2-join $(X_1',X_2')$ of $G$, neither $X_1'\subsetneq X_i$ nor $X_2'\subsetneq X_i$. In this case $X_i$ is a {\em minimal side} of this minimally-sided 2-join.

A 2-join $(X_1,X_2)$ of $G$ is an {\em extreme 2-join} if for some $i\in \{ 1,2 \}$ and all $k\geq 3$ the block of decomposition $G_i^k$ has no 2-join.
In this case $X_i$ is an {\em extreme side} of such a 2-join.

Graphs in general do not necessarily have extreme 2-joins (an example is given in \cite{nicolas.kristina:two}), but it is shown in  \cite{nicolas.kristina:two}
that graphs with no star cutset do. It is also shown in \cite{nicolas.kristina:two} that if $G$ has no star cutset then the blocks of decomposition w.r.t.\ a 2-join
whose marker paths are of length at least 3, also have no star cuset. This is then used to show that in a graph with no star cutset, a minimally-sided 2-join is extreme.
We summarize these results in the following lemma.

A {\it flat path} of $G$ is any path of $G$ of length at least 3, whose
interior vertices are of degree 2 (this definition is slightly different than in \cite{nicolas.kristina:two},
but this does not affect the proof of part (iii) of the following lemma).
The following statement can be extracted  from Lemmas 3.2, 4.2, 4.3 and 4.4 of  \cite{nicolas.kristina:two}.

\begin{lemma}[\cite{nicolas.kristina:two}]\label{extreme}
Let $G$ be a graph with no star cutset. Let
$(X_1,X_2,A_1,A_2,B_1,B_2)$ be a split of a minimally-sided 2-join of
$G$ with $X_1$ being a minimal side, and let $G_1$ and $G_2$ be the
corresponding blocks of decomposition whose marker paths are of length at least 3.
Then the following hold:
\begin{enumerate}
\item $|A_1|\geq 2$, $|B_1|\geq 2$, and in particular all the vertices of $A_2\cup B_2$ are of degree at least~3.
\item If $G_1$ and $G_2$ do not have star cutsets, then 
$(X_1,X_2)$ is an extreme 2-join, with $X_1$ being an extreme side (in particular, $G_1$ has no 2-join).
\item If $P$ is a flat path of $G$, such that $P\cap X_1\neq \emptyset$ and $P\cap X_2\neq\emptyset$, then one of the following holds:
    \begin{itemize}
      \item[(a)] for an endnode $u$ of $P$, $P\setminus u\subseteq X_1$ and $u\in A_2\cup B_2$;
      \item[(b)] for endnodes $u$ and $v$ of $P$, $u\in A_2$, $v\in
        B_2$, $P\setminus\{u,v\}\subseteq X_1$, the length of $P$ is
        at least 3 and $G[X_1]$ has exactly two connected components that are both a path with one end in $A_1$, one end in $B_1$ and interior in $C_1$.
    \end{itemize}
\end{enumerate}
\end{lemma}

\begin{remark}\label{remark1}
 We will be applying Lemma \ref{extreme} only to graphs $G\in \mathcal D$, and so by Lemmas \ref{Star=Clique} and
\ref{new2}, all of $G,G_1,G_2$ have no star cutsets.
Furthermore, for $G\in\mathcal D$, by Lemma \ref{l:consistent}, every 2-join is consistent, so outcome (b) of (iii) of
Lemma \ref{extreme}  is not possible (i.e.\ (a) is the only possible outcome of Lemma \ref{extreme}(iii)).
\end{remark}

%
%
%
%
In \cite{nicolas.kristina:two} it is shown that one can decompose a graph with no star cutset
using a sequence of `non-crossing' 2-joins into graphs with no star cutset and no 2-join
(which will in our case be basic).   In this paper we will use minimally-sided 2-joins (as opposed to \cite{nicolas.kristina:two} and \cite{twf-p3}, where minimally-sided 2-joins were 'moved' to allow marker paths to be disjoint). This will be particularly important when solving the induced paths problem. We now describe this 2-join decomposition.
%

\subsubsection*{2-Join decomposition tree $T_G$ of depth $p\geq 1$ of a graph $G$ that has no star cutset and has a 2-join}
\begin{itemize}
\item[(i)] The root of $T_G$ is $G^0=G$.
\item[(ii)] The non-leaf nodes of $T_G$ are $G^0,G^1,\ldots,G^{p-1}$. Each non-leaf node $G^i$ has two children: one is $G^{i+1}$ and the other one is $G^{i+1}_{B}$.

    The leaf-nodes of $T_G$ are graphs $G^1_B,G^2_B,\ldots,G^p_{B}$ and $G^p$.
    Graphs $G^1_B,G^2_B,\ldots,G^p_B,G^p$ have no star cutset nor 2-join.

  \item[(iii)] For $i\in\{0,1,\ldots,p-1\}$, $G^i$ has a 2-join $(X_1^i,X_2^i)$ that is minimally-sided with minimal side $X_1^i$.
  We denote the split of this 2-join by $(X_1^i,X_2^i,A_1^i,A_2^i,B_1^i,B_2^i)$.
  Graphs $G^{i+1}$ and $G_B^{i+1}$ are blocks of decomposition of $G^i$ w.r.t.\ $(X_1^i,X_2^i)$ whose marker paths are of length 5.
  The block $G_B^{i+1}$ corresponds to the minimal side $X_1^i$, i.e.\ $X_1^i\subseteq V(G_B^{i+1})$. We denote with
  $P^{i+1}$ the marker path used to build $G^{i+1}$.
\end{itemize}


\begin{lemma}\label{DT-construction}
There is an algorithm with the following specification.
\begin{description}
\item[ Input:] A graph $G\in\mathcal D$ that has a 2-join.
\item[ Output:]
A 2-join decomposition tree $T_G$ of depth at most $n$, such that all graphs that correspond to the nodes of
$T_G$ belong to $\mathcal D$, and all graphs that correspond to the leaves of $T_G$ (i.e. $G^1_B, \ldots ,G^p_B,G^p$)
belong to $\mathcal B$.
\item[ Running time:] $\mathcal O(n^4m)$.
\end{description}
\end{lemma}

\begin{proof}
Let $G^0=G$. Suppose that a decomposition tree of depth $i\geq 0$ of $G$ has been constructed. By Lemmas \ref{Star=Clique} and \ref{new2} graph $G^i$ belongs to $\mathcal D$ and has no star cutset. We apply on $G^i$ the algorithm from \cite{fast2j} that finds  a minimally-sided 2-join (the running time of this algorithm is $\mathcal O(n^3m)$). If no 2-join is found, then $i=p$, $G^i$ is basic (by Theorem \ref{decomposeTW}) and we stop the algorithm. If a 2-join $(X_1^i,X_2^i)$ is found, then we build graphs $G^{i+1}$ and $G_B^{i+1}$ as in the definition of a 2-join decomposition tree. Note that, by Lemmas \ref{Star=Clique}, \ref{new2} and \ref{extreme} and Theorem \ref{decomposeTW}, $G_B^{i+1}\in\mathcal B$ and $G^{i+1}$ does not have a star cutset. Hence, $(X_1^i,X_2^i)$ is an extreme 2-join of $G^i$ (by Lemma \ref{extreme}).

Let $T_G$ be the 2-join decomposition tree that is obtained using the algorithm we described. Note that every 2-join used to construct $T_G$ is in fact extreme and by Lemma \ref{l:consistent} consistent. So, the conclusion of Lemma 8.1 from \cite{nicolas.kristina:two} holds, i.e., the depth of $T_G$ is at most $n$ (Lemma 8.1 from \cite{nicolas.kristina:two} is formulated for a different graph class and the marker paths used there have length 3 or 4, but in our case we can derive almost identical proof; the only step there that is specific to that situation can be  obtained in our case using the fact that the 2-joins that we use are consistent). Hence, the running time of our algorithm is $\mathcal O(n^4m)$.
\end{proof}

\section{Induced cycles}\label{sec:hole}

In this section we consider an instance $(G,\mathcal V)$ of the $k$-\textsc{in-a-Cycle} problem on $\mathcal C$, that is a graph $G\in\mathcal C$ and a set $\mathcal V=\{v_1,\ldots,v_k\}$ of $k$ distinct vertices of $G$. The problem is to decide whether there exists a chordless cycle that contains all vertices of $\mathcal V$.

\subsection*{2-\textsc{in-a-Cycle}}

The 2-\textsc{in-a-Cycle} problem is to decide whether a graph $G$ contains a chordless cycle through two specified vertices $u$ and $v$ of $G$. By Bienstock's construction, this problem is NP-complete for general graphs (see \cite{bienstock}). In this section we prove that for $G\in\mathcal C$ the only obstructions for the existence of such a cycle are  clique cutsets that separate $u$ and $v$. This leads to an algorithm of running time $\mathcal O(nm)$ for 2-\textsc{in-a-Cycle} for $\mathcal C$.

An additional property of P-graphs that we need in the proof of Theorem \ref{2inAhole} below is the following:
a P-graph is a graph
$G\in \mathcal B$ with skeleton $R$ and special clique $K$ such that $K\neq \emptyset$ and each vertex of $L(R)$
that corresponds to an edge of $R$ incident with a degree 1 vertex (in $R$), has a neighbor in $K$.
We also need the following result from \cite{twf-p2}.

\begin{lemma}[Lemma 3.4 in \cite{twf-p2}]\label{holeInP}
Let $R$ be a skeleton of a P-graph. If $e_1$ and $e_2$ are edges of $R$, then there exists a cycle of $R$ that goes through $e_1$ and $e_2$, or there exists a path in $R$ whose endnodes are of degree 1 (in $R$) and that goes through $e_1$ and $e_2$.
\end{lemma}

\begin{theorem}\label{2inAhole}
Let $G$ be a graph of $\mathcal C$, and $u$ and $v$ be non-adjacent vertices of $G$. Then there exists a hole of $G$ that contains both $u$ and $v$, or $G$ admits a clique cutset that separates $u$ and $v$.
\end{theorem}

\begin{proof}
Our proof is by induction on $|V(G)|$. By Theorem \ref{decomposeTW}, it is enough to examine the following cases.

\medskip
\noindent{\bf Case 1.} $G$ is a line graph of a triangle-free chordless graph.

\medskip
\noindent Let $R$ be the root graph of $G$, and let $e=x_1x_2$ and
$f=y_1y_2$ be the edges of $R$ that correspond to $u$ and $v$ (note
that $\{x_1,x_2\}\cap\{y_1,y_2\} = \emptyset$). By Menger's theorem, in $R$ there exist two vertex disjoint paths $P$ and $Q$ from $\{x_1,x_2\}$ to $\{y_1,y_2\}$, or there is a vertex $z$ that separates $\{x_1,x_2\}$ from $\{y_1,y_2\}$. In the first case $P\cup Q$ is a hole of $R$ (since $R$ is triangle-free and chordless), and hence $L(P\cup Q)$ is a hole of $G$ that contains $u$ and $v$. In the second case the set $\{zz'\,:\,z'\in N_R(z)\}\setminus\{x_1x_2,y_1y_2\}$ correspond to a clique cutset of $G$ that separates $u$ and $v$.

\medskip
\noindent{\bf Case 2.} $G$ is a P-graph.

\medskip
\noindent Let $R$ be the skeleton and $K$ the special clique of $G$. First, let us consider the case when both $u$ and $v$ are in $L(R)$.  We apply Lemma \ref{holeInP} to edges $e_1$ and $e_2$ that in $R$ correspond to $u$ and $v$. If a hole $H$ is obtained, then $L(H)$ is a hole of $G$ that contains $u$ and $v$. If a path $P$ is obtained, then $L(P)$ together with neighbor(s) in $K$ of its endnodes induces a hole in $G$ that contains $u$ and $v$.

So, we may assume that $u\in K$ and let $u'\in L(R)$ be a neighbor of
$u$. Then $v\in L(R)$. Since $u'$ is of degree 1 in $L(R)$, by Lemma
\ref{holeInP}, there exists a chordless path $Q$ in $L(R)$ that
contains both $u'$ and $v$, and such that its other endnode $w'$
($w'\neq u'$) is of degree 1 in $L(R)$. Let $w$ be the neighbor of
$w'$ in $K$. Then $V(Q)\cup\{u,w\}$ induces a desired hole in $G$.

\medskip
\noindent{\bf Case 3.} $G$ admits a clique cutset.

\medskip
\noindent Let $(A,K,B)$ be a split of this cutset. If $u$ and $v$ are separated by this cutset, we are done, so we may assume that $u,v\in V(G_A)$. Then, by induction, there exists a hole $H$ in $G_A$ that contains both $u$ and $v$, or a clique cutset $K'$ of $G_A$ that separates $u$ and $v$. In the first case $H$ is a hole of $G$ that contains both $u$ and $v$, and in the second $K'$ is a clique cutset of $G$ that separates $u$ and $v$.

\medskip
\noindent{\bf Case 4.} $G$ admits a 2-join.

\medskip
\noindent Let $(X_1,X_2,A_1,A_2,B_1,B_2)$ be a 2-join of $G$. By Case 3, we may assume that $G$ does not admit a clique cutset, and hence, by Lemma \ref{l:consistent}, that $(X_1,X_2)$ is a consistent 2-join. So, by property (vi) of consistent 2-joins, for $i\in\{1,2\}$ there is a chordless path $Q^i$ in $G[X_i]$ that has exactly one vertex from both  $A_i$ and $B_i$. Also, by Lemma \ref{new2}, blocks of decomposition $G_1$ and $G_2$ belong to $\mathcal D$.

First, let us assume that $u,v\in X_i$, for some $i\in\{1,2\}$. Then, by induction, there is a hole $H$ in $G_i$ that contains $u$ and $v$. If $H$ is contained in $G[X_i]$, then we are done. Otherwise $H$ contains the marker path $P^{3-i}$, and hence to obtain a desired hole it is enough to replace $P^{3-i}$ with $Q^{3-i}$ in $H$.

So, we may assume that $u\in X_1$ and $v\in X_2$. By induction  there exists a hole $H_1$ (resp.\ $H_2$) in $G_1$ (resp.\ $G_2$) that contains $u$ and $c_2$ (resp.\ $v$ and $c_1$), where $c_i$ is an internal vertex of $P^i$, for $i\in\{1,2\}$. Let $R^1$ (resp.\ $R^2$) be the path obtained from $H_1$ (resp.\ $H_2$) by removing the vertices of the marker path $P^2$ (resp.\ $P^1$). Then $V(R^1)\cup V(R^2)$ induces a hole in $G$ that contains both $u$ and $v$.
\end{proof}

\begin{theorem}
There is an algorithm with the following specifications:
\begin{description}
\item[ Input:] A graph $G\in\mathcal C$ and two vertices $u$ and $v$ of $G$.
\item[ Output:]
YES if there is a chordless cycle of $G$ that contains both $u$ and $v$, and NO otherwise.
\item[ Running time:] $\mathcal O(nm)$.
\end{description}
\end{theorem}

\begin{proof}
First, let us consider the case when $uv$ is an edge of $G$. Let $G'$ be the graph obtained from $G$ by deleting $uv$. Now, if there exists a path from $u$ to $v$ in $G'$, then return YES, and otherwise return NO.

So, we may  assume that $u$ and $v$ are not adjacent.  We build a clique cutset decomposition tree $T$ for $G$ and remember (all) leaf nodes of this tree that contain some of the vertices $u$ and $v$. By Theorem \ref{th:tarjan}, this can be done in time $\mathcal O(nm)$. If there is a leaf node of $T$ that contains both $u$ and $v$, then there does not exist a clique cutset that separates $u$ and $v$, and hence, by Theorem \ref{2inAhole}, we return YES. Otherwise, we return NO. The running time of this algorithm is $\mathcal O(nm)$.
\end{proof}

\subsection*{$k$-\textsc{in-a-Cycle}}

In this section we prove that the problem $k$-\textsc{in-a-Cycle} is fixed-parameter tractable, when parameterized by $k$, for graphs in $\mathcal C$.

We use the following result of Robertson and Seymour.

\begin{theorem}[\cite{RS:GraphMinor13}]\label{Th:RS}
The $k$-\textsc{Disjoint paths} problem is fixed-parameter tractable, when parameterized by $k$. More precisely, there is a computable function $h$, that depends only on $k$, such that the $k$-\textsc{Disjoint paths} problem can be solved in time $h(t)n^3$.
\end{theorem}

\begin{lemma}\label{k-in-a-cycle-basic}
For any fixed integer $k$, there is an algorithm with the following specifications:
\begin{description}
\item[ Input:] A graph $G\in\mathcal B$ and a set $\mathcal V=\{v_1,\ldots,v_k\}$ of $k$ distinct vertices from $G$.
\item[ Output:]
YES if the problem $(G,\mathcal V)$ has a solution, and NO otherwise.
\item[ Running time:] $\mathcal O(n^5)$.
\end{description}
\end{lemma}

\begin{proof}
Clearly, we may assume that $G$ is connected. Recall that $G$ is
(wheel, diamond)-free. In time
$\mathcal O(n^2m)$ we can find the set $C$ of all centers of claws
in $G$. If $|C|\geq 2$, then let $K$ be the maximal (w.r.t.\
inclusion) clique that contains $C$ (this can be done in
$\mathcal O(n)$-time since $G$ is diamond-free). Otherwise, let
$K=C$.  So $G\setminus K$ is (wheel, diamond, claw)-free. By Lemma
2.4 in \cite{twf-p1} it follows that $G\setminus K$ is a line graph
of a triangle-free chordless graph. In $\mathcal O(n+m)$-time we can
compute graph $R$ such that $G\setminus K=L(R)$ (see \cite{lehot,
rous}). It follows that $K$ is a special clique and $R$ a skeleton
of $G$.

To solve the given problem it is enough to solve $k!$ problems $(G,\mathcal W)$, where $\mathcal W=\{(v_{\sigma(1)},v_{\sigma(2)}),(v_{\sigma(2)},v_{\sigma(3)}),\ldots,(v_{\sigma(k-1)},v_{\sigma(k)}),(v_{\sigma(k)},v_{\sigma(1)})\}$ and $\sigma$ is a permutation of $\{1,2,\ldots,k\}$ (note that these are not \textsc{Induced Disjoint Path} problems by our definition, but they are by the definition used in \cite{IDP-clawFPT}, where it is allowed that the paths in the solution share endnodes; see also Subsection \ref{sub:related}). Hence, it is enough to show that each of them can be solved in $\mathcal O(n^5)$ time.
\smallskip

\noindent{\bf Case 1.} $K=\emptyset$.
\smallskip

Then $G=L(R)$, so, as noted in \cite[Lemma 3.7]{IDP-clawFPT}, to solve each of the given problems it is enough to solve $2^{2k}$ \textsc{Disjoint Paths} problems on $R$. These problems can be solved using the algorithm from Theorem \ref{Th:RS}. Since edges of $R$ are vertices of $G$, graph $R$ has $\mathcal O(n)$ vertices (in the proof of
\cite[Lemma 3.7]{IDP-clawFPT} a worse bound $\mathcal O(n^2)$ is used, which leads to a worse running time in that lemma), so the total running time in this case is $\mathcal O(n^3+n^2m)=\mathcal O(n^4)$.
\smallskip

\noindent{\bf Case 2.} $K\neq \emptyset$.
\smallskip

Let $S$ be the set of all vertices of $L(R)$ that have a neighbor in $K$. For each pair $(u,v)$ of vertices from $S$ we build a graph $G_{u,v}$ in the following way: we start with $L(R)$, add to it the unique chordless $(u,v)$-path $P_{uv}$ (of length 2 or 3) whose interior vertices are from $K$, and remove all vertices from $L(R)\setminus\{u,v\}$ that have a neighbor in the interior of $P_{uv}$. It is easy to see that $(G,\mathcal W)$ has a solution if and only if $(L(R),\mathcal W)$ or $(G_{u,v},\mathcal W)$, for some $u,v\in S$, has a solution. Each of the graphs $L(R)$ and $G_{u,v}$, for $u,v\in S$, is (wheel, diamond, claw)-free, and hence, by Lemma
2.4 in \cite{twf-p1}, it is the line graph of a triangle-free chordless graph. So, each of the problems $(L(R),\mathcal W)$ and $(G_{u,v},\mathcal W)$, for $u,v\in S$, can be solved as in Case 1, and since $|S|=\mathcal O(n)$, this implies an $\mathcal O(n^2m+n^2\cdot n^3)=\mathcal O(n^5)$-time algorithm for solving the given problem.
\end{proof}

\begin{theorem}\label{k-in-a-cycle-main}
For any fixed integer $k$, there is an algorithm with the following specifications:
\begin{description}
\item[ Input:] A graph $G\in\mathcal C$ and a set $\mathcal V=\{v_1,\ldots,v_k\}$ of $k$ distinct vertices from $G$.
\item[ Output:]
YES if the problem $(G,\mathcal V)$ has a solution, and NO otherwise.
\item[ Running time:] $\mathcal O(n^6)$.
\end{description}
\end{theorem}

\begin{proof}
In \cite{twf-p2} an $\mathcal O(n^2m)$-time algorithm is given for recognizing whether a graph belongs to $\mathcal B$. If $G\in\mathcal B$, then the problem $(G,\mathcal V)$ can be solved in time $\mathcal O(n^5)$ using Lemma \ref{k-in-a-cycle-basic}. So we may assume that $G\in \mathcal C \setminus \mathcal B$.

Next, we consider the case when $G\in\mathcal D\setminus \mathcal B$. Using Lemma \ref{DT-construction} we build a 2-join decomposition tree $T_G$ in time $\mathcal O(n^4m)$ (throughout the proof we use the notation from the definition of $T_G$). Let $c_i$ (resp.\ $c_i^B$) be an internal vertex of the marker path of $G^i$ (resp.\ $G_B^i$), for $1\leq i\leq p$. By Theorem \ref{decomposeTW}, graphs $G_B^1,\ldots,G_B^{p},G^p$ are in $\mathcal B$.

Let $\mathcal V^0=\mathcal V$. We now describe the problems $(G^i,\mathcal V^i)$ and $(G_B^i,\mathcal V_B^i)$, $1\leq i\leq p$, that we solve during our algorithm to obtain the solution of the problem $(G,\mathcal V)$.

We first introduce some notation. Let $C^i$ be any chordless cycle of $G^i$, and $Q_B^i$ (resp.\ $Q^i$) be the part of $C^i$ contained in $X_1^i$ (resp.\ $X_2^i$). Note that one of the following holds: (1) $Q_B^i$ and $Q^i$ are subpaths of $C^i$ of length at least 1; or (2) $Q^i$ is not a path of length at least 1; or (3) $Q_B^i$ is not a path of length at least 1. Then we define $C_B^{i+1}$ (resp.\ $C^{i+1}$) as the chordless cycle obtained from $C^i$ in the following way: (1) by replacing $Q^{i}$ (resp.\ $Q_{B}^{i}$) with the marker path of $G_B^{i+1}$ (resp.\ $G^{i+1}$); (2) by replacing the vertices (if they exist) of $C^i\cap A_2^i$ or $C^i\cap B_2^i$ with endnodes $a_2^i$ and $b_2^i$ of the marker path of $G_B^{i+1}$ (resp.\ $C^{i+1}$ is empty); (3) $C_B^{i+1}$ is empty (resp.\ $C^{i+1}$ is equal to $C^i$). (Note that $A_2^i$ and $B_2^i$ are cliques, so in case (3) $C^i$ is contained in $X_2^i$.)

Similarly, for a (non-empty) chordless cycle $C_B^{i+1}$ (resp.\ $C^{i+1}$) of $G_B^{i+1}$ (resp.\ $G^{i+1}$) that contains the marker path of $G_B^{i+1}$ (resp.\ $G^{i+1}$) we define a chordless cycle $C^i$ of $G^i$ as follows. Let $Q^i$ (resp.\ $Q_B^i$) be any chordless path from a vertex of $A_1^i$ (resp.\ $A_2^i$) to a vertex of $B_1^i$ (resp.\ $B_2^i$) in $X_1^i$ (resp.\ $X_2^i$) (this path exists by Lemma \ref{l:consistent}). Then $C^i$ is obtained from $C_B^{i+1}$ (resp.\ $C^{i+1}$) by replacing the marker path of $G_B^{i+1}$ (resp.\ $G^{i+1}$) by $Q^i$ (resp.\ $Q_B^i$).
\smallskip

Now, we describe how our algorithm solves the problem $(G^i,\mathcal V^i)$. It is enough to consider the following 3 cases.
\smallskip

\noindent{\bf Case 1.} $\mathcal V^i\subseteq X_1^i$.
\smallskip

\noindent We note that $C^i$ is a solution of $(G^i,\mathcal V^i)$ if and only if $C_B^{i+1}$ is a solution of $(G_B^{i+1},\mathcal V^i)$. So, in this case it is enough to decide whether $(G_B^{i+1},\mathcal V^i)$ has a solution, which can be done using Lemma \ref{k-in-a-cycle-basic} in time $\mathcal O(n^5)$.
\smallskip

\noindent{\bf Case 2.} $\mathcal V^i\subseteq X_2^i$.
\smallskip

\noindent We note that $C^i$ is a solution of $(G^i,\mathcal V^i)$ if and only if $C^{i+1}$ is a solution of $(G^{i+1},\mathcal V^i)$. So, we proceed recursively, by solving the problem $(G^{i+1},\mathcal V^i)$.
\smallskip

\noindent{\bf Case 3.} $\mathcal V^i\cap X_1^i\neq\emptyset$ and $\mathcal V^i\cap X_2^i\neq\emptyset$.
\smallskip

\noindent Let $\mathcal V_B^{i+1}=(\mathcal V^i\cap X_1^i)\cup\{c_B^{i+1}\}$ and $\mathcal V^{i+1}=(\mathcal V^i\cap X_2^i)\cup\{c^{i+1}\}$. We note that $C^i$ is a solution of $(G^i,\mathcal V^i)$ if and only if $C^{i+1}$ is a solution of $(G^{i+1},\mathcal V^{i+1})$ and $C_B^{i+1}$ is a solution of $(G_B^{i+1},\mathcal V_B^{i+1})$ (indeed, in this case the cycle $C^{i+1}$ (resp.\ $C_B^{i+1}$) constructed from $C^i$ contains the marker path of $G^{i+1}$ (resp.\ $G_B^{i+1}$)). So, we solve the problem $(G_B^{i+1},\mathcal V^{i+1})$ using algorithm from Lemma \ref{k-in-a-cycle-basic}. If this algorithm returns NO, then we return NO and stop. Otherwise, we proceed recursively, by solving the problem $(G^{i+1},\mathcal V^{i+1})$ (note that $|\mathcal V^{i+1}|\leq |\mathcal V^i|$).
\smallskip

The running time of the described algorithm is $\mathcal O(n^4m+n\cdot n^5)=\mathcal O(n^6)$, since there are at most $p+1=\mathcal O(n)$ calls to the algorithm from Lemma \ref{k-in-a-cycle-basic}.

\smallskip

Finally, let us consider the general case, that is $G\in\mathcal C$. First, using Theorem \ref{th:tarjan} we build a clique cutset decomposition tree $T$ of $G$ in time $\mathcal O(nm)$ (in what follows we use the notation from the definition of $T$). Note that for each $1\leq i\leq p$ a chordless cycle can not contain vertices from both $A_i$ and $B_i$. So, for each $1\leq i\leq p$ we check if both $A_i\cap \mathcal V$ and $B_i\cap\mathcal V$ are non-empty, and if this is the case we return NO. This can be done in time $\mathcal O(n)$. Hence, we may assume that $\mathcal V$ is contained in $G_j^B$, for some $j\in\{1,2,\ldots,p+1\}$. Then a desired cycle, if it exists, is also contained in $G_j^B$, so it is enough to solve the problem $(G_j^B,\mathcal V)$. Since $G_j^B\in\mathcal D$, this can be done in time $\mathcal O(n^6)$ using the previous part of the proof. Hence the running time of the algorithm is $\mathcal O(n+n^6)=\mathcal O(n^6)$.
\end{proof}

\begin{corollary}\label{k-in-a-cycle-FPT}
For graphs in $\mathcal C$ the $k$-\textsc{in-a-Cycle} problem is fixed-parameter tractable,
when parameterized by $k$.
\end{corollary}

\begin{proof}
Let $(G,\mathcal V)$ be an instance of the \textsc{$k$-in-a-Cycle} problem.
By Theorem \ref{Th:RS}, the problem $k$-\textsc{Disjoint Paths} can be solved in time $h(k)n^3$, where $h$ is a computable function that depending only on $k$. Since Lemma \ref{k-in-a-cycle-basic} has at most $2^{2k}k!n^2$ calls to this algorithm, the problem $(G,\mathcal V)$ can be solved in time $2^{2k}k!h(k)n^5$ for graphs in $\mathcal B$. Now, since for each $i\in\{0,1,\ldots,p-1\}$ the algorithm from Theorem \ref{k-in-a-cycle-main} has at most one call to the algorithm from Lemma \ref{k-in-a-cycle-basic}, and $p\leq n$ (by Lemma \ref{DT-construction}), we conclude that the problem $(G,\mathcal V)$ can be solved in time $2^{2k}k!h(k)n^6$ for graphs in $\mathcal C$.
\end{proof}


\section{Induced disjoint paths}\label{sec:paths}

In this section, we consider an instance $(G,\mathcal W)$ of the
$k$-\textsc{Induced Disjoint Paths} problem, that is a graph $G$, a
set $\mathcal W=\{(s_1,t_1),(s_2,t_2),\ldots,(s_k,t_k)\}$ of pairs of
vertices of $G$ such that all $2k$ vertices are distinct and the only
edges between these vertices are of the form $s_it_i$, for some
$1\leqslant i\leqslant k$. We denote
$W = \{s_1,\ldots,s_k, t_1,\ldots,t_k\}$.  Vertices in
$W$ are the {\it terminals} of $\mathcal W$.

Recall that $k$ is a fixed integer (that is not part of the input).
We have to decide whether there exist $k$ vertex-disjoint paths
$P_i=s_i\ldots t_i$, $1\leqslant i\leqslant k$, such that there are no edges
between vertices of paths $P_i$ and $P_j$, for $i\neq j$.

Note that $(G,\mathcal W)$ has a solution (say $P_i=s_i\ldots t_i$, $1\leqslant i\leqslant k$) if and only if $(G,\mathcal W)$ has a
solution so that all the paths in this solution are chordless (it is enough to take
$P_i'=s_i\ldots t_i$, $1\leqslant i\leqslant k$, where $P_i'$ is a chordless path of $G$ contained in
$P_i$). So, when solving $(G,\mathcal W)$ we may add the condition
that the paths $P_i=s_i\ldots t_i$, $1\leqslant i\leqslant k$, are
chordless. Throughout the paper when we say that a set of paths
$\mathcal P$ is a solution of $(G,\mathcal W)$, we will assume that
all paths in $\mathcal P$ are chordless.

In Subsection \ref{sub:cliquecut}, we study how clique cutset can be
used. In Subsection \ref{sub:basic}, we study basic graphs. In
Subsection \ref{sub:2join}, we study how 2-joins can be used, and we
give the main algorithm. In Subsection~\ref{sub:related}, we study
related problems.

\subsection{Clique cutsets and induced paths}
\label{sub:cliquecut}

Let $\mathcal G$ be a class of graphs that is closed under taking
induced subgraphs, and let
$\mathcal G_{\textsc{basic}}\subseteq \mathcal G$. Suppose that
$\mathcal G$ and $\mathcal G_{\textsc{basic}}$ satisfy the following:

\begin{center} {\it If $G\in\mathcal G$, then
    $G\in\mathcal G_{\textsc{basic}}$ or $G$ has a clique cutset.}
\end{center}

Then we say that $\mathcal G$ is {\it
  $\mathcal G_{\textsc{basic}}$-decomposable using clique cutsets}.

Let $\mathcal G$ be a class that is
$\mathcal G_{\textsc{basic}}$-decomposable using clique cutsets. In
this section we show how an $\mathcal O(n^c)$-time algorithm for
$k$-\textsc{Induced Disjoint Paths} problem for graphs in
$\mathcal G_{\textsc{basic}}$, where $c$ is a constant (that does not
depend on $k$), can be turned into an $\mathcal O(n^{2k+c})$-time
algorithm for $k$-\textsc{Induced Disjoint Paths} problem for graphs
in $\mathcal G$.

Throughout the rest of the section, we consider an instance
$(G,\mathcal W)$ as described above for a graph $G\in\mathcal G$.
Graphs from $\mathcal G_{\textsc{basic}}$ will be refered to as
\emph{basic graphs}.

Let $T$ be a clique cutset decomposition tree of $G$ of depth $p$
(with all notations as in Section~\ref{sec:decTh}).  The next
lemma tells how a set obtained during the decomposition process
behaves at the root level (that is in $G$).

\begin{lemma}
  \label{l:neighborhood}
  Suppose that $G$ is not basic (so $p\geq 1$).  Let
  $i\in\{0,1, \dots, p-1\}$ and $X\in\{A_i \cup K_i,K_i\cup B_i,K_p\cup B_p\}$.  If $C$ is a connected component of $G\sm X$, then
  $N_G(C) = N_G(C)\cap X$ is a clique.
\end{lemma}

\begin{proof}
  We prove the result by induction on $p$.  If $p=1$, then $i=0$,
  $(A_0, K_0, B_0)$ is a split of a clique cutset of $G$ and either
  $X = A_0\cup K_0 = K_1 \cup B_1$ or $X=K_0 \cup B_0$. In both cases,
  $N_G(C) \cap X \subseteq K_0$, so the conclusion holds.

  Suppose $p > 1$.  If $i=0$, the conclusion holds as above.  So,
  suppose $i\geq 1$.  Observe that then
  $X\subseteq V(G_1) = A_0\cup K_0$.  Also, the tree obtained from $T$
  by deleting $G_0$ and $G_1^B$ is a decomposition tree for $G_1$ with
  depth $p-1$, and with exactly the same cutsets and splits as in $T$,
  so we may apply the induction hypothesis to it.  Therefore, we know
  that for every connected component $D$ of $G_1\sm X$, $N_{G_1}(D)$
  is a clique.

  Let $C$ be a connected component of $G\sm X$.  Recall that
  $(A_0, K_0, B_0)$ is a split of a clique cutset of $G$ and
  $X\subseteq A_0\cup K_0$.  If $C\subseteq A_0$, then $C$ is
  connected component of $G_1\sm X$, so $N_G(X) = N_{G_1}(C)$ is a
  clique by the induction hypothesis.  Hence, we may assume that $C$
  contains at least one vertex in $K_0\cup B_0$.  In fact, since
  $G[K_0\cup B_0]$ has no clique cutset (by the definition of a
  decomposition tree), it follows that $B_0$ is connected and every
  vertex of $K_0$ has a neighbor in $B_0$, and hence
  $(B_0\cup K_0)\sm X\subseteq C$.

  If $C=B_0$ (meaning in fact that $K_0\subseteq X)$, then
  $N_G(C)\subseteq K_0$.  So suppose $B_0 \subsetneq C$.  Hence,
  $C\cap K_0 \neq \emptyset$.  Note that $C\sm B_0$ is connected,
  because $K_0$ is a clique and every vertex in $C\sm B_0$ can be
  linked by a path to some vertex in $K_0$.  Hence, $C\sm B_0$ is a
  connected component of $G_1\sm X$, so $N_{G_1}(C\sm B_0)$ is a
  clique $K$, and $K_0 \cap X\subseteq K$.  Since $N_G(B_0) \subseteq
  K_0$, we have $N_G(C)\subseteq K$.
\end{proof}

\begin{lemma}
  For every $i=0, \dots, p$, $G[K_i \cup B_i]$ is basic.
\end{lemma}

\begin{proof}
  By the definition of decomposition trees, $G[K_i \cup B_i]$ has no
  clique cutset. From the definition of
  $\mathcal G_{\textsc{basic}}$-decomposable classes,
  $G[K_i \cup B_i]$ is therefore basic.
\end{proof}

Let us assume that a {\it non-negative integer weight} function $w:V\rightarrow \mathbb Z_{\geq 0}$ is assigned to vertices of $G$ (here $\mathbb Z_{\geq 0}$ is the set of non-negative integers). Then for $A\subseteq V$ we define $w(A)=\Sigma_{v\in A}w(v)$.

\begin{lemma}
  \label{l:twoCases}
  Suppose that $G$ is not basic (so $p\geq 1$) and that a non-negative
  integer weight $w(v)$ is given to each vertex $v\in V(G)$. Then
  one of the following holds:
  \begin{enumerate}
  \item\label{o:twotwo} For some $i\in\{0,1, \dots, p-1\}$, $K_i$ is a clique
    cutset of $G$ for which there exists a split $(A, K_i, B)$ such
    that $w(A) \geq 2$ and $w(B) \geq 2$.
  \item\label{o:Xbasic}  For some $i\in\{0,1,\dots, p\}$, every connected component $C$
    of $G\sm (K_i\cup B_i)$, satisfies $w(C) \leqslant 1$.
  \end{enumerate}
\end{lemma}

\begin{proof}
  Suppose first that $w(B_0)\geq 2$. Then, we may assume that $w(A_0)
  \leqslant 1$ for otherwise, \ref{o:twotwo} holds with $K_i = K_0$ and the
  split $(A_0, K_0, B_0)$.  Hence, \ref{o:Xbasic} holds for $i=0$.
  So, we may assume that $w(B_0) \leqslant 1$.

  It follows that every connected component $C$ of $G\sm (A_0\cup K_0)$
  satisfies $w(C) \leqslant 1$.  Hence, it is well defined to consider the
  maximal index $\ell\in \{0, \dots, p-1\}$ such that every connected component $C$
  of $G\sm (A_\ell\cup K_\ell)$ satisfies $w(C) \leqslant 1$. If $\ell=p-1$, then
  since $B_p=A_{p-1}$ and $K_p=K_{p-1}$, \ref{o:Xbasic} holds for
  $i=p$, so we may assume $\ell<p-1$.  Therefore,
  $(A_{\ell+1}, K_{\ell+1}, B_{\ell+1})$ is a split of a clique cutset of
  $G_{\ell+1}=G[A_\ell\cup K_\ell]$.  Since, by Lemma \ref{l:neighborhood}, for every connected component $C$ of $G\sm (A_\ell\cup K_\ell)$, $N_G(C)$ is a clique, there are three types of
  such components~$C$:

  \begin{itemize}
  \item {\it type A:} connected components that have neighbors in $A_{\ell+1}$ but no
    neighbor in $B_{\ell+1}$ (possibly in $K_{\ell+1}$);
 \item {\it type B:} connected components that have neighbors in $B_{\ell+1}$ but no
   neighbor in $A_{\ell+1}$ (possibly in $K_{\ell+1}$);
 \item {\it type K:} connected component whose neighborhood (possibly empty) is
   included in $K_{\ell+1}$.
 \end{itemize}

 We denote by $A$ (resp.\ $B$, $K$) the union of all connected
 components of type A (resp.\ B, K).

 The connected components of $G\sm (K_{\ell+1} \cup B_{\ell+1})$ are
 the connected components of type B or K (that all have weight at most
 1) and the connected components of $G[A\cup A_{\ell+1}]$.  Therefore,
 unless $K_{\ell+1} \cup B_{\ell+1}$ satisfies \ref{o:Xbasic}, we may
 assume that $w(A\cup A_{\ell+1})\geq 2$.

 The connected components of $G\sm (A_{\ell+1} \cup K_{\ell+1})$ are the
 connected components of type A or K (that all have weight at most 1)
 and the connected components of $G[B\cup B_{\ell+1}]$.  Therefore, by
 the maximality of $\ell$, we know that $w(B\cup B_{\ell+1})\geq 2$.

 Now, we observe that $(A_{\ell+1} \cup A \cup K, K_{\ell+1}, B_{\ell+1} \cup
   B)$ is the split of a clique cutset of $G$ that satisfies \ref{o:twotwo}.
\end{proof}

\begin{theorem}\label{k-IndPaths-Cliques}
  Let $\mathcal G$ be a class that is
  $\mathcal G_{\textsc{basic}}$-decomposable using clique
  cutsets. Furthermore, let us assume that there is an
  $\mathcal O(n^c)$-time algorithm for $k$-\textsc{Induced Disjoint
    Paths} problem for graphs in $\mathcal G_{\textsc{basic}}$, where
  $c\geq 1$ is a constant (that does not depend on $k$). Then there is
  an algorithm that solves in $\mathcal O(n^{2k+c})$ time the
  $k$-\textsc{Induced Disjoint Paths} problem for every instance
  $(G, \mathcal W)$ such that $G\in \mathcal G$.
\end{theorem}

\begin{proof}
  Our proof is by induction on $k$. If $k=1$, then we can solve our
  problem in time $\mathcal O(n+m)=\mathcal O(n^2)$. So, we may assume
  that $k\geq 2$.

  We build a clique cutset decomposition tree $T$ for $G$. Let us
  assume that it is of depth $p$. By Theorem \ref{th:tarjan}, this can
  be done in time $\mathcal O(nm)=\mathcal O(n^3)$. Clearly, we may
  assume that $p\geq 1$.  We give weight 1
  to every vertex in $W$, weight 0 to every vertex in $V(G)\sm W$, and
  then we apply Lemma~\ref{l:twoCases}. This leads to two cases.

  \medskip
  \noindent{\bf Case 1:} For some $\ell\in\{0,1, \dots, p-1\}$, $K_\ell$ is a
  clique cutset of $G$ for which there exists a split $(A, K_\ell, B)$
  such that $|A\cap W|\geq 2$ and $|B\cap W|\geq 2$.
  \medskip

  Note that this situation can be detected in time $\mathcal O(n^3)$ by computing
  connected components of $G\setminus K_\ell$, for each $\ell\in\{0,1, \dots, p-1\}$ (by Theorem \ref{th:tarjan}, $p=\mathcal O(n)$).

  We set $K=K_\ell$. Note that no two paths from a solution of $(G, \mathcal W)$
  can have a vertex in $K$, so if for some $i\neq j$ both $\{s_i,t_i\}$ and $\{s_j,t_j\}$ have a non-empty intersection with both $A\cup K$ and $B\cup K$, then we return NO
  and stop the algorithm. Hence, we may assume that at most one pair from $\mathcal W$ has non-empty intersection with both $A\cup K$ and $B\cup K$.

  Let $\mathcal W_A=\{(s_i,t_i)\,|\,s_i,t_i\in A\}$ and
  $\mathcal W_B=\{(s_i,t_i)\,|\, s_i,t_i\in B\}$.  By the previous remark and the assumption in this case, both $\mathcal W_A$ and $\mathcal W_B$ are non-empty and $|\mathcal W\sm (\mathcal W_A\cup\mathcal W_B)|\leqslant 1$.

  We observe that for every solution $\mathcal P$ of $(G, \mathcal W)$ at most two
  vertices of $K$ are contained in paths of $\mathcal P$, and if two vertices of $K$ are members of these paths, then they are in the same path of $\mathcal P$.

  We now build several pairs of instances of \textsc{Induced Disjoint Paths} problem that we use to solve $(G,\mathcal W)$.  The construction depends on $|\mathcal W\sm (\mathcal W_A\cup\mathcal W_B)|$ and the position of vertices from the pair from $\mathcal W\sm (\mathcal W_A\cup\mathcal W_B)$ (if $|\mathcal W\sm (\mathcal W_A\cup\mathcal W_B)|=1$). By symmetry, it is enough to examine the following cases.
  \smallskip

\noindent{\it Case 1.1:} If $|\mathcal W\sm (\mathcal W_A\cup\mathcal W_B)|=0$, then these instances are:
  \begin{itemize}
  \item $(G[A], \mathcal W_A)$ and $(G[B], \mathcal W_B)$;
  \item for every $x\in K$, $(G[A\cup\{x\}], \mathcal W_A)$ and $(G[B\sm N(x)], \mathcal W_B)$;
  \item for every $x\in K$, $(G[A\sm N(x)], \mathcal W_A)$ and $(G[B\cup\{x\}], \mathcal W_B)$;
  \item for every distinct $x, y\in K$, $(G[A\cup\{x, y\}]), \mathcal W_A)$ and
    $(G[B\sm (N(x) \cup N(y))], \mathcal W_B)$;
  \item for every distinct $x, y\in K$, $(G[A\sm (N(x) \cup N(y))], \mathcal W_A)$ and
    $(G[B\cup\{x, y\}], \mathcal W_B)$.
  \end{itemize}
  \smallskip

\noindent{\it Case 1.2:} If $|\mathcal W\sm (\mathcal W_A\cup\mathcal W_B)|=1$, then w.l.o.g.\
  $\{(s_1,t_1)\}=\mathcal W\sm (\mathcal W_A\cup\mathcal W_B)$.

  In case $s_1\in A$ and $t_1\in B$, we build the following pairs of instances:
  \begin{itemize}
  \item for every $x\in K$, $(G[A\cup\{x\}], \mathcal W_A\cup \{(s_1, x)\})$
    and $(G[B\cup\{x\}], \mathcal W_B \cup \{(x, t_1)\})$;
  \item for every distinct $x, y\in K$, $(G[(A\sm N(y))\cup\{x\}],\mathcal W_A\cup\{(s_1, x)\})$ and $(G[(B\sm N(x))\cup\{y\}], \mathcal W_B\cup \{(y, t_1)\})$.
  \end{itemize}

  In case $s_1 \in A$ and $t_1 \in K$, we build the following pairs of instances:
  \begin{itemize}
   \item $(G[A\cup\{t_1\}], \mathcal W_A\cup \{(s_1, t_1)\})$
    and $(G[B\sm N(t_1)], \mathcal W_B)$;
  \item for every $x\in K$, $(G[A\cup\{x\}], \mathcal W_A\cup \{(s_1,x)\})$
    and $(G[B\sm (N(x) \cup N(t_1))], \mathcal W_B)$.
  \end{itemize}

  In case $s_1,t_1\in K$, we build the following pair of instances:
  \begin{itemize}
  \item $(G[A\sm(N(s_1)\cup N(t_1))], \mathcal W_A)$ and $(G[B\sm(N(s_1)\cup N(t_1))], \mathcal W_B)$.
 \end{itemize}

 In each case it is straightforward to check that $(G, \mathcal W)$ has a
 solution if and only if for some pair of instances that we built  both
 of them have a solution. So, we run recursively our
 algorithm for all these instances.

  \medskip
  \noindent{\bf Case 2:} For some $\ell\in\{0,1,\dots, p\}$, every connected component $C$
  of $G\sm (K_\ell\cup B_\ell)$ satisfies $|C \cap W| \leqslant 1$.
  \medskip

  Note that this case can be detected in time $\mathcal O(n^3)$ by computing
  connected components of $G\sm (K_\ell\cup B_\ell)$, for $l\in\{0,1,\dots,p\}$
  (by Theorem \ref{th:tarjan}, $p=\mathcal O(n)$).

  We first delete all connected components $C$ of
  $G\sm (K_\ell\cup B_\ell)$ that satisfies $|C \cap W|=0$.  This
  is correct because a path in a solution for $(G,\mathcal W)$ cannot
  contain a vertex from such a connected component (by Lemma \ref{l:neighborhood}).

  Now, for each connected component $C$ of $G\sm (K_\ell\cup B_\ell)$ (that is left after previous deletions) there exists a unique
  $s_i$ or $t_i$ in $C$ and we set: $S_i=N(C)$ if $s_i\in C$, and $T_i=N(C)$ if
  $t_i\in C$. Let $\mathcal I_S\subseteq\{1,2,\ldots,k\}$ (resp.\ $\mathcal I_T\subseteq\{1,2,\ldots,k\}$) be the set of all $i$ for which
  $S_i$ (resp.\ $T_i$) is defined.

  Choose $s_i'\in S_i$, for $i\in\mathcal I_S$, and $t_i'\in T_i$, for $i\in \mathcal I_T$,
  and let $s_i'=s_i$ for $i\in\{1,2,\dots,k\}\sm\mathcal I_S$, and $t_i'=t_i$ for $i\in\{1,2,\dots,k\}\sm\mathcal I_T$  (note that $i\in\{1,2,\dots,k\}\sm\mathcal I_S$ implies $s_i\in K_\ell\cup B_\ell$). Furthermore, let $\mathcal W'=\{(s_i',t_i')\,|\,i\in\{1,2,\dots,k\}\}$. It is now straightforward to check
  that $(G, \mathcal W)$ has a solution if and only if $(G[K_\ell \cup B_\ell], \mathcal W')$ has a solution for at least one $\mathcal W'$ constructed in this way. Note that the
  graph $G[K_\ell \cup B_\ell]$ is basic.

\medskip
\noindent{\bf Complexity analysis}
\smallskip

Suppose that for some constant $q\geq 1$, the algorithm for the basic class
runs in time at most $qn^c$.  We claim that our algorithm then runs in
time at most $T(n,k) \leqslant q 3^k n^{2k+c} = \mathcal O(n^{2k+c})$
(because $k$ is not part of the instance).

If we are in Case 2, this is direct, since we call at most $n^{2k}$ times
the algorithm for the basic class.

If we are in Case 1, we first observe
that we run recursively the algorithm at most $n^2$ times on each
side. Let $k_A=|\mathcal W_A|$ and $k_B=|\mathcal W_B|$. Then $1\leqslant k_A, k_B \leqslant k-1$.

Now, in Case 1.1 the running time $T(n, k)$ satisfies:
\begin{align*}
  T(n, k)  & \leqslant n^2 T(n, k_A) + n^2 T(n, k_B) + n^3\\
           & \leqslant q n^2 3^{\max\{k_A,k_B\}}  [n^{2k_A +c} + n^{2k_B +c} + n]\\
           & \leqslant q n^2 3^{k-1}  [n^{2k +c - 2} + n^{2k +c - 2} + n]\\
           & \leqslant q  3^{k}  n^{2k +c}.
\end{align*}
In Case 1.2 we have $k_A, k_B \leqslant k-2$ and the running time $T(n, k)$ satisfies:
\begin{align*}
  T(n, k)  & \leqslant n^2 T(n,k_A+1) + n^2 T(n,k_B+1) + n^3\\
           & \leqslant q n^2 3^{\max\{k_A+1,k_B+1\}}  [n^{2k_A+2+c} + n^{2k_B+2+c} + n]\\
           & \leqslant q n^2 3^{k-1}  [n^{2k +c - 2} + n^{2k +c - 2} + n]\\
           & \leqslant q  3^{k}  n^{2k +c},
\end{align*}
which concludes our proof.
\end{proof}


\subsection{Basic graphs}
\label{sub:basic}

In this section we provide a polynomial-time algorithm that solves the
problem $(G,\mathcal W)$ for graphs $G\in\mathcal B$.

\begin{lemma}\label{k-IndPathsInBasic}
For a fixed integer $k$, there is an algorithm with the following specifications:
\begin{description}
\item[ Input:] A graph $G\in\mathcal B$ and a set of pairs of vertices
  $\mathcal W=\{(s_1,t_1),(s_2,t_2),\ldots,(s_k,t_k)\}$ from $G$, such that all $2k$ vertices are distinct and the only possible edges between these vertices are of the form $s_it_i$, for some $1\leqslant i\leqslant k$.
\item[ Output:]
YES if the problem $(G,\mathcal W)$ has a solution, and NO otherwise.
\item[ Running time:] $\mathcal O(n^5)$.
\end{description}
\end{lemma}

\begin{proof}
We use almost the same algorithm as in Lemma \ref{k-in-a-cycle-basic} -- the only difference is that we do not construct $k!$ \textsc{Induced Disjoint Path} problems, but work directly with the problem $(G,\mathcal W)$.
\end{proof}

With a more complicated algorithm the running time in the above lemma can be reduced to $\mathcal O(n^2m+n^3)$, but since this will not help improve the overall complexity of the algorithm in this section we do not include it here.

\subsection{2-joins and induced paths}
\label{sub:2join}

In this section we give a polynomial-time algorithm that solves the
problem $(G,\mathcal W)$ for graphs $G\in\mathcal D$. In fact, we
solve a similar problem on certain structures, called {\it o-graphs},
which allow us to use 2-joins in a more convenient way.

\begin{defeng}
  An {\it o-graph} $G_{\mathcal F,\mathcal O}$ is a triple
  $(G,\mathcal F,\mathcal O)$, where $G$ is a graph, $\mathcal F$ is a
  set of some flat paths of $G$ of length at most 7 and some vertices
  of $G$ (viewed as paths of length 0), and $\mathcal O$ is a set such
  that for each $\mathcal W\in \mathcal O$, $(G, \mathcal W)$ is an
  instance of the \textsc{Induced Disjoint Paths} problem where every
  terminal vertex is contained in a path $P\in\mathcal F$.
\end{defeng}

We say that an o-graph $G_{\mathcal F,\mathcal O}$ is {\it linkable}
if for at least one $\mathcal W\in \mathcal O$, the problem
$(G, \mathcal W)$ has a solution.

Note that an instance $(G, \mathcal W)$ of the \textsc{Induced
  Disjoint Paths} problem has a solution if and only if the o-graph
$(G, \mathcal F, \mathcal O)$ is linkable, where $\mathcal F$ is the
set of terminals of $\mathcal W$ and $\mathcal O= \{\mathcal W\}$.
So, to solve $(G, \mathcal W)$, it is enough to decide whether
$G_{\mathcal F,\mathcal O}$ is linkable.

We do this with 2-join decompositions that are performed to reduce the
problem to basic graphs. If $G\in \mathcal D$ is not basic, then it
has a 2-join and we consider a minimally sided 2-join $(X_1, X_2)$
where $X_1$ is a minimal side.  So, the block of decomposition $G^1$
(that contains $X_1$) is a basic graph, and our algorithm solves a
number of instances of the \textsc{Induced Disjoint Paths} problem in
it, to know how disjoint paths may exist through $X_1$.  We then
define an o-graph $G_{\mathcal F^2,\mathcal O^2}^2$, where $G^2$ is
the block of decomposition that contains~$X_2$.  The sets
$\mathcal F^2,\mathcal O^2$ are designed using the information gained
from the computations made in $G_1$, so that
$G_{\mathcal F,\mathcal O}$ is linkable if and only if
$G_{\mathcal F^2,\mathcal O^2}^2$ is linkable (see
Lemma~\ref{l:main-o-graphs}).  The marker path of $G^2$ (or an
extension of it) is sometimes added to $\mathcal F^2$ to record
different types of interactions of the solution paths with the 2-join.
Throughout all these steps of the algorithm, we can prove that
$|\mathcal F|$ is bounded by the initial number of terminals, and this
leads to an FPT algorithm.  In the following lemma we prove that
$|\mathcal O|$ is bounded by a function of $|\mathcal F|$.

\begin{lemma}\label{l:boundedNumber}
  Let $G_{\mathcal F,\mathcal O}$ be an o-graph, such that
  $|\mathcal F|\leq t$. Then $|\mathcal O|\leq 2^{8t}(8t)!$.
\end{lemma}

\begin{proof}
  Each path in $\mathcal F$ has at most 8 vertices. Hence
  $|\mathcal F|\leq t$ implies that each set in $\mathcal O$ has at
  most $4t$ pairs of vertices. So, there is at most $2^{8t}$ ways to
  choose the vertices that are going to be in the pairs of an element
  of $\mathcal O$.  Once we have chosen these vertices, say $2s$ of
  them ($2s\leq 8t$), to obtain an element of $\mathcal O$ we only
  need to group them in $s$ (disjoint) pairs. This can be done in
  $\frac{1}{s!}\binom{2s}{2}\binom{2s-2}{2}\cdots\binom{2}{2}=\frac{(2s)!}{2^ss!}$
  ways, and hence
\[|\mathcal O|\leq 2^{8t}\frac{(2s)!}{2^ss!}\leq 2^{8t}{(8t)!}.\]
\end{proof}

\begin{corollary}\label{l:BasicO-graph}
  Let $t$ be a fixed integer. If $G_{\mathcal F,\mathcal O}$ is an o-graph such
  that $G\in\mathcal B$ and  $|\mathcal F|\leq t$, then there is
  an $\mathcal O(n^5)$-time algorithm that decides whether $G_{\mathcal F,\mathcal O}$
  is linkable.
\end{corollary}

\begin{proof}
Follows from Lemmas \ref{k-IndPathsInBasic} and \ref{l:boundedNumber}.
\end{proof}

Let $G\in\mathcal D$ be a graph that has a 2-join $(X_1, X_2)$ and
$(G,\mathcal W)$ an instance of the $k$-\textsc{Induced Disjoint
  Paths} problem. For a split $(X_1,X_2,A_1,A_2,B_1,B_2)$ of
$(X_1, X_2)$ we define:
\begin{itemize}
\item $\mathcal I_1=\{1\leqslant i\leqslant k:\,s_i,t_i\in X_1\}$;
\item $\mathcal I_2=\{1\leqslant i\leqslant k\,:\,s_i,t_i\in X_2\}$;
  \item $\mathcal J=\{1,2,\ldots,k\}\setminus(\mathcal I_1\cup\mathcal I_2)$;
  \item $\mathcal W_1'=\{(s_i,t_i)\,:\,i\in \mathcal I_1\}$;
    \item $\mathcal W_2'=\{(s_i,t_i)\,:\,i\in \mathcal I_2\}$.
  \end{itemize}

Furthermore, we may assume that for $j\in \mathcal J$, $s_j\in X_1$ and $t_j\in X_2$.

Recall that we build a block $G^j$, for $j\in\{1,2\}$, by replacing
$X_{3-j}$ with a chordless path
$P^{3-j}=a_{3-j}a_{3-j}'c_{3-j}d_{3-j}b_{3-j}'b_{3-j}$ (called the
marker path) such that $a_{3-j}$ (resp.\ $b_{3-j}$) is complete to
$A_j$ (resp.\ $B_j$). See Fig.~\ref{f:blocks}.

\newcommand{\incj}[1]{\parbox[c]{5cm}{\includegraphics[height=4cm]{#1}}}
\begin{figure}
  \begin{center}
    \begin{tabular}{ccc}
      \incj{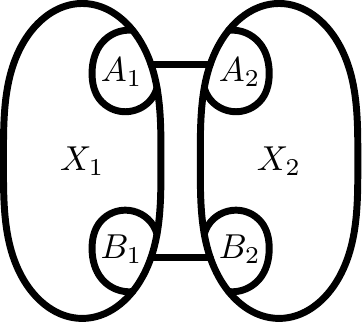}&\incj{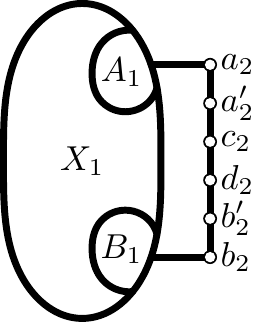}&\hspace{-1cm}\incj{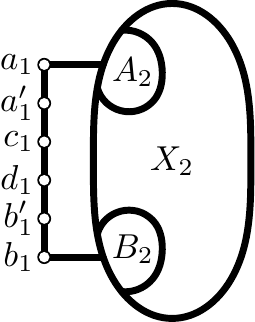}
    \end{tabular}
  \end{center}
  \caption{$G$ and its blocks of decomposition $G^1$ and $G^2$\label{f:blocks}}
\end{figure}

Next, we list a number of pairs of problems
($\mathcal S_1, \mathcal S_2)$. Problem $\mathcal S_1$ is going to be
solved on $G^1$ and problem $\mathcal S_2$ is going to be solved on
$G^2$. Each pair ($\mathcal S_1, \mathcal S_2)$ corresponds to a
possible way a solution of $(G, \mathcal W)$ interacts with the
2-join.  We aim to prove Lemma~\ref{k-IndPathsLemma} showing that
$(G,\mathcal W)$ has a solution if and only if for one of the pairs
($\mathcal S_1, \mathcal S_2)$, both $\mathcal S_1$ and $\mathcal S_2$
have a solution.  We therefore call these pairs {\it potential
  solutions} (for example, in (2.3) a potential solution is
described and we will refer to it as {\it potential solution (2.3)}).

Our definition of pairs $(\mathcal S_1,\mathcal S_2)$ depends on the
{\it type} of interaction of $\mathcal W$ with the 2-join
$(X_1,X_2)$. We list these types and the potential solutions in what
follows. (Note: it is best suited to examine the potential solutions
while reading the proof of Lemma \ref{k-IndPathsLemma}. Also, see
Fig.~\ref{f:t11}, \ref{f:t12}, \ref{f:t2} and~\ref{f:t3}.)

\newcommand{\vSp}{\parbox[c]{0cm}{\rule{0cm}{2.9cm}}}
\renewcommand{\incj}[1]{\parbox[c]{3.5cm}{\includegraphics[height=2.6cm]{#1}}}
\begin{figure}
    \begin{center}
\begin{tabular}{cccc}
(1.1)&\vSp\incj{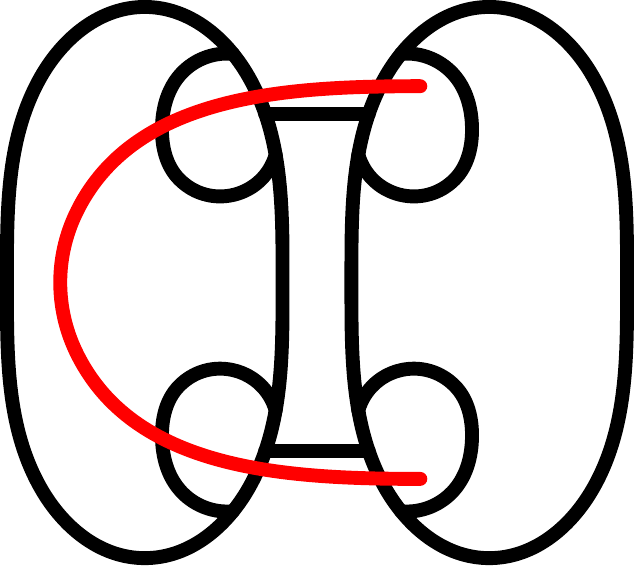}&\vSp\incj{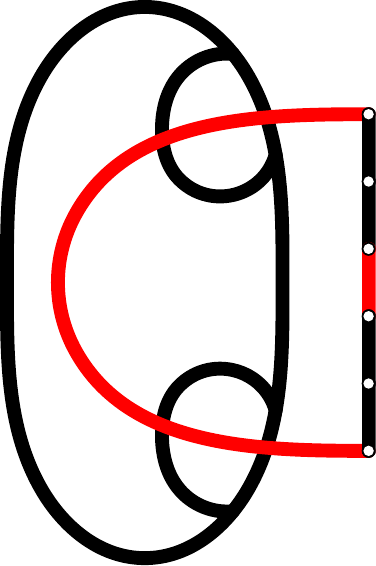}&\hspace{-1.5cm}\vSp\incj{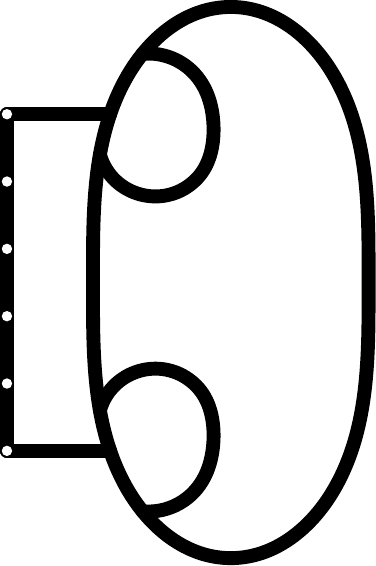}\\
(1.2)&\vSp\incj{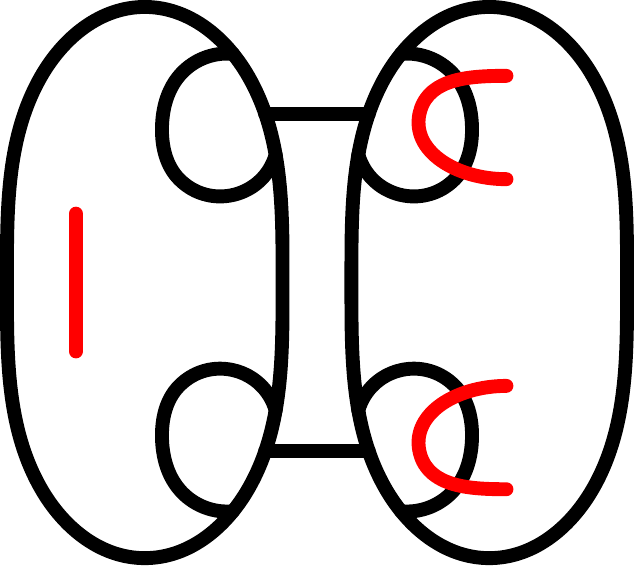}&\vSp\incj{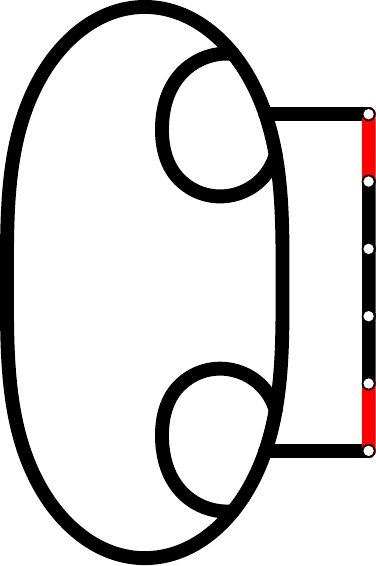}&\hspace{-1.5cm}\vSp\incj{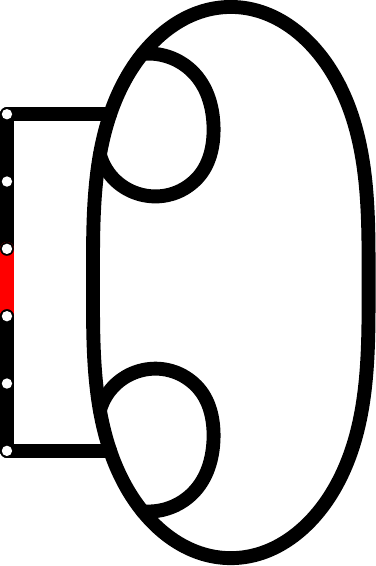}\\
(1.3)&\vSp\incj{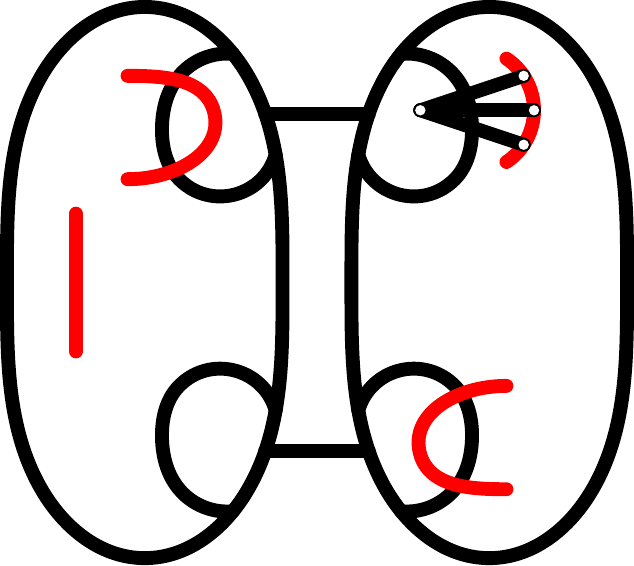}&\vSp\incj{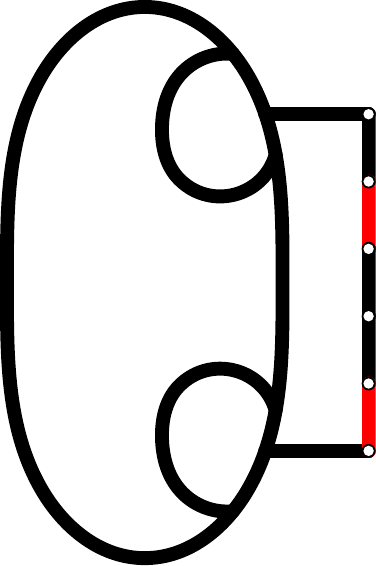}&\hspace{-1.5cm}\vSp\incj{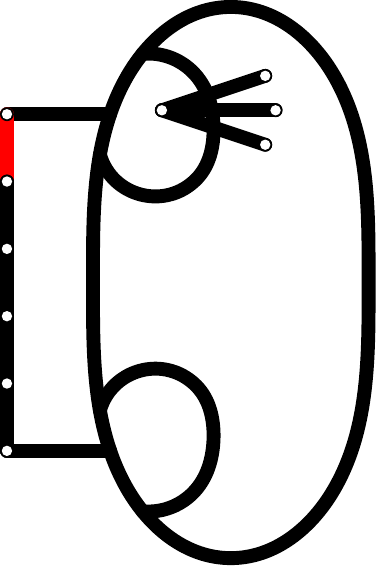}\\
(1.4)&\vSp\incj{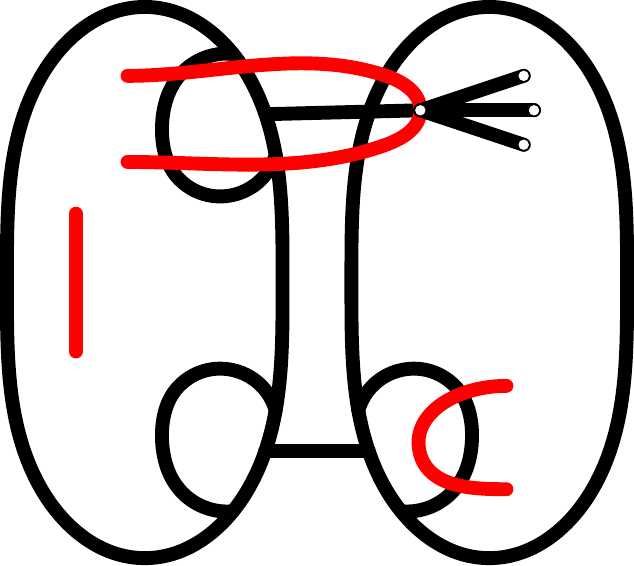}&\vSp\incj{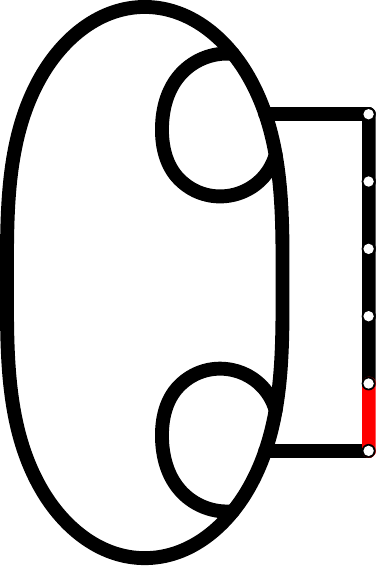}&\hspace{-1.5cm}\vSp\incj{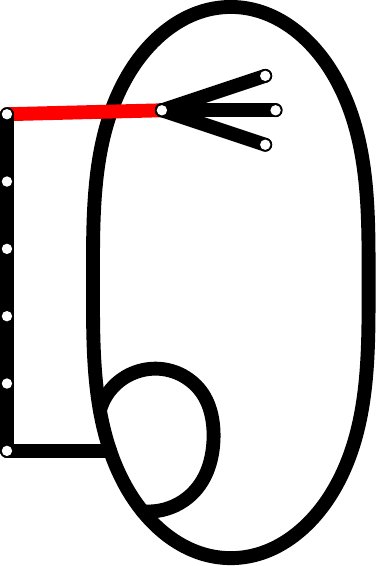}\\
(1.5)&\vSp\incj{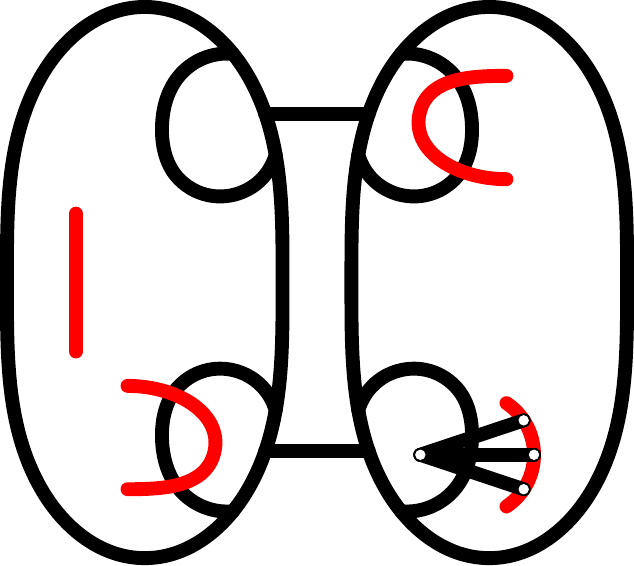}&\vSp\incj{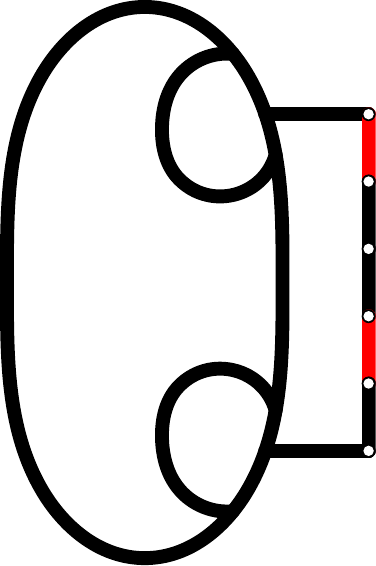}&\hspace{-1.5cm}\vSp\incj{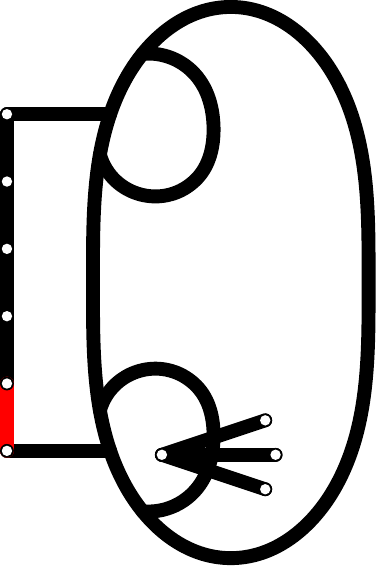}\\
(1.6)&\vSp\incj{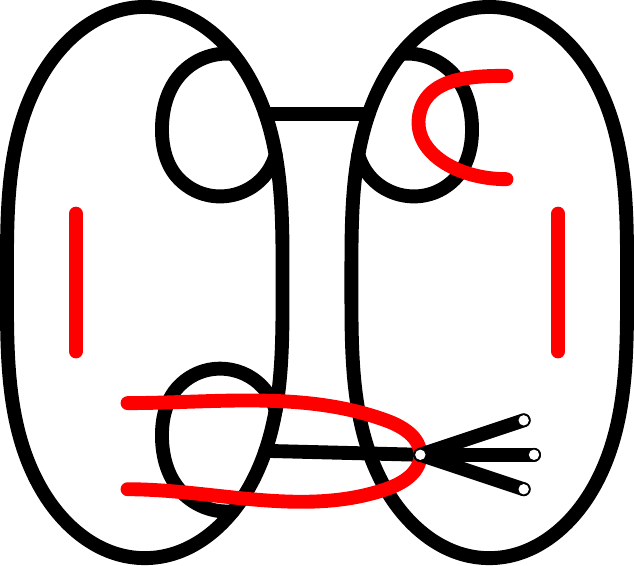}&\vSp\incj{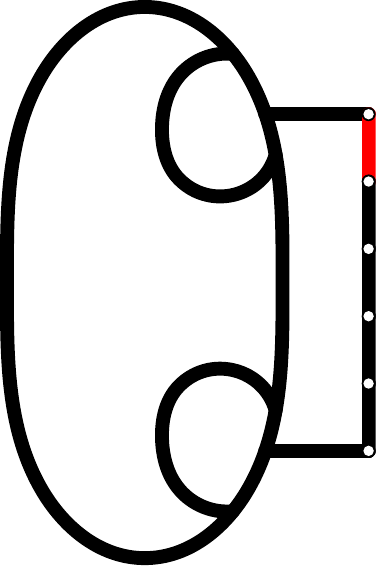}&\hspace{-1.5cm}\vSp\incj{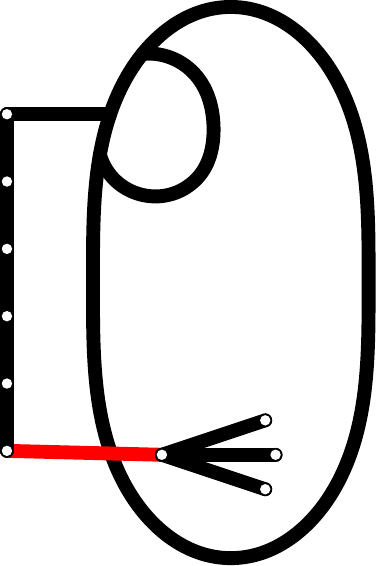}\\
\end{tabular}
  \caption{Type 1 potential solutions (1.1)--(1.6)\label{f:t11}}
  \end{center}
\end{figure}

\begin{figure}
    \begin{center}
\begin{tabular}{cccc}
(1.7)&\vSp\incj{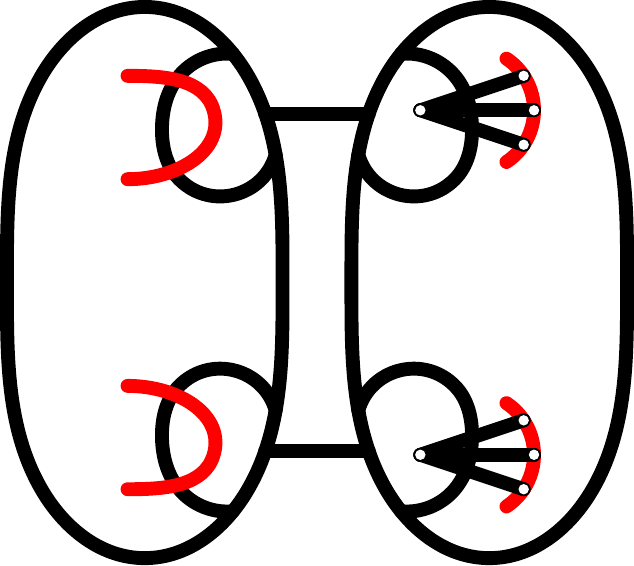}&\vSp\incj{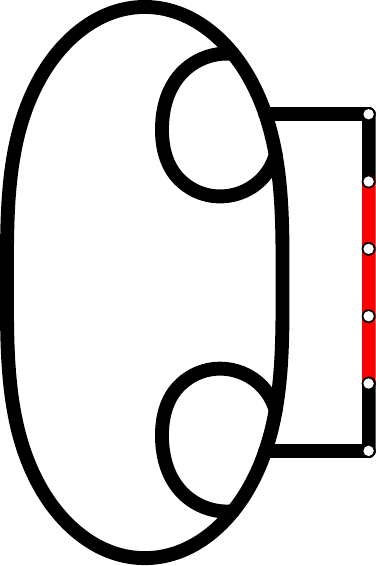}&\hspace{-1.5cm}\vSp\incj{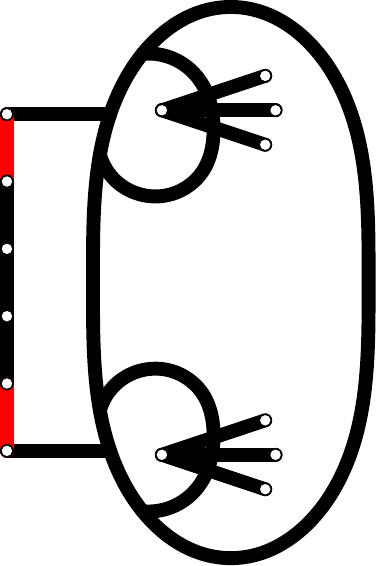}\\
(1.8)&\vSp\incj{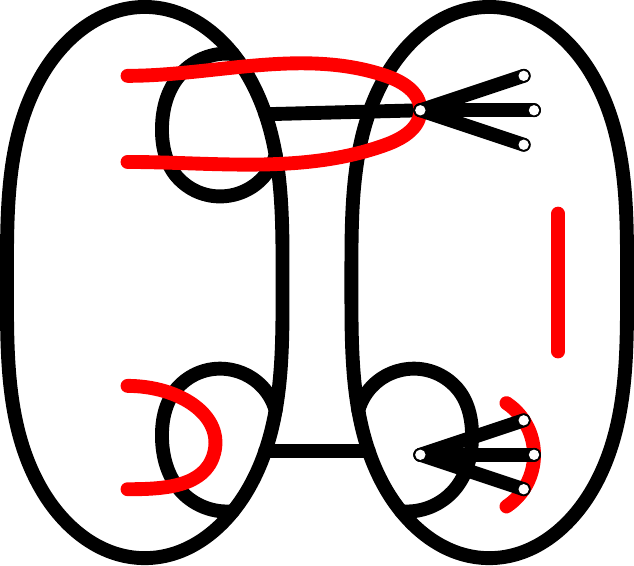}&\vSp\incj{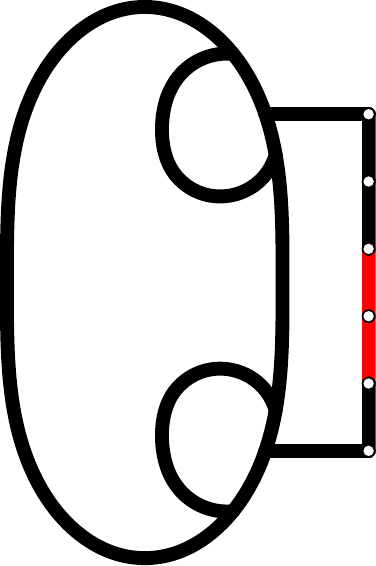}&\hspace{-1.5cm}\vSp\incj{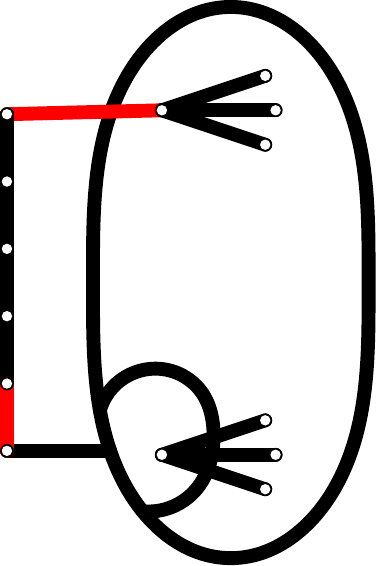}\\
(1.9)&\vSp\incj{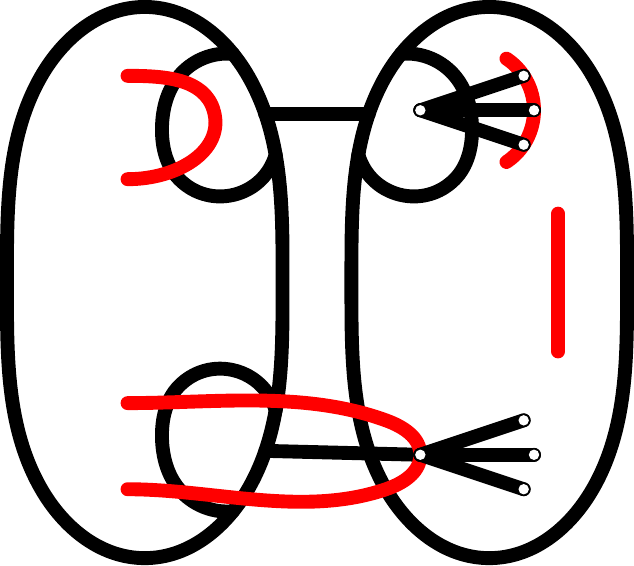}&\vSp\incj{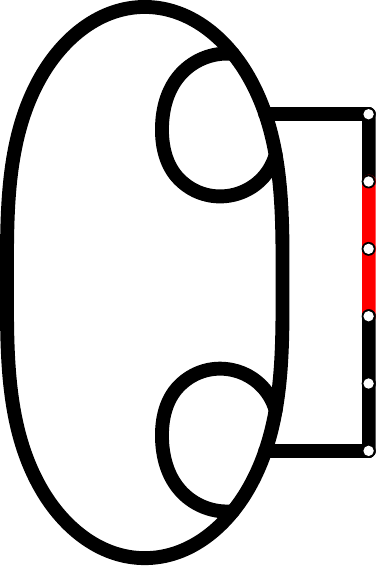}&\hspace{-1.5cm}\vSp\incj{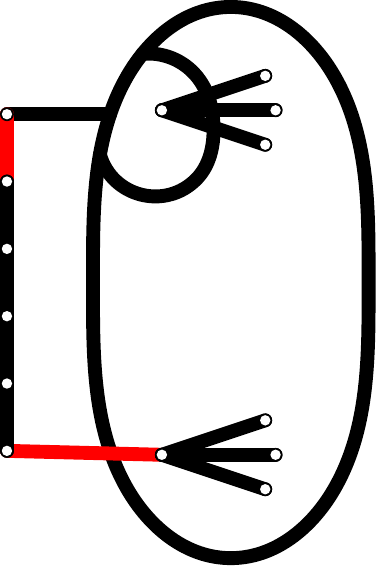}\\
(1.10)&\vSp\incj{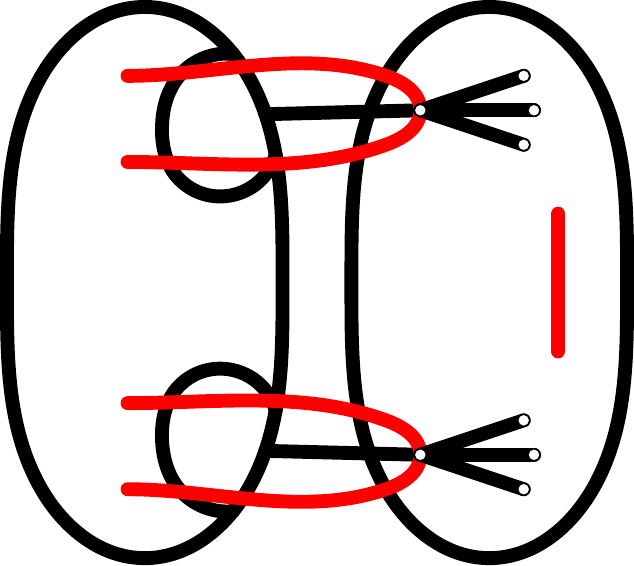}&\vSp\incj{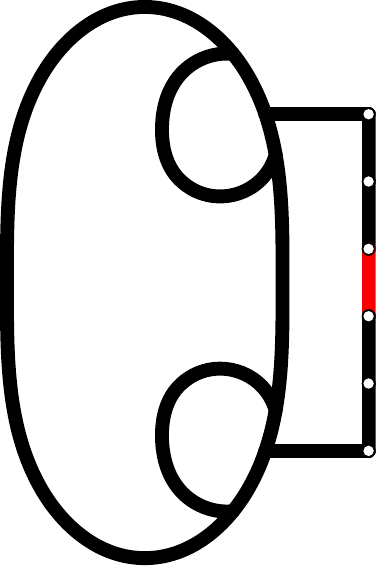}&\hspace{-1.5cm}\vSp\incj{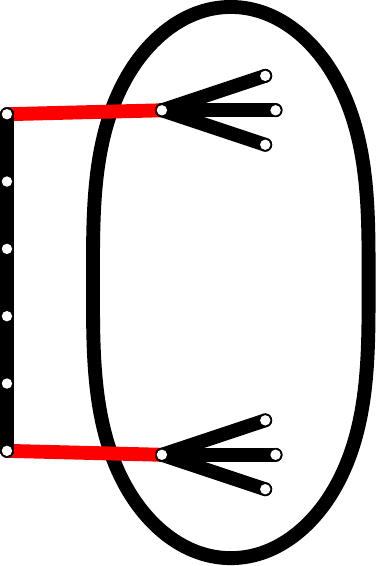}\\
(1.11)&\vSp\incj{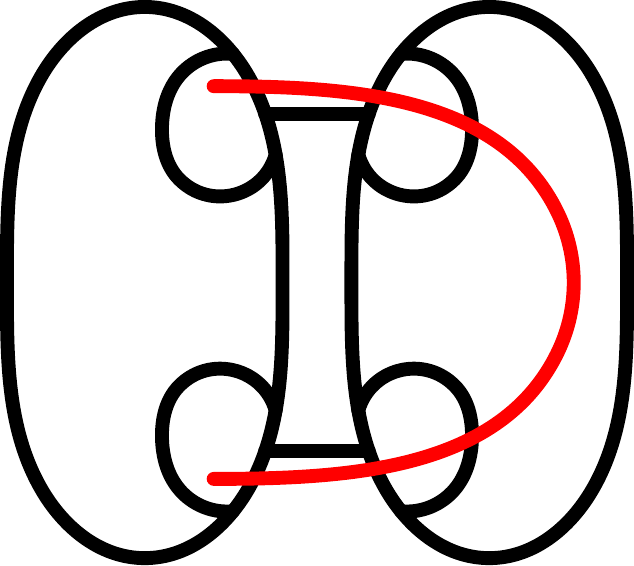}&\vSp\incj{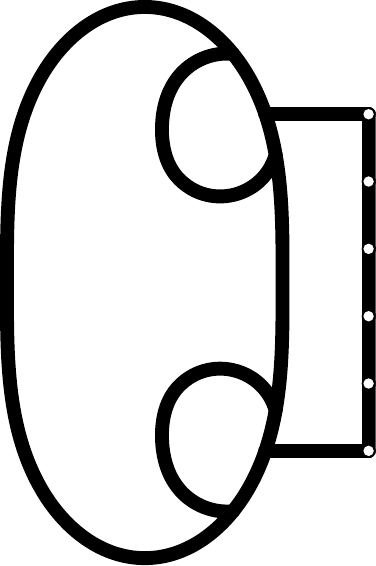}&\hspace{-1.5cm}\vSp\incj{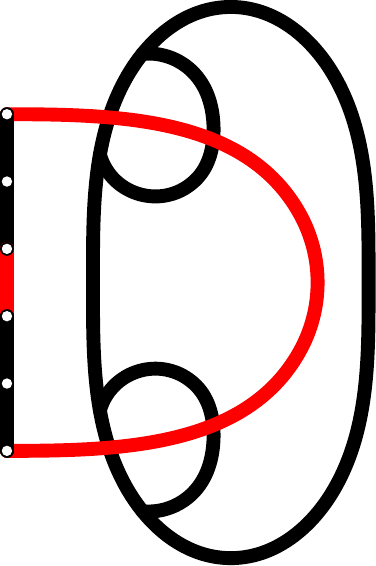}
\end{tabular}
  \caption{Type 1 potential solutions (1.7)--(1.11)\label{f:t12}}
  \end{center}
\end{figure}

\medskip

\noindent{\bf Type 1:} $\mathcal J=\emptyset$.
\begin{itemize}
  \item[(1.1)] $\mathcal S_1=(G^1,\mathcal W_1'\cup\{(a_2,b_2),(c_2,d_2)\})$,
  $\mathcal S_2=(G^2,\mathcal W_2')$;

  \item[(1.2)] $\mathcal S_1=(G^1,\mathcal W_1'\cup\{(a_2,a_2'),(b_2,b_2')\})$,
  $\mathcal S_2=(G^2,\mathcal W_2'\cup \{ (c_1,d_1)\})$;

  \item[(1.3)] $\mathcal S_1=(G^1,\mathcal W_1'\cup\{(b_2,b_2'),(a_2',c_2)\})$,
  $\mathcal S_2=(G^2,\mathcal W_2'\cup \{(a_1,a_1')\})$;

  \item[(1.4)] $A_2=\{a\}$, $\mathcal S_1=(G^1,\mathcal W_1'\cup\{(b_2,b_2')\})$,
  $\mathcal S_2=(G^2,\mathcal W_2'\cup \{(a,a_1)\})$;

  \item[(1.5)] $\mathcal S_1=(G^1,\mathcal W_1'\cup\{(a_2,a_2'),(b_2',d_2)\})$,
  $\mathcal S_2=(G^2,\mathcal W_2'\cup \{(b_1,b_1')\})$;

  \item[(1.6)] $B_2=\{b\}$, $\mathcal S_1=(G^1,\mathcal W_1'\cup\{(a_2,a_2')\})$,
  $\mathcal S_2=(G^2, \mathcal W_2'\cup \{(b,b_1)\})$;

  \item[(1.7)] $\mathcal S_1=(G^1,\mathcal W_1'\cup\{(a_2',b_2')\})$,
  $\mathcal S_2=(G^2,\mathcal W_2'\cup\{(a_1,a_1'),(b_1,b_1')\})$;

  \item[(1.8)] $A_2=\{a\}$, $\mathcal S_1=(G^1,\mathcal W_1'\cup\{(c_2,b_2')\})$,
  $\mathcal S_2=(G^2,\mathcal W_2'\cup \{(a,a_1),(b_1,b_1')\})$;

  \item[(1.9)] $B_2=\{b\}$, $\mathcal S_1=(G^1,\mathcal W_1'\cup\{(a_2',d_2)\})$,
  $\mathcal S_2=(G^2,\mathcal W_2'\cup \{(b,b_1),(a_1,a_1')\})$;

  \item[(1.10)] $A_2=\{a\}$, $B_2=\{b\}$, $\mathcal S_1=(G^1,\mathcal W_1'\cup\{(c_2,d_2)\})$,
  $\mathcal S_2=(G^2,\mathcal W_2'\cup \{(a,a_1),(b,b_1)\})$;

   \item[(1.11)] $\mathcal S_1=(G^1,\mathcal W_1')$,
   $\mathcal S_2=(G^2,\mathcal W_2'\cup\{(a_1,b_1),(c_1,d_1)\})$.
\end{itemize}

\medskip

\begin{figure}
  \begin{center}
\begin{tabular}{cccc}
(2.1)&\vSp\incj{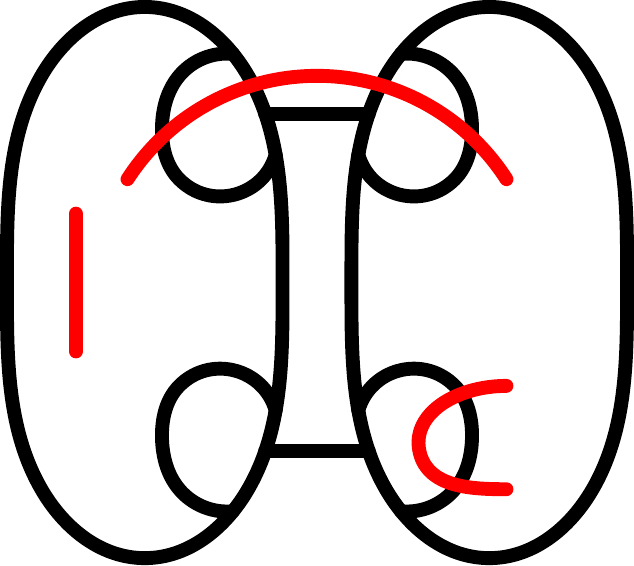}&\vSp\incj{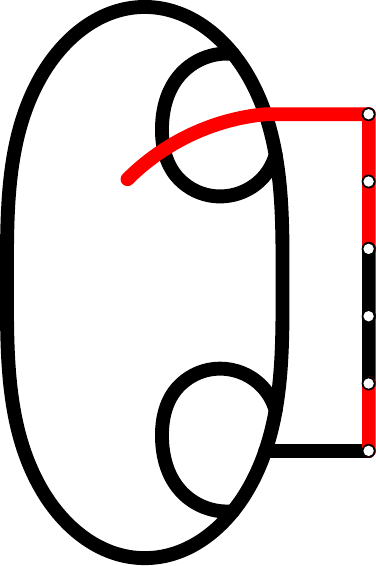}&\hspace{-1.5cm}\vSp\incj{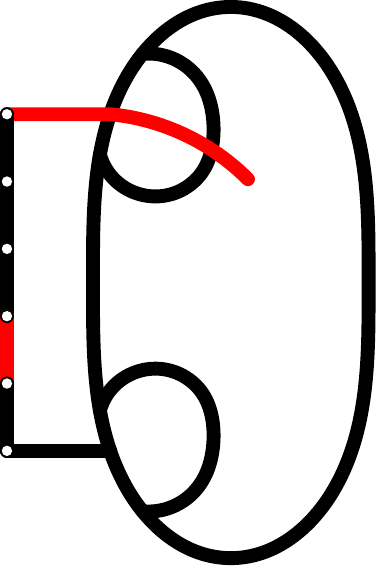}\\
(2.2)&\vSp\incj{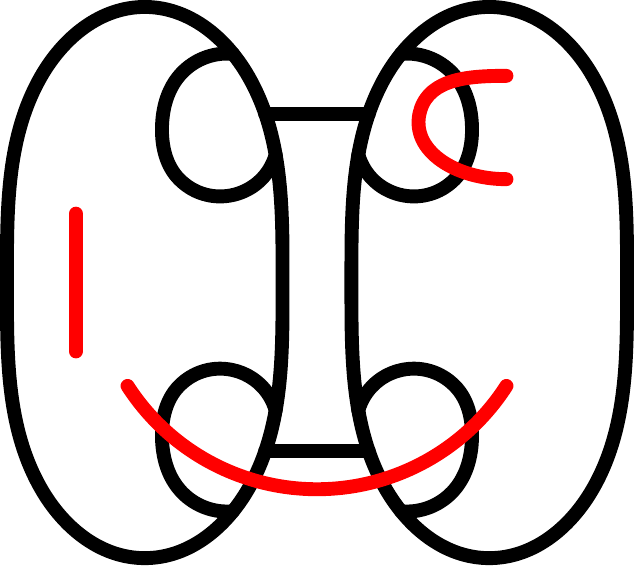}&\vSp\incj{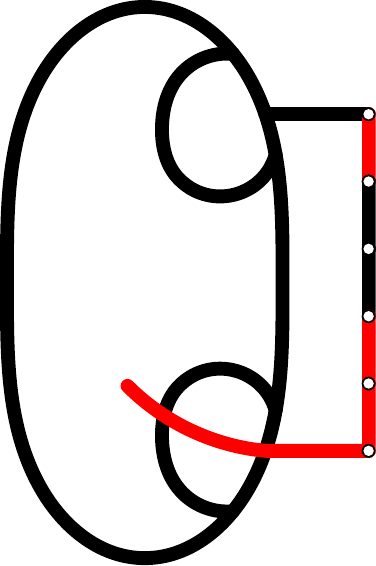}&\hspace{-1.5cm}\vSp\incj{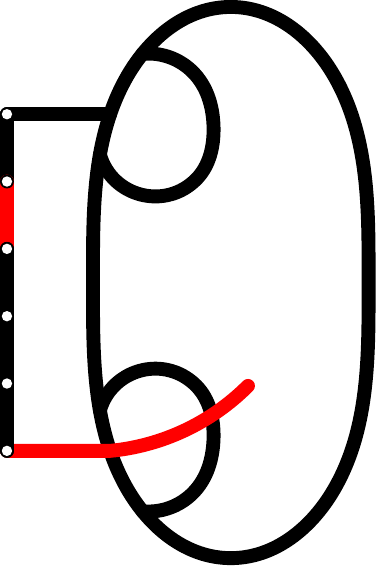}\\
(2.3)&\vSp\incj{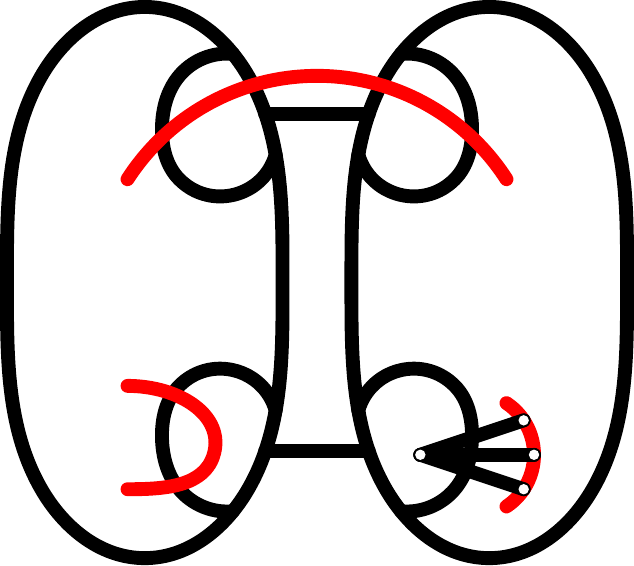}&\vSp\incj{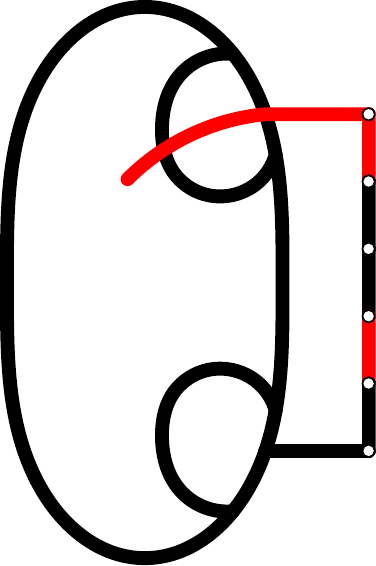}&\hspace{-1.5cm}\vSp\incj{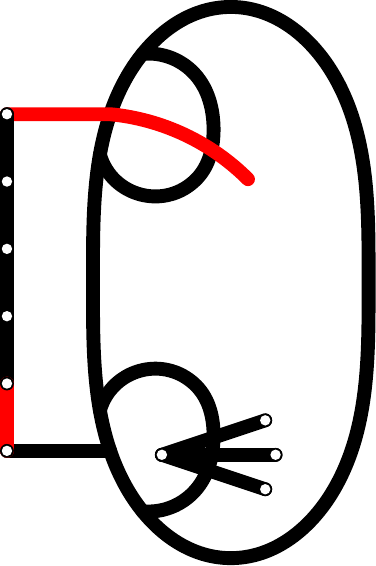}\\
(2.4)&\vSp\incj{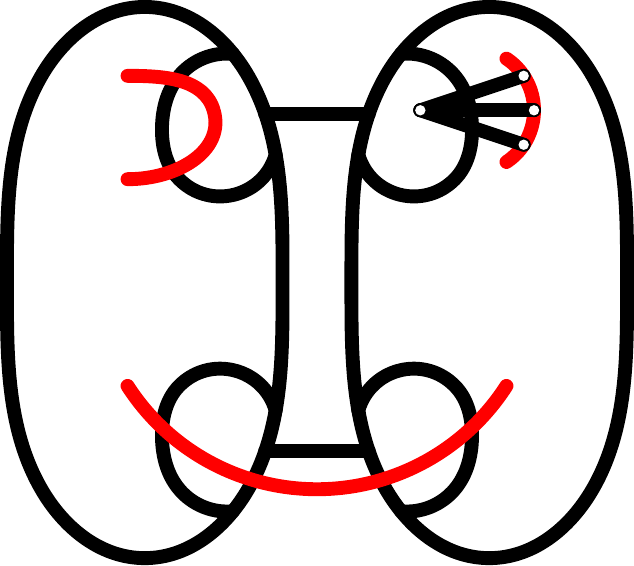}&\vSp\incj{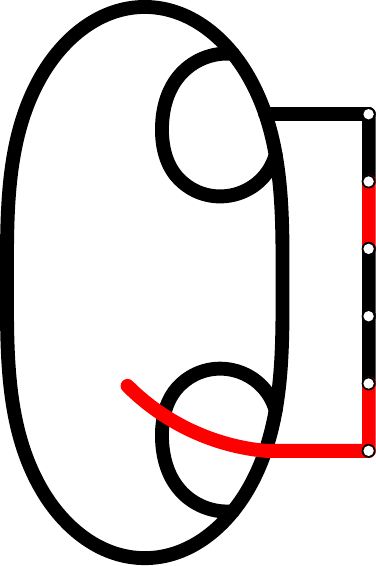}&\hspace{-1.5cm}\vSp\incj{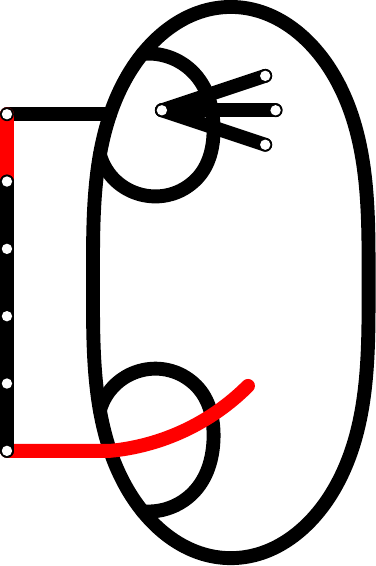}\\
(2.5)&\vSp\incj{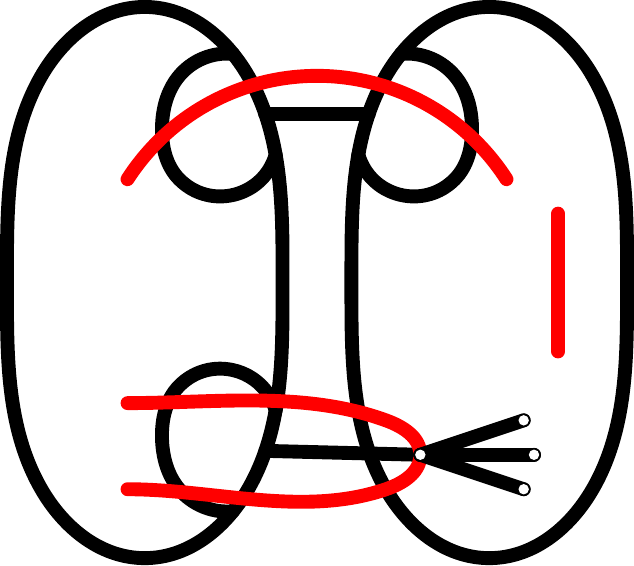}&\vSp\incj{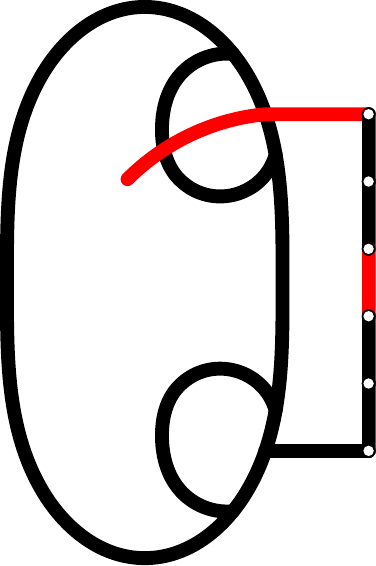}&\hspace{-1.5cm}\vSp\incj{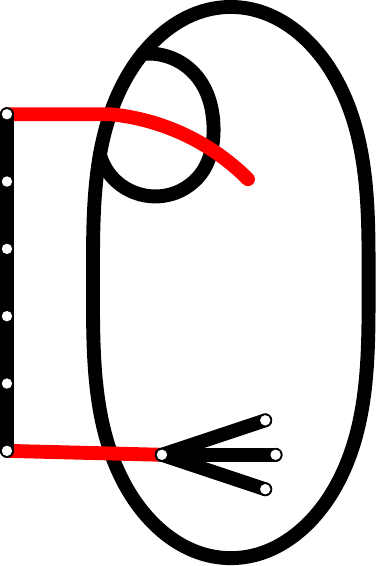}\\
(2.6)&\vSp\incj{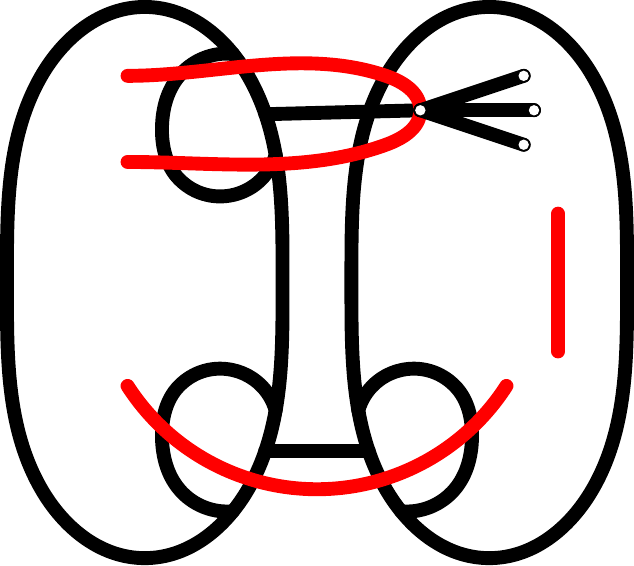}&\vSp\incj{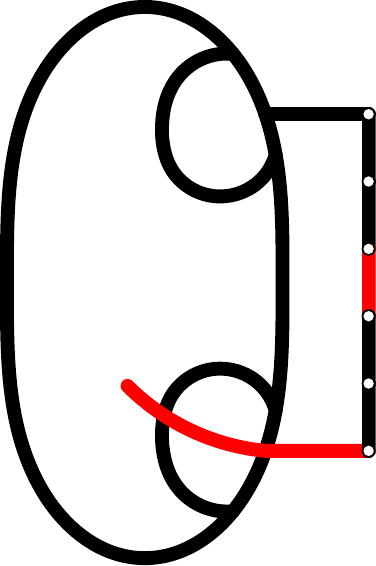}&\hspace{-1.5cm}\vSp\incj{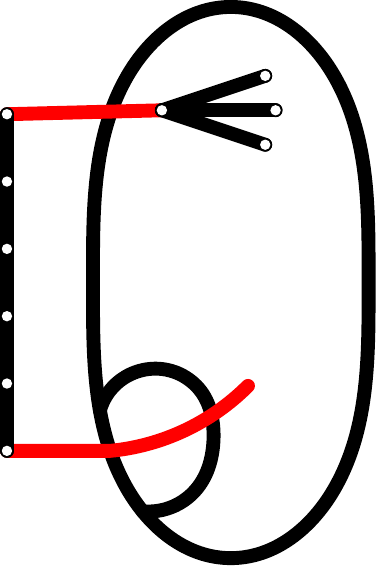}\\
\end{tabular}
  \caption{Type 2 potential solutions\label{f:t2}}
  \end{center}
\end{figure}

\noindent{\bf Type 2:} $\mathcal J=\{j_1\}$.
\begin{itemize}
  \item[(2.1)] $\mathcal S_1=(G^1,\mathcal W_1'\cup\{(s_{j_1},c_2),(b_2,b_2')\})$,
  $\mathcal S_2=(G^2,\mathcal W_2'\cup\{(a_1,t_{j_1}),(b_1',d_1)\})$;

  \item[(2.2)] $\mathcal S_1=(G^1,\mathcal W_1'\cup\{(s_{j_1},d_2),(a_2,a_2')\})$,
   $\mathcal S_2=(G^2,\mathcal W_2'\cup\{(b_1,t_{j_1}),(a_1',c_1)\})$;

  \item[(2.3)] $\mathcal S_1=(G^1,\mathcal W_1'\cup\{(s_{j_1},a_2'),(b_2',d_2)\})$,
  $\mathcal S_2=(G^2,\mathcal W_2'\cup \{(a_1,t_{j_1}),(b_1,b_1')\})$;

    \item[(2.4)] $\mathcal S_1=(G^1,\mathcal W_1'\cup\{(s_{j_1},b_2'),(a_2',c_2)\})$,
  $\mathcal S_2=(G^2,\mathcal W_2'\cup \{(b_1,t_{j_1}),(a_1,a_1')\})$;

  \item[(2.5)] $B_2=\{b\}$, $\mathcal S_1=(G^1,\mathcal W_1'\cup\{(s_{j_1},a_2),(c_2,d_2)\})$,
  $\mathcal S_2=(G^2,\mathcal W_2'\cup \{(a_1,t_{j_1}),(b,b_1)\})$;

   \item[(2.6)] $A_2=\{a\}$, $\mathcal S_1=(G^1,\mathcal W_1'\cup\{(s_{j_1},b_2),(c_2,d_2)\})$,
   $\mathcal S_2=(G^2,\mathcal W_2'\cup \{(b_1,t_{j_1}),(a,a_1)\})$.
   \end{itemize}

\medskip

\noindent{\bf Type 3:} $\mathcal J=\{j_1,j_2\}$.
\begin{itemize}
  \item[(3.1)] $\mathcal S_1=(G^1,\mathcal W_1'\cup \{(s_{j_1},a_2),(s_{j_2},b_2)\})$,
  $\mathcal S_2=(G^2,\mathcal W_2'\cup\{(a_1,t_{j_1}),(b_1,t_{j_2})\})$;

  \item[(3.2)] $\mathcal S_1=(G^1,\mathcal W_1'\cup \{(s_{j_1},b_2),(s_{j_2},a_2)\})$,
  $\mathcal S_2=(G^2,\mathcal W_2'\cup\{(b_1,t_{j_1}),(a_1,t_{j_2})\})$.
\end{itemize}

\begin{figure}
  \begin{center}
\begin{tabular}{cccc}
(3.1)&\vSp\incj{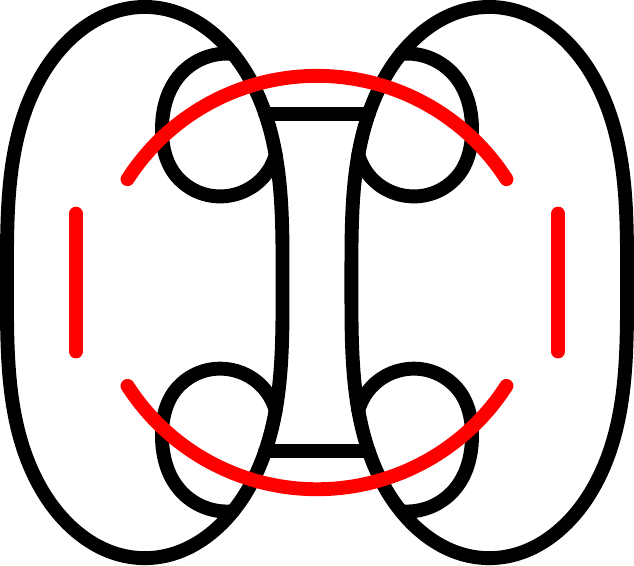}&\vSp\incj{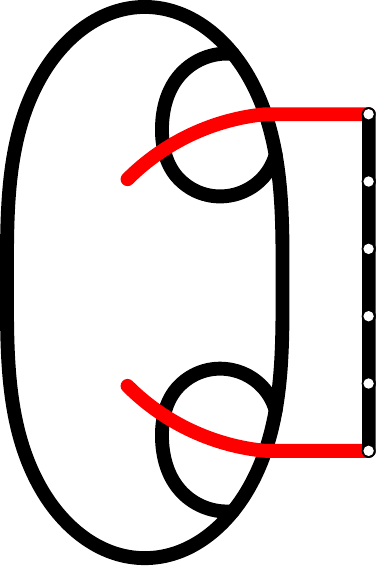}&\hspace{-1.5cm}\vSp\incj{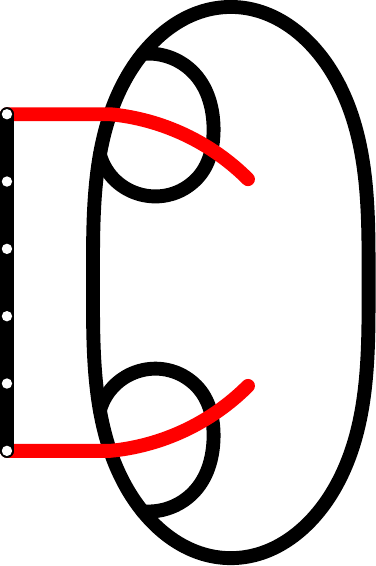}\\
(3.2)&\vSp\incj{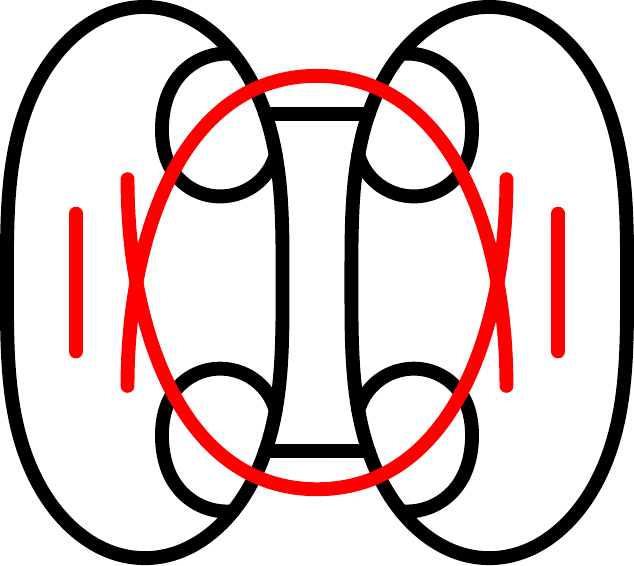}&\vSp\incj{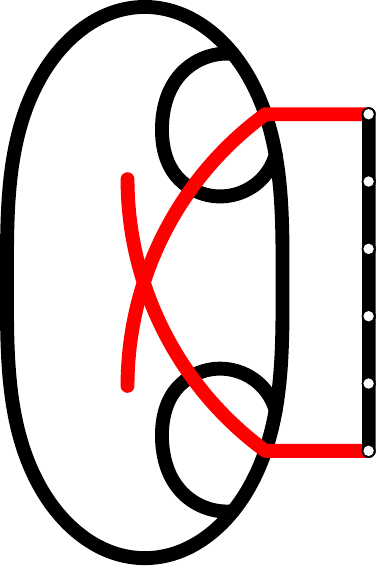}&\hspace{-1.5cm}\vSp\incj{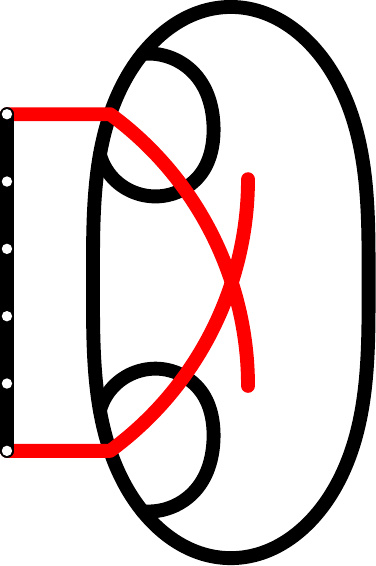}\\
\end{tabular}

  \caption{Type 3 potential solutions\label{f:t3}}
  \end{center}
\end{figure}

In what follows, for a set of paths $\mathcal P$, we denote $V(\mathcal P)=\bigcup_{P\in \mathcal P} V(P)$.

\begin{lemma}\label{k-IndPathsLemma}
  Let $(X_1,X_2,A_1,A_2,B_1,B_2)$ be a split of a minimally-sided
  2-join of $G\in\mathcal D$ such that $X_1$ is minimal side. Then
  $(G,\mathcal W)$ has a solution if and only if in one of the
  potential solutions listed above both $\mathcal S_1$ and
  $\mathcal S_2$ have a solution.
\end{lemma}

\begin{proof}
  Let $R^j$, for $j\in\{1,2\}$, be a chordless path in $G[X_j]$ whose
  one endnode is in $A_j$, the other in $B_j$ and no interior node is
  in $A_j\cup B_j$ (these paths exist by Lemma \ref{l:consistent}).
  Furthermore, let $\mathcal I_1$, $\mathcal I_2$, $\mathcal J$,
  $\mathcal W_1'$ and $\mathcal W_2'$ be defined as above.  If
  $|\mathcal J|\geq 3$, then $(G,\mathcal W)$ has no solution. So, we
  may assume that $|\mathcal J|\leqslant 2$.  Depending on how
  $\mathcal W$ interacts with the 2-join $(X_1,X_2)$ we define
  different options for $\mathcal S_1$ and $\mathcal S_2$ as above.
  For $i=1,2$, we denote by $\mathcal W_i$ the set of terminal pairs
  of the problem $\mathcal S_i$.

  By Lemma \ref{Star=Clique} and Lemma~\ref{extreme}(i) we have
  $|A_1|,|B_1|\geq 2$, and consequently, by Lemma~\ref{l:consistent},
  $A_2$ and $B_2$ are both cliques, and if $A_1$ (resp. $B_1$) is not
  a clique then $|A_2|=1$ (resp. $|B_2|=1$).

\medskip
\noindent($\Rightarrow:$) Let $\mathcal P=\{ P_1,\ldots ,P_k\}$ be a solution of $(G,\mathcal W)$ (where for $i=1,\ldots ,k$, $P_i$ is a path from $s_i$ to $t_i$)
 and let $\mathcal P_j=\{P_i\in\mathcal P\,:\,i\in\mathcal I_j\}$, for $j\in\{1,2\}$.
 Note that since $A_2$ and $B_2$ are cliques, if a path from $\mathcal P_2$ contains a vertex from $X_1$, then it contains vertices from both $A_1$ and $B_1$.
 We now consider the following cases (that correspond to different types of interaction of $\mathcal W$ with 2-join
 $(X_1,X_2)$ from the definition of $\mathcal S_1$ and
 $\mathcal S_2$).

\medskip
\noindent{\bf Case 1:} $\mathcal J=\emptyset$.
\medskip

First, let us assume that a path $P_i$ from $\mathcal P_2$ contains a vertex from $X_1$. Then $P_i$ contains vertices from both $A_1$ and $B_1$. Let $P_i'$ be the subpath of $P_i$ contained in $G[X_1]$. Then $\mathcal P_1\cup\{P_i'',c_2d_2\}$ is a solution of the problem $\mathcal S_1$ from (1.1), where $P_i''$ is the path induced by $V(P_i')\cup\{a_2,b_2\}$, and
$(\mathcal P_2\cup\{P_i'''\})\setminus\{P_i\}$ is a solution of the problem $\mathcal S_2$ from (1.1), where $P_i'''$ is the path induced by $(V(P_i)\setminus V(P_i'))\cup\{a_1,a_1',b_1,b_1',c_1,d_1\}$. So, in what follows, we may assume that
$V(\mathcal P_2) \cap X_1=\emptyset$.

Next assume that $V(\mathcal P_1)\cap X_2=\emptyset$.
If $V(\mathcal P_1)\cap (A_1\cup B_1)=\emptyset$,
then $\mathcal P_1\cup\{a_2a_2',b_2b_2'\}$ is a solution of $\mathcal S_1$ from (1.2) and $\mathcal P_2\cup\{c_1d_1\}$ is a solution of $\mathcal S_2$ from (1.2).
If $V(\mathcal P_1)\cap A_1\neq \emptyset$ and  $V(\mathcal P_1)\cap B_1= \emptyset$,
 then no path from $\mathcal P_2$ has a vertex in $A_2$  and hence $\mathcal P_1\cup\{b_2b_2',a_2'c_2\}$ is a solution of $\mathcal S_1$ from (1.3) and $\mathcal P_2\cup\{a_1a_1'\}$ is a solution of $\mathcal S_2$ from (1.3).
 Similarly, if
 $V(\mathcal P_1)\cap A_1= \emptyset$ and  $V(\mathcal P_1)\cap B_1\neq \emptyset$,
 then both $\mathcal S_1$ and $\mathcal S_2$ from (1.5) have a solution.
 If $V(\mathcal P_1)\cap A_1\neq \emptyset$ and  $V(\mathcal P_1)\cap B_1\neq \emptyset$,
 then no path from $\mathcal P_2$ has a vertex in $A_2\cup B_2$ and hence $\mathcal P_1\cup\{a_2'c_2d_2b_2'\}$ is a solution of $\mathcal S_1$ from (1.7) and $\mathcal P_2\cup\{a_1a_1',b_1b_1'\}$ is a solution of $\mathcal S_2$ from (1.7).

We may now assume that $V(\mathcal P_1)\cap X_2\neq \emptyset$.
First, we examine the case when $|V(\mathcal P_1)\cap X_2|=1$. Then $V(\mathcal P_1)\cap X_2=\{x\}$, and
$x\in A_2$ or $x\in B_2$. Suppose that $x\in A_2$. Then $A_1$ is not a clique, and hence $A_2=\{x\}$. Also, no vertex of a path from $\mathcal P_2$ is adjacent to or coincident with $x$. Let $P_i$ be the path of $\mathcal P_1$ that contains $x$, and let $P_i'$ be the path of $G^1$ obtained from $P_i$ by replacing $x$ with $a_2$. Also, let $\mathcal P_1'=(\mathcal P_1\setminus\{P_i\})\cup\{P_i'\}$.
If $V(\mathcal P_1)\cap B_1=\emptyset$, then  $\mathcal P_1'\cup\{b_2b_2'\}$ is a solution of $\mathcal S_1$ from (1.4) and
$\mathcal P_2\cup\{a_1x\}$ is a solution of $\mathcal S_2$ from (1.4).
 If $V(\mathcal P_1)\cap B_1\neq\emptyset$,  then $V(\mathcal P_2)\cap B_2=\emptyset$
  and hence $\mathcal P_1'\cup\{c_2d_2b_2'\}$ is a solution of $\mathcal S_1$ from (1.8)  and $\mathcal P_2\cup\{a_1x,b_1b_1'\}$
  is a solution of $\mathcal S_2$ from (1.8).
  So we have shown that if $x\in A_2$ then either both $\mathcal S_1$ and $\mathcal S_2$ from (1.4) have a solution, or
  both $\mathcal S_1$ and $\mathcal S_2$ from (1.8) have a solution.
  By symmetric argument, if $x\in B_2$ then either both $\mathcal S_1$ and $\mathcal S_2$ from (1.6) have a solution, or
  both $\mathcal S_1$ and $\mathcal S_2$ from (1.9) have a solution.

  Next suppose that $V(\mathcal P_1)\cap X_2=\{ a,b\}$, where $ab$ is not an edge. Without loss of generality, let $a\in A_2$ and $b\in B_2$. It follows,
  as above, that $A_1$ and $B_1$ are not cliques, and hence $A_2=\{a\}$ and $B_2=\{b\}$.
  Also, no vertex of a path from $\mathcal P_2$ is adjacent to $a$ or $b$. Let $P_i$ (resp.\ $P_j$) be the path of $\mathcal P_1$ that contains $a$ (resp.\ $b$), and let $P_i'$ (resp.\ $P_j'$) be the path of $G^1$ obtained from $P_i$ (resp.\ $P_j$) by replacing $a$ (resp.\ $b$) with $a_2$ (resp.\ $b_2$). (Note that in this case possibly $P_i=P_j$, i.e.\ the path contains both $a$ and $b$. In this case $P_i'=P_j'$ is the path obtained from $P_i=P_j$ by replacing $a$ with $a_2$ and $b$ with $b_2$.)  Then
  $(\mathcal P_1\setminus\{P_i,P_j\})\cup\{P_i',P_j',c_2d_2\}$ a solution of $\mathcal S_1$ from (1.10), and
$\mathcal P_2\cup\{aa_1,bb_1\}$ is a solution of $\mathcal S_2$
from (1.10).

  Finally, suppose that a path $P_i$ from $\mathcal P_1$ contains an edge of $X_2$.
  Let  $P_i'$ be the subpath of $P_i$ contained in $G[X_2]$, $P_i''$ the path induced by
   $(V(P_i)\setminus V(P_i'))\cup\{a_2,a_2',b_2,b_2',c_2,d_2\}$, and  $P_i'''$ the path induced by $V(P_i')\cup\{a_1,b_1\}$.
  Then $(\mathcal P_1\setminus\{P_i\})\cup\{P_i''\}$ is a solution of $\mathcal S_1$ from (1.11) and
  $\mathcal P_2\cup\{P_i''',c_1d_1\}$ is a solution of $\mathcal S_2$ from (1.11).

\medskip
\noindent{\bf Case 2:} $|\mathcal J|=1$.
\medskip

Let $\mathcal J=\{j\}$. So $P_j\in\mathcal P$ is a  path from $s_j$ to $t_j$.
We consider the case when $V(P_j)\cap A_1$ and $V(P_j)\cap A_2$ are both non-empty, and show that in that case
in one of the cases (2.1), (2.3) or (2.5), both $\mathcal S_1$ and $\mathcal S_2$ have a solution.
The case when $V(P_j)\cap B_1$ and $V(P_j)\cap B_2$ are both non-empty is handled similarly and leads to
$\mathcal S_1$ and $\mathcal S_2$ both having a solution in one of the cases (2.2), (2.4) or (2.6).
So assume that  both $V(P_j)\cap A_1$ and $V(P_j)\cap A_2$ are non-empty.  Let $P_j^1$ (resp.\ $P_j^2$) be the
 $s_ja'$-subgraph (resp.\ $a''t_j$-subpath) of $P_j$, where $V(P_j)\cap A_1=\{a'\}$ (resp.\ $V(P_j)\cap A_2=\{a''\}$). Furthermore, let $P_j'$ be the
 path induced by $V(P_j^1)\cup\{a_2,a_2',c_2\}$ and  $P_j''$ be the path induced by $V(P_j^2)\cup \{a_1\}$.

 First, let us assume that no path from $\mathcal P_1\cup\{P_j^1\}$ has a vertex in $B_2$.
 If no path from $\mathcal P_1\cup \{P_j^1\}$ has a vertex in $B_1$, then $\mathcal P_1\cup\{P_j',b_2b_2'\}$ is a solution of
 $\mathcal S_1$ from (2.1) and $\mathcal P_2\cup \{P_j'',b_1'd_1\}$ is a solution of $\mathcal S_2$ from (2.1).
  If some path from $\mathcal P_1\cup \{P_j^1\}$ has a vertex in $B_1$, then no path from $\mathcal P_2$  has a vertex in $B_2$. Hence, $\mathcal P_1\cup\{P_j'\setminus\{c_2\},b_2'd_2\}$ is a solution of $\mathcal S_1$ from (2.3) and
  $\mathcal P_2\cup \{P_j'',b_1b_1'\}$ is a solution of  $\mathcal S_2$ from (2.3).

  Let us now assume that a path from $\mathcal P_1\cup\{P_j^1\}$ has a vertex in $B_2$.
  Then  this path has two vertices in $B_1$, which implies that $B_1$ is not a clique and hence that $B_2$ has a single vertex $b$. So, no vertex from a path from $\mathcal P_2$ is adjacent to $b$. Let $P_i$ be the path of $\mathcal P_1\cup\{P_j^1\}$ that contains $b$, and let $P_i'$ be the path of $G^1$ obtained from $P_i$ by replacing $b$ with $b_2$. Now, if $P_i\neq P_j^1$, then $(\mathcal P_1\cup\{P_j'\setminus\{a_2',c_2\},P_i',c_2d_2\})\setminus\{P_i\}$  is a
  solution of $\mathcal S_1$ from (2.5);  if $P_i=P_j^1$, then $\mathcal P_1\cup\{P_i'\setminus\{a_2',c_2\},c_2d_2\}$ is a
  solution of $\mathcal S_1$ from (2.5). Clearly, $\mathcal P_2\cup \{P_j'',bb_1\}$ is a solution of  $\mathcal S_2$ from (2.5).

\medskip
\noindent{\bf Case 3:} $|\mathcal J|=2$.
\medskip

Let $\mathcal J=\{j_1,j_2\}$. So $P_{j_1},P_{j_2}\in\mathcal P$ are   paths from $s_{j_1}$ to $t_{j_1}$ and from $s_{j_2}$ to $t_{j_2}$, respectively.
Then no path from $\mathcal P\setminus\{P_{j_1},P_{j_2}\}$ has a vertex in both $X_1$ and $X_2$.
For $s,r\in\{1,2\}$, let $P_{j_r}^s$ be the subpath of $P_{j_r}$ contained in $G[X_s]$. Let $P_{j_r}'$ (resp.\ $P_{j_r}''$), for $r\in\{1,2\}$, be the path induced by $V(P_{j_r}^1)\cup \{a_{2}\}$ (resp.\ $V(P_{j_r}^2)\cup \{a_{1}\}$) if an endnode of $P_{j_r}^1$ (resp.\ $P_{j_r}^2$) is in $A_1$ (resp.\ $A_2$), or the path induced by $V(P_{j_r}^1)\cup \{b_2\}$ (resp.\ $V(P_{j_r}^2)\cup \{b_1\}$) if an endnode of $P_{j_r}^1$ (resp.\ $P_{j_r}^2$) is in $B_1$ (resp.\ $B_2$). So, if an endnode of $P_{j_1}^1$ is in $A_1$ (resp.\ $B_1$), then $\mathcal P_1\cup\{P_{j_1}',P_{j_2}'\}$ is a solution of $\mathcal S_1$ from (3.1) (resp.\ (3.2)) and $\mathcal P_2\cup\{P_{j_1}'',P_{j_2}''\}$ is a solution for $\mathcal S_2$ from (3.1)
(resp.\ from (3.2)).

\medskip

\noindent ($\Leftarrow:$) Let $\mathcal P_1$ be a solution of $\mathcal S_1$ and $\mathcal P_2$ a solution of $\mathcal  S_2$ of some potential solution. We consider the following cases.

 \medskip
\noindent{\bf Case 1:}  $\mathcal J=\emptyset$.
\medskip

Suppose that $\mathcal S_1$ and $\mathcal S_2$ are from (1.1).
Let $P_i\in\mathcal P_1$ be the path from $a_2$ to $b_2$.
Since $c_2d_2\in \mathcal P_1$, $P_i$ contains a vertex of $A_1$ and a vertex of $B_1$.
So no path from $\mathcal P_1\setminus\{P_i,c_2d_2\}$ contains a vertex from $A_1\cup B_1$.
Hence, if no path from $\mathcal P_2$ contains a vertex from $\{a_1,a_1',b_1,b_1',c_1,d_1\}$, then
$(\mathcal P_1\cup\mathcal P_2)\setminus\{P_i,c_2d_2\}$ is a solution of $(G,\mathcal W)$.
So we may assume that $P_j\in\mathcal P_2$ contains a vertex of $\{a_1,a_1',b_1,b_1',c_1,d_1\}$. Since $A_2$ and $B_2$ are cliques it follows
that $P_j$ contains all vertices of $\{a_1,a_1',b_1,b_1',c_1,d_1\}$. Let $P_j'$ be the subgraph of $P_j$ contained in $G[X_2]$. Furthermore,
let $P_i'$ be the path induced by $V(P_i)\setminus\{a_2,b_2\}$, and $P_j''$ be the path induced by $V(P_i')\cup V(P_j')$.
Then $(\mathcal P_1\setminus\{P_i,c_2d_2\})\cup (\mathcal P_2\setminus\{P_j\})\cup\{P_j''\}$ is a solution for
$(G,\mathcal W)$.

If $\mathcal S_1$ and $\mathcal S_2$ are from (1.2), then $V(\mathcal P_1)\setminus \{ a_2,a_2',b_2,b_2'\} \subseteq X_1\setminus (A_1\cup B_1)$ and
$V(\mathcal P_2)\setminus \{ c_1,d_1\} \subseteq X_2$ (as $A_2$ and $B_2$ are cliques), and hence
$(\mathcal P_1\cup\mathcal P_2)\setminus\{c_1d_1,a_2a_2',b_2b_2'\}$ is a solution of $(G,\mathcal W)$.

If $\mathcal S_1$ and $\mathcal S_2$ are from (1.3) (resp.\ (1.5)), then $V(\mathcal P_1)\setminus \{ a_2',b_2,b_2',c_2\} \subseteq X_1\setminus B_1$ (resp.\  $V(\mathcal P_1)\setminus \{ a_2,a_2',b_2',d_2\} \subseteq X_1\setminus A_1$) and no path from $\mathcal P_2$ has a vertex in
$A_2$ (resp.\ $B_2$). Hence $(\mathcal P_1\cup\mathcal P_2)\setminus\{a_1a_1',b_2b_2',a_2'c_2\}$ (resp.\ $(\mathcal P_1\cup\mathcal P_2)\setminus\{b_1b_1',a_2a_2',b_2'd_2\}$) is a solution of $(G,\mathcal W)$.

 Suppose that $\mathcal S_1$ and $\mathcal S_2$ are from (1.4).
 Then no path of $\mathcal P_2 \setminus \{ aa_1\}$ contains a vertex that
 is adjacent to or coincident with the vertex of $A_2$, and no path
 of $\mathcal P_1$ contains a vertex of $B_1$.
 So, if $V(\mathcal P_1) \cap \{ a_2\} =\emptyset$, then clearly $(\mathcal P_1 \cup \mathcal P_2)\setminus \{ aa_1,b_2b_2'\})$ is
 a solution of $(G,\mathcal W)$. Now suppose that a path $P_i\in \mathcal P_1$ contains $a_2$. Let $P_i'$ be the path obtained from $P_i$ by replacing $a_2$ by $a$. Then  $(\mathcal P_1\cup\mathcal P_2\cup\{P_i'\})\setminus\{P_i,aa_1,b_1b_1'\}$  is a solution of $(G,\mathcal W)$.
 Similarly, if $\mathcal S_1$ and $\mathcal S_2$ are from (1.6), then $(G,\mathcal W)$ has a solution.

 If  $\mathcal S_1$ and $\mathcal S_2$ are from (1.7), then no path from $\mathcal P_2$ has a vertex in
 $A_2\cup B_2$ and hence $(\mathcal P_1\cup\mathcal P_2)\setminus\{a_1a_1',b_1b_1',P_t\}$ is a solution of $(G,\mathcal W)$, where $P_t\in\mathcal P_1$ is the path from $a_2'$ to $b_2'$.

 Suppose that $\mathcal S_1$ and $\mathcal S_2$ are from (1.8).
Then no path of $\mathcal P_2\setminus\{aa_1,b_1b_1'\}$ contains a vertex
that is adjacent to or coincident with the vertex of $A_2$,
nor coincident with a vertex of $B_2$. Let $P_t$ be the path of $\mathcal S_1$ from $c_2$ to $b_2'$.
If $V(\mathcal P_1)\cap \{ a_2\} =\emptyset$, then clearly
$(\mathcal P_1\cup\mathcal P_2)\setminus\{aa_1,b_1b_1',P_t\}$ is a solution of
 $(G,\mathcal W)$.
So, suppose that a path $P_i\in \mathcal P_1$ contains $a_2$.  Let $P_i'$ be the path obtained from $P_i$ by replacing $a_2$ by $a$. Then
$(\mathcal P_1\cup\mathcal P_2\cup\{P_i'\})\setminus\{P_i,aa_1,b_1b_1',P_t\}$ is a solution of
 $(G,\mathcal W)$.
 Similarly, if $\mathcal S_1$ and $\mathcal S_2$ are from (1.9), then $(G,\mathcal W)$ has a solution.

Suppose that $\mathcal S_1$ and $\mathcal S_2$ are from (1.10).
Then no path of $\mathcal P_2$ contains a vertex that is adjacent to or coincident with a node of $A_2\cup B_2$.
If $a_2$ (resp. $b_2$) is contained in some path of $\mathcal P_1$ then replace it by $a$ (resp. $b$).
Let $\mathcal P_1'$ be the resulting set of paths. Then $(\mathcal P_1' \cup \mathcal P_2)\setminus\{aa_1,bb_1,c_2d_2\}$ is a solution of $(G,\mathcal W)$.

Finally suppose that  $\mathcal S_1$ and $\mathcal S_2$ are from (1.11).
Let $P'$ be the path of $\mathcal P_2$ from $a_1$ to $b_1$.
If $V(\mathcal P_1) \cap \{ a_2,a_2',c_2,d_2,b_2',b_2\}=\emptyset$, then
 $\mathcal P_1\cup (\mathcal P_2\setminus\{P',c_1d_1\})$ is a solution of $(G,\mathcal W)$. If some path of $\mathcal P_1$ contains $a_2$ (resp.\ $b_2$) and no other vertex of the marker path, then replace $a_2$ (resp.\ $b_2$) in that path by a vertex $a\in A_2$ (resp.\ $b\in B_2$), and let $\mathcal P_1'$ be the resulting family of paths. Note that by Lemma \ref{l:consistent} 2-join $(X_1,X_2)$ is consistent and hence we may choose $a$ and $b$ to be non-adjacent. Then  $\mathcal P_1'\cup (\mathcal P_2\setminus\{P',c_1d_1\})$ is a solution of $(G,\mathcal W)$.
 So we may assume that some path $P_i$ of $\mathcal P_1$ contains  all of the vertices $a_2,a_2',c_2,d_2,b_2',b_2$.
 Let $P_i'$ be the path of $G$ induced by $(V(P_i)\cap X_1)\cup (V(P')\setminus \{ a_1,b_1\})$. Then $(\mathcal P_1\setminus\{P_i\})\cup (\mathcal P_2\setminus\{P',c_1d_1\})\cup\{P_i'\}$ is a solution for $(G,\mathcal W)$.

\medskip
\noindent{\bf Case 2:} $|\mathcal J|=1$.
\medskip

Let $\mathcal J=\{j\}$, and $P_j^1\in\mathcal P_1$  be the path from $s_j$ to the appropriate vertex of the marker path $P^2$. First, suppose that
$\mathcal S_1$ and $\mathcal S_2$
are from (2.1) (resp.\ (2.2)), and let $P_j^2\in\mathcal P_2$  be the path from $t_j$ to $a_1$ (resp.\ $b_1$). Furthermore, let $P_j$ be the path induced by $(V(P_j^1)\cup V(P_j^2))\setminus\{a_2,a_2',c_2,a_1\}$ (resp.\ $(V(P_j^1)\cup V(P_j^2))\setminus\{b_2,b_2',d_2,b_1\}$). Then no path from $\mathcal P_1$ has a vertex in $B_1$ (resp.\ $A_1$), and hence $(\mathcal P_1\cup\mathcal P_2\cup \{P_j\})\setminus\{P_j^1,P_j^2,b_1'd_1,b_2b_2'\}$ (resp.\ $(\mathcal P_1\cup\mathcal P_2\cup \{P_j\})\setminus\{P_j^1,P_j^2,a_1'c_1,a_2a_2'\}$) is a solution for $(G,\mathcal W)$.

Next, suppose that $\mathcal S_1$ and $\mathcal S_2$ are
from (2.3) (resp.\ (2.4)), and let $P_j^2\in\mathcal P_2$  be the path from $t_j$ to $a_1$ (resp.\ $b_1$). Furthermore, let $P_j$ be the path induced by $(V(P_j^1)\cup V(P_j^2))\setminus\{a_2,a_2',a_1\}$ (resp.\ $(V(P_j^1)\cup V(P_j^2))\setminus\{b_2,b_2',b_1\}$). Then no path from $\mathcal P_2$ has a vertex in $B_2$ (resp.\ $A_2$), and hence $(\mathcal P_1\cup\mathcal P_2\cup \{P_j\})\setminus\{P_j^1,P_j^2,b_1b_1',b_2'd_2\}$ (resp.\ $(\mathcal P_1\cup\mathcal P_2\cup \{P_j\})\setminus\{P_j^1,P_j^2,a_1a_1',a_2'c_2\}$) is a solution for $(G,\mathcal W)$.

Suppose that $\mathcal S_1$ and $\mathcal S_2$ are from (2.5),
and let $P_j^2\in\mathcal P_2$  be the path from $t_j$ to $a_1$.
If $P_j^1$ contains $b_2$ then let $P_j$ be the path induced by
$(V(P_j^1)\cup V(P_j^2)\setminus\{a_2,b_2, a_1\})\cup \{ b\}$, and otherwise
let $P_j$ be the path induced by
$(V(P_j^1)\cup V(P_j^2))\setminus\{a_2,a_1\}$.
If no path of $\mathcal P_1 \setminus P_j^1$ contains $b_2$ then
$(\mathcal P_1 \cup \mathcal P_2\cup \{ P_j\})\setminus \{ P_j^1,P_j^2,bb_1,c_2d_2\}$ is a solution for $(G,\mathcal W)$.
So suppose that $P_i\in \mathcal P_1\setminus \{ P_j^1\}$ contains $b_2$. Let $P_i'$ be the path obtained from $P_i$
by replacing $b_2$ with $b$. Then
$(\mathcal P_1 \cup \mathcal P_2\cup \{ P_i', P_j\})\setminus \{ P_i, P_j^1,P_j^2,bb_1,c_2d_2\}$ is a solution for $(G,\mathcal W)$.
Similarly, if $\mathcal S_1$ and $\mathcal S_2$ are from (2.6), then $(G,\mathcal W)$ has a solution.

\medskip
\noindent{\bf Case 3:} $|\mathcal J|=2$.
\medskip

Let $\mathcal J=\{j_1,j_2\}$. Suppose that  $\mathcal S_1$ and $\mathcal S_2$ are from (3.1).
Let $P_{j_1}^1,P_{j_2}^1\in\mathcal P_1$  be the  paths from $s_{j_1}$ to $a_{2}$  and from $s_{j_2}$ to $b_{2}$,
 respectively, and let $P_{j_1}^2,P_{j_2}^2\in\mathcal P_2$  be the  paths from $t_{j_1}$ to $a_{1}$ and from $t_{j_2}$ to $b_{1}$, respectively. Furthermore, let $P_{j_1}$ be the path induced by
 $(V(P_{j_1}^1)\cup V(P_{j_1}^2))\setminus\{a_1,a_2\}$  and $P_{j_2}$ be the path induced by $(V(P_{j_2}^1)\cup V(P_{j_2}^2))\setminus\{b_1,b_2\}$. Then $(\mathcal P_1\cup\mathcal P_2\cup\{P_{j_1},P_{j_2}\})\setminus\{P_{j_1}^1,P_{j_1}^2,P_{j_2}^1,P_{j_2}^2\}$ is a solution for $(G,\mathcal W)$.
 Similarly, if $\mathcal S_1$ and $\mathcal S_2$ are from (3.2), then $(G,\mathcal W)$ has a solution.
\end{proof}


By the previous lemma, in order to solve $(G,\mathcal W)$ it is enough
to solve $\mathcal S_1$ and $\mathcal S_2$ in each of the potential
solutions explained above. This will be recorded in the corresponding
o-graph. Further, we note that in our main algorithm the graph $G^1$
is going to be basic (we use the decomposition tree $T_G$), so all the
problems $\mathcal S_1$ can be solved using Lemma
\ref{k-IndPathsInBasic}.

Let $G_{\mathcal F,\mathcal O}$ be an o-graph such that
$G\in\mathcal D$ has a 2-join $(X_1,X_2)$. Further, we assume that
$(X_1,X_2)$ is a minimally-sided 2-join of $G$ such that $X_1$ is the
minimal side. Let $G^1$ and $G^2$ be the blocks of decompositions
w.r.t.\ this 2-join and $P^1$ and $P^2$ the marker paths used to build
these blocks ($P^j$ is contained in $G^{3-j}$, for $j\in\{1,2\}$). Let
the {\it extended marker path} $P^1_{\rm{ext}}$ of $G^2$ be the path
obtained by adding to $P^1$ vertex $a$ if it is the unique vertex of
$A_2$ and $b$ if it is the unique vertex of $B_2$. Note that, by this
definition, $P^1_{\rm{ext}}$ is a flat path of length at most 7. In
what follows we define the o-graph $G^2_{\mathcal F^2,\mathcal O^2}$.

For $j\in\{1,2\}$, let $\mathcal F_j'$ be the set of flat paths
$\mathcal F$ that are contained in $X_{j}$, together with the set of
paths $P\in\mathcal F$ of length 0 contained in $X_j$. Let
$\mathcal F^0=\mathcal F\setminus(\mathcal F_1'\cup\mathcal
F_2')$. So, $\mathcal F^0$ contains only flat paths.  By Lemma
\ref{extreme} and Remark \ref{remark1}, for every $P\in\mathcal F^0$,
the set ${P}\cap X_2$ is of size 1 and contained in $A_2\cup
B_2$. Furthermore, since $P$ is a flat path of length greater than~1,
if $P \cap A_2\neq \emptyset$, then $|A_2| = 1$ and similarly if
$P \cap B_2\neq \emptyset$, then $|B_2| = 1$.

Let us now define $\mathcal F^2$ and $\mathcal O^2$. If
$\mathcal F_1' \cup \mathcal F^0 = \emptyset$, then
$\mathcal F^2 = \mathcal F_2'$ and otherwise
$\mathcal F^2 = \mathcal F_2' \cup \{P^1_{\rm{ext}}\}$.  Next, for
every $\mathcal W\in\mathcal O$ and every potential solution
associated with the problem $(G,\mathcal W)$ we solve the
corresponding problem $\mathcal S_1$ and if it has a solution we add
the set of terminal pairs of the corresponding problem $\mathcal S_2$
to~$\mathcal O^2$, except when $\mathcal S_1$ from (1.1) has a
solution.  In this last case, we add the set of terminal pairs of the
corresponding problem $\mathcal S_2$ to $\mathcal O^2$
and we disregard all other potential solutions for~$\mathcal W$.  By
construction, all terminal vertices of $\mathcal S_2$ are in the paths
from $\mathcal F_{2}'\cup\{P^1_{\rm{ext}}\}$. Further, if each
terminal vertex of $\mathcal W$ is in a path from $\mathcal F_2'$ (in
particular when $\mathcal F_1'\cup\mathcal F_0 = \emptyset$), then all
terminal vertices of a set that is added to $\mathcal O^2$ are in the
paths from $\mathcal F_{2}'$. Indeed, then in potential solution (1.1)
we have $\mathcal S_1=(G^1,\{(a_2,b_2),(c_2,d_2)\})$, which has a
solution since $(X_1,X_2)$ is a consistent 2-join (by Lemma
\ref{l:consistent}).  Hence, only the set of terminals for
$\mathcal S_2$ from the potential solution (1.1) is added to
$\mathcal O^2$, and this set contains only vertices from paths of
$\mathcal F_2'$.  So, $(G^2,\mathcal F^2,\mathcal O^2)$ is a
well-defined o-graph.


\begin{lemma}\label{l:main-o-graphs}
Let $G_{\mathcal F,\mathcal O}$ be an o-graph such that $G\in\mathcal
D$, and let $(X_1,X_2)$ be a minimally sided 2-join of $G$ such that
$X_1$ is minimal side. Further, let $G^1$ and $G^2$ be the blocks of
decompositions w.r.t.\ this 2-join and  $G^2_{\mathcal F^2,\mathcal
  O^2}$ o-graph constructed above. Then $G_{\mathcal F,\mathcal O}$ is
linkable if and only if  $G^2_{\mathcal F^2, \mathcal O^2}$ is linkable.
\end{lemma}

\begin{proof}
  Suppose $G^2_{\mathcal F^2, \mathcal O^2}$ is linkable.  Then
  $G_{\mathcal F,\mathcal O}$ is linkable as shown by
  Lemma~\ref{k-IndPathsLemma} and the definition of
  $G^2_{\mathcal F^2, \mathcal O^2}$.

  So, suppose that $G_{\mathcal F,\mathcal O}$ is linkable, that is
  some problem $(G, \mathcal W)$ with $\mathcal W \in \mathcal O$ has
  a solution. Then the converse also follows from
  Lemma~\ref{k-IndPathsLemma}, except in the case when for
  $(G,\mathcal W)$ the problem $\mathcal S_1$ from (1.1) has a
  solution. Recall that in this case we add the set of terminal pairs
  of problem $\mathcal S_2$ from (1.1) to $\mathcal O^2$ and disregard
  all other potential solutions for $\mathcal W$.  Since $\mathcal W$
  is of type~1 and $(G,\mathcal W)$ has a solution, by
  Lemma~\ref{k-IndPathsLemma} problem $\mathcal S'_2$ from (1.$i$),
  for some $i=1,\dots, 11$ has a solution.  This implies that
  $\mathcal S_2$ from (1.1) has a solution since the set of pairs of
  terminals for $\mathcal S_2$ is a subset of the one for
  $\mathcal S'_2$.  Hence, $G^2_{\mathcal F^2, \mathcal O^2}$ is
  linkable.
\end{proof}

\begin{theorem}\label{linkable-2join}
  Let $c\in\mathbb N$ be a constant. There is an algorithm with the
  following specifications:
\begin{description}
\item[ Input:] An o-graph $G_{\mathcal F,\mathcal O}$, where
  $|\mathcal F|\leq c$ and $G\in\mathcal D$.
\item[ Output:] YES if  $G_{\mathcal F,\mathcal O}$ is linkable, and NO otherwise.
\item[ Running time:] $\mathcal O(n^6)$.
\end{description}
\end{theorem}

\begin{proof}
In \cite{twf-p2} an $\mathcal O(n^2m)$-time algorithm is given for recognizing whether a graph belongs to $\mathcal B$.
If $G\in\mathcal B$, then the problem can be solved in time $\mathcal O(n^5)$ using Corollary \ref{l:BasicO-graph}. So we may assume that $G\in \mathcal D \setminus \mathcal B$.
By Theorem \ref{decomposeTW}, $G$ has a 2-join, and by Lemma \ref{Star=Clique} $G$ has no star cutset.
Using Lemma \ref{DT-construction} we build a 2-join decomposition tree $T_G$ in time $\mathcal O(n^4m)$ and use the notation from the definition of $T_G$.
\smallskip

\noindent{\bf Description and correctness of the algorithm.} We now process the decomposition tree $T_G$ from its root $G^0=G$ and the associated o-graph $G^0_{\mathcal F^0,\mathcal O^0}$ (where $\mathcal F^0=\mathcal F$ and $\mathcal O^0=\mathcal O$). For each $0\leq i\leq p-1$ we build the o-graph $(G^{i+1},\mathcal F^{i+1},\mathcal O^{i+1})$.  Then, by Lemma \ref{l:main-o-graphs}, $G_{\mathcal F,\mathcal O}$ is linkable if and only if $(G^{p},\mathcal F^{p},\mathcal O^{p})$ is linkable. So, it is enough to check whether $(G^{p},\mathcal F^{p},\mathcal O^{p})$ is linkable; since $G^p\in\mathcal B$, this can be done using Corollary \ref{l:BasicO-graph} (in the next paragraph we will verify that all hypothesis to Corollary \ref{l:BasicO-graph} hold).
\smallskip

\noindent{\bf Complexity of the algorithm.} Let $|\mathcal F|=t$. We
prove that for each $i\in\{0,1,\ldots,p\}$,
$|{\mathcal F}^{i}|\leq t$.  Our proof is by induction on $i$. So,
suppose that this is true for $i\leq p-1$ and let us prove it for
$i+1$. Note that ${\mathcal F}^{i+1}$ is obtained from
${\mathcal F}^{i}$ by adding at most one element (that corresponds to
the extended marker path of $G^{i+1}$) and removing all elements from
${\mathcal F}^{i}$ that are in the corresponding set
$\mathcal F_1'\cup\mathcal F^0$. So,
$|{\mathcal F}^{i+1}|\leq|{\mathcal F}^{i}|\leq t$, unless in
$(G^i,\mathcal F^i,\mathcal O^i)$ the set of corresponding paths
$\mathcal F_1' \cup \mathcal F^0$ is empty. But then the extended
marker path of $G^{i+1}$ is not an element of $\mathcal F^{i+1}$, and
hence ${\mathcal F}^{i+1}\subseteq{\mathcal F}^{i}$, which implies
$|{\mathcal F}^{i+1}|\leq t$. In particular, since
$|\mathcal F|\leq c$, we have $|\mathcal F^i|\leq c$, for
$0\leq i\leq p$.

Let us now examine the complexity of our algorithm. By Lemma
\ref{l:boundedNumber} and previous paragraph, for each
$i\in\{0,1,\ldots,p-1\}$, to build o-graphs
$(G^{i+1},\mathcal F^{i+1},\mathcal O^{i+1})$ from the o-graph
$(G^{i},\mathcal F^{i},\mathcal O^{i})$ it is enough to solve at most
$11\cdot|\mathcal O^i|\leq 11\cdot 2^{8c}(8c)!$ \textsc{Induced
  Disjoint Paths} problems on $G_B^{i+1}$ (since for each problem
$(G^i,\mathcal W)$, where $\mathcal W\in\mathcal O^{i}$ we need to
solve at most 11 problems from the potential solutions) and each of
these problems has at most $4|\mathcal F^{i}| \leq 4c$ terminal
pairs. So, by Lemma~\ref{k-IndPathsInBasic}, we can build o-graphs
$(G^{k+1},\mathcal F^{k+1},\mathcal O^{k+1})$, for
$i\in\{0,1,\ldots,p-1\}$, in time
$\mathcal O(p\cdot n^5)=\mathcal O(n^6)$ (since $p=\mathcal O(n)$ by
Lemma \ref{DT-construction}). Finally, using Corollary
\ref{l:BasicO-graph} we can check if
$(G^{p},\mathcal F^{p},\mathcal O^{p})$ is linkable in time
$\mathcal O(n^5)$.
\end{proof}

\begin{theorem}\label{k-IndPaths-2join}
For a fixed integer $k$, there is an algorithm with the following specifications:
\begin{description}
\item[ Input:] A graph $G\in\mathcal D$ and a set of pairs $\mathcal W=\{(s_1,t_1),(s_2,t_2),\ldots,(s_k,t_k)\}$ of vertices of $G$ such that all $2k$ vertices are distinct and the only possible edges between these vertices are of the form $s_it_i$, for some $1\leqslant i\leqslant k$.
\item[ Output:]
YES if the problem $(G,\mathcal W)$ has a solution, and NO otherwise.
\item[ Running time:] $\mathcal O(n^6)$.
\end{description}
\end{theorem}

\begin{proof}
  Let $G_{\mathcal F,\mathcal O}$ be the o-graph defined with
  $\mathcal F=\{s_i\,:\,1\leq i\leq k\}\cup\{t_i\,:\,1\leq i\leq k\}$
  and $\mathcal O=\{\mathcal W\}$. Then
  $G_{\mathcal F,\mathcal O}$ is linkable if and only if
  $(G,\mathcal W)$ has a solution. Hence, to solve the given problem
  it is enough to apply Theorem \ref{linkable-2join} for
  $G_{\mathcal F,\mathcal O}$.
\end{proof}

\begin{corollary}
  For graphs in $\mathcal D$ the $k$-\textsc{Induced Disjoint Paths}
  problem is fixed-parameter tractable, when parameterized by $k$.
\end{corollary}
\begin{proof}
  Let $(G,\mathcal W)$ be an instance of the $k$-\textsc{Induced
    Disjoint Paths} problem.  As in the proof of Corollary
  \ref{k-in-a-cycle-FPT}, we conclude that the problem
  $(G,\mathcal W)$ can be solved in time $2^{2k}h(k)n^5$ for graphs in
  $\mathcal B$. Now, for each $i\in\{0,1,\ldots,p\}$ the algorithm
  from Theorem \ref{k-IndPaths-2join} has at most
  $11\cdot 2^{16k}(16k)!$ calls to the algorithm from
  Lemma~\ref{k-IndPathsInBasic} (here $c=2k$), $p\leq n$ (by
  Lemma~\ref{DT-construction}) and in each call of
  Lemma~\ref{k-IndPathsInBasic} we solve a problem with at most $8k$
  terminal pairs.  We conclude that the problem $(G,\mathcal W)$ can
  be solved in time
  $$(p+1)\cdot 11 \cdot 2^{16k} (16k)!  2^{16k} h(8k) n^5 \leq
  22 \cdot 2^{32k} (16k)!h(8k)n^6$$ for graphs in
  $\mathcal D$.
\end{proof}


\subsection{Induced disjoint paths on $\mathcal C$ and some related problems}
\label{sub:related}

Theorems \ref{decomposeTW},
\ref{k-IndPaths-Cliques} and \ref{k-IndPaths-2join} directly imply the following theorem.

\begin{theorem}\label{k-IndPaths}
There is an algorithm with the following specifications:
\begin{description}
\item[ Input:] A graph $G\in\mathcal C$ and a set of pairs $\mathcal W=\{(s_1,t_1),(s_2,t_2),\ldots,(s_k,t_k)\}$ of vertices of $G$ such that all $2k$ vertices are distinct and the only possible edges between these vertices are of the form $s_it_i$, for some $1\leqslant i\leqslant k$.
\item[ Output:]
YES if the problem $(G,\mathcal W)$ has a solution, and NO otherwise.
\item[ Running time:] $\mathcal O(n^{2k+6})$.
\end{description}
\end{theorem}

In what follows we consider several problems (for graphs in $\mathcal C$) that are related to the $k$-\textsc{Induced Disjoint Paths} problem. The first is in fact its generalization.

Let $k$ be a fixed integer (that is not part of the input), $G$ a graph and $\mathcal W=\{(s_1,t_1),(s_2,t_2),\ldots,(s_k,t_k)\}$ a multiset of pairs  of vertices of $G$ (so, pairs need not be disjoint, there may be edges between vertices from $W=\bigcup_{i=1}^k\{s_i,t_i\}$ that are not of the form $s_it_i$ and we allow $(s_i,t_i)=(s_j,t_j)$, for $i\neq j$). Again, we say that vertices of $W$ are {\it terminals} of $\mathcal W$.

We consider the following problem: decide whether there exist $k$ paths $P_i=s_i\ldots t_i$, $1\leqslant i\leqslant k$, that are vertex-disjoint except that they can have a common endnode, and such that there are no edges between vertices of these paths except those on the paths and the ones from $G[W]$. For this problem we use the same notation as for the $k$-\textsc{Induced Disjoint Paths} problem, that is, we denote it with $(G,\mathcal W)$.

\begin{theorem}\label{k-GeneralIndPaths}
There is an algorithm with the following specifications:
\begin{description}
\item[ Input:] A graph $G\in\mathcal C$ and a multiset $\mathcal W=\{(s_1,t_1),(s_2,t_2),\ldots,(s_k,t_k)\}$ of pairs  of vertices of $G$.
\item[ Output:]
YES if the problem $(G,\mathcal W)$ has a solution, and NO otherwise.
\item[ Running time:] $\mathcal O(n^{4k+6})$.
\end{description}
\end{theorem}

\begin{proof}
Let $W=\bigcup_{i=1}^k\{s_i,t_i\}$ be the set of terminals of $\mathcal W$ in $G$. First, let $\mathcal Q$ be the multiset
of all pairs of type $(v,v)$ from $\mathcal W$. Furthermore, let $Q=\bigcup_{(v,v)\in\mathcal Q}\{ v\}$ and
$N=\left(Q\cup\bigcup_{v\in Q}N(v)\right)\setminus W'$, where $W'$ is the set of terminals of the multiset $\mathcal W\setminus \mathcal Q$. Then the problem $(G,\mathcal W)$ is equivalent to $(G\setminus N,\mathcal W\setminus\mathcal Q)$. So, we may assume that $s_i\neq t_i$, for $1\leqslant i\leqslant k$.

Next, let $\mathcal F$ be the multiset of all pairs of type $(v,w)$ from $\mathcal W$ such that $vw$ is an edge of $G$. Furthermore, let $F=\bigcup_{(v,w)\in\mathcal F}\{v,w\}$ and $N=\left(\bigcup_{v\in F}N(v)\right)\setminus W'$, where $W'$ is the set of terminals of the multiset $\mathcal W\setminus \mathcal F$. Then the problem $(G,\mathcal W)$ is equivalent
to $(G\setminus N,\mathcal W\setminus\mathcal F)$. So, we may also assume that $s_it_i$ is not an edge of $G$, for $1\leqslant i\leqslant k$.

Let $W=\{w_1,w_2,\ldots,w_s\}$ (here $W$ is a set), $W_j=\{i\,|\,s_i=w_j\mbox{ or }t_i=w_j\}$ and $k_j=|W_j|$, for $1\leqslant j\leqslant s$. Now, let us assume that $(G,\mathcal W)$ has a solution $\mathcal P$. Then for each $j$, $w_j$ is an endnode of $k_j$ paths from $\mathcal P$ and the set of neighbors $N_j$ of $w_j$ on these paths is contained in $N(w_j)$ (note that $|N_j|=k_j$). Let $N=\bigcup_{j=1}^s N_j$. If $v\in N_j\cap N_i$, for some $j\neq i$, then $v$ is on path of $\mathcal P$ of length 2, $v$ has no other neighbors in $N$ and is not contained in $N_t$, for $t\not\in\{i,j\}$; if for $v',v''\in N$, $v'v''$ is an edge of $G$, then $v'$ and $v''$ are on the same path of a solutions of $(G,\mathcal W)$ (these two conditions are later called {\it conditions} $(*)$). Let $S$ be the set of vertices that are in more than one of the sets $N_i$. Then a solution of $(G,\mathcal W)$ gives a solution of  $(G',\mathcal W')$, where $G'$ is the graph obtained from $G$ by removing all vertices from $W\cup S$ and all neighbors of vertices from $S$, and $\mathcal W'$ is certain set of pairs of vertices of $N\setminus S$, that satisfies conditions of Theorem \ref{k-IndPaths}. So, now it is clear that to solve $(G,\mathcal W)$ it is enough to do the following. For every $1\leqslant j\leqslant k$ choose a $k_j$-element subset $N_j$ of $N(w_j)$, check if these subsets satisfy the conditions $(*)$, build the corresponding problem $(G',\mathcal W')$ and use Theorem \ref{k-IndPaths} to solve it. If for at least one choice of the subsets $N_j$ the corresponding problem $(G',\mathcal W')$ has a solution we return YES, and otherwise return NO.

We have $|N(w_j)|=\mathcal O(n)$ and $|N_j|=k_j$, so the subset $N_j$ can be chosen in $\mathcal O(n^{k_j})$ ways. Hence, all subsets $N_j$, for $1\leqslant j\leqslant s$, can be chosen in $\mathcal O(n^{2k})$ ways, since $\sum_{j=1}^sk_j=2k$. So, in our algorithm we solve $\mathcal O(n^{2k})$ problems of type $(G',\mathcal W')$ (using Theorem \ref{k-IndPaths}), and hence its running time is $\mathcal O(n^{4k+6})$.
\end{proof}

Now, we list some problems that can be reduced to the problem $(G,\mathcal W)$ from the previous theorem. These reductions are simple and given in more details in \cite{IDP-clawPoly} and \cite{IDP-clawFPT} (there the reductions are given for claw-free graphs, but they work equally well in general). Throughout $k$ is fixed.

\subsubsection{$k$-\textsc{in-a-Path}}

The problem $k$-\textsc{in-a-Path} is to decide whether there is a chordless path in $G$ that contains given vertices $v_1,v_2,\ldots,v_k$ of $G$. To solve this problem it is enough to solve $k!$ problems $(G,\mathcal W)$, where $\mathcal W=\{(v_{\sigma(1)},v_{\sigma(2)}),(v_{\sigma(2)},v_{\sigma(3)}),\ldots,(v_{\sigma(k-1)},v_{\sigma(k)})\}$ and $\sigma$ is a permutation of $\{1,2,\ldots,k\}$. So, by Theorem \ref{k-GeneralIndPaths} the problem $k$-\textsc{in-a-Path} can be solved in time $\mathcal O(n^{4k+6})$ for graphs in $\mathcal C$.

\subsubsection{$H$-\textsc{Anchored Induced Topological Minor}}

Let $H$ be a fixed graph with vertex set $\{x_1,x_2,\ldots,x_k\}$, $l$ edges and $s$ isolated vertices
(then $l\leqslant \binom{k}{2}$ and $s\leqslant k$). The problem $H$-\textsc{Anchored Induced Topological Minor} is to decide
whether for the given vertices $v_1,v_2,\ldots,v_k$ of $G$ there exists an induced subgraph of $G$ that is isomorphic to a subdivision of $H$ such that the isomorphism maps $v_i$ to $x_i$, for $1\leqslant i\leqslant k$.
We may assume that if $x_i$ is an isolated vertex of $H$, then $v_i$ is not adjacent to $v_j$ for $j\neq i$, since
otherwise the problem clearly does not have a solution.
This problem is equivalent to $(G,\mathcal W)$, where $\mathcal W=\{(v_i,v_j)\,|\,x_ix_j\ \mathrm{ is\ an\ edge\ of }\ H\}\cup\{(v_i,v_i)\,|\,x_i\ \mathrm{ is\ an\ isolated\ vertex\ of }\ H\}$, so by Theorem \ref{k-GeneralIndPaths} it can be solved in time $\mathcal O(n^{4(l+s)+6})$ for graphs in $\mathcal C$.

\subsubsection{$H$-\textsc{Induced Topological Minor}}

Let $H$ be a fixed graph with vertex set $\{x_1,x_2,\ldots,x_k\}$, $l$ edges and $s$ isolated vertices (then $l\leqslant \binom{k}{2}$ and $s\leqslant k$). The problem $H$-\textsc{Induced Topological Minor} is to decide whether $G$ contains (as an induced subgraph) a subdivision of $H$. To solve this problem it is enough to solve $H$-\textsc{Anchored Induced Topological Minor} problem for any $k$ distinct  vertices of $G$. So this problem can be solved in time $\mathcal O(n^{k+4(l+s)+6})$ for graphs in~$\mathcal C$.

\subsubsection{$k$-\textsc{in-a-Tree}}

The problem $k$-\textsc{in-a-Tree} is to decide whether there is an induced tree of $G$ that contains given vertices $v_1,v_2,\ldots,v_k$ of $G$. The problem is trivial for $k=1$, so we may assume that $k\geq 2$. Let us suppose that this problem has a solution, and let $T'$ be a minimal tree (w.r.t.\ inclusion) that contains vertices $v_1,v_2,\ldots,v_k$. Then all leaves of $T'$ are from $\{v_1,v_2,\ldots,v_k\}$, and hence $T'$ is an induced subdivision of a tree $T$ that satisfies: all leaves and vertices of degree 2 of $T$ are from $\{v_1,v_2,\ldots,v_k\}$ and all vertices from $\{v_1,v_2,\ldots,v_k\}$ are vertices of $T$. Note that $T$ has at most $2k-2$ vertices. Indeed, if $T$ has $e$ edges, $v$ vertices, $a$ leaves and $b$ vertices of degree 2, then $2v-2=2e=\sum_{u\in V(T')}\deg u\geq a+2b+3(v-a-b)$, and hence $v\leqslant 2a+b-2\leqslant 2k-2$.

So, to solve $k$-\textsc{in-a-Tree} for a graph $G\in\mathcal C$ (and its vertices $\{v_1,v_2,\ldots,v_k\}$) it is enough to do the following. First, we find the set $\mathcal S$ of all non-isomorphic trees $T$ with at least $k$, but at most $2k-2$ vertices, and such that the total number of leaves and degree 2 vertices of $T$ is at most $k$. Then $|\mathcal S|$ is bounded by a constant (depending on $k$), and hence the set $\mathcal S$ can be found in constant time. Now, for each $T\in \mathcal S$ we do the following. First, we choose $s-k$ vertices from $V(G)\setminus\{v_1,v_2,\ldots,v_k\}$, where $s$ is the number of vertices of $T$, and label them with $v_{k+1},v_{k+2},\ldots,v_s$ arbitrarily. This can be done in time $\mathcal O(n^{s-k})=\mathcal O(n^{k-2})$. Let $a$ be the number of leaves and $b$ the number of degree 2 vertices of $T$. Now, we assign labels $\{1,2\ldots,s\}$ to vertices of $T$ (we build several instances of such assignments) and solve the appropriate $T$-\textsc{Anchored Induced Topological Minor} problem. First, we choose  a set $A$ of $a$ distinct vertices from $\{v_1,v_2,\ldots,v_k\}$ and then label the leaves of $T$ with a permutation of the appropriate $a$ indices. Next, we choose a set $B$ of $b$ distinct vertices from $\{v_1,v_2,\ldots,v_k\}\setminus A$ and then label the degree 2 vertices of $T$ with a permutation of the appropriate $b$ indices. The remaining vertices of $V(T)$  (which are all of degree at least 3) are then label with a permutation of indices of elements from $\{v_1,v_2,\ldots,v_s\}\setminus(A\cup B)$. So, the total number of $T$-\textsc{Anchored Induced Topological Minor} problem that is obtained is bounded by a constant (depending on $k$), and each of them can be solved in time $\mathcal O(n^{4(s-1)+6})=\mathcal O(n^{8k-6})$.

Hence, the problem $k$-\textsc{in-a-Tree} can be solved in time $\mathcal O(n^{k-2}\cdot n^{8k-6})=\mathcal O(n^{9k-8})$ for graphs in~$\mathcal C$.

\section{When $k$ is part of the input}\label{sec:NPC}

In the previous section we proved that the problems $k$-\textsc{Induced Disjoint Paths}, $k$-\textsc{in-a-Path} and $k$-\textsc{in-a-Cycle} are polynomially solvable on $\mathcal C$ for any fixed $k$. Using the reductions similar to the ones
used in  \cite{IDP-clawPoly}, we prove that these problems are NP-complete in the class of line graphs of triangle-free chordless graphs when $k$ is part of the input, and hence in $\mathcal C$.

\begin{theorem}
The \textsc{Induced Disjoint Paths} problem is $\mathrm{NP}$-complete for the class of line graphs of triangle-free chordless graphs.
\end{theorem}

\begin{proof}
We can check in polynomial time if a given collection of paths is a solution of an instance of the \textsc{Induced Disjoint Paths} problem. Hence, this problem is in $\mathrm{NP}$. To prove the $\mathrm{NP}$-completeness of this problem in the class of line graphs of triangle-free chordless graphs, we reduce from the \textsc{Disjoint Paths} problem, which is $\mathrm{NP}$-complete  \cite{Karp-DisjointNP}.

Let $G$ be a graph and $(s_1,t_1),(s_2,t_2),\ldots,(s_k,t_k)$ disjoint pairs of vertices from $G$ that are terminals of an instance $\mathcal S$ of the \textsc{Disjoint Paths} on $G$. Also, let $G'$ be the graph obtained from $G$ by subdividing each of its edges once. Then $G'$ is chordless and triangle-free and $\mathcal S$ is equivalent with the instance $\mathcal S'$ of the \textsc{Disjoint Paths} problem on $G'$ with the same terminals.
Let $G''$ be the graph obtained from $G'$ by adding new vertices $s_i'$ and $t_i'$, and edges $s_i's_i$
and $t_it_i'$, for $1\leqslant i\leqslant k$ (the graph $G''$ remains chordless and triangle-free). Then $\mathcal S'$ has a solution if and only if the problem $(L(G''),\mathcal W)$, where $\mathcal W=\{(s_i's_i,t_i't_i)\,|\,1\leqslant i\leqslant k\}$, has a solution. This completes our reduction, since $L(G'')$ is the line graph of a triangle-free chordless graph.
\end{proof}

A graph is {\it cubic} if the degree of each of its vertices is 3.
A path of a graph $G$ is {\it Hamiltonian} if it contains all vertices of $G$. The \textsc{Hamiltonian Path} problem is to decide whether the given graph has a Hamiltonian path.

\begin{theorem}
The $k$-\textsc{in-a-Path} problem and the $k$-\textsc{in-a-Cycle} problem are $\mathrm{NP}$-complete in the class of line graphs of triangle-free chordless graphs, when $k$ is part of the input.
\end{theorem}

\begin{proof}
We can check in polynomial time if a given path (resp.\ cycle) is a solution of an instance of the $k$-\textsc{in-a-Path} (resp.\ $k$-\textsc{in-a-Cycle}) problem, and hence both problems are in $\mathrm{NP}$. To prove the $\mathrm{NP}$-completeness of these problem in the class of line graphs of triangle-free chordless graphs, we reduce from the \textsc{Hamiltonian Path} problem, which is $\mathrm{NP}$-complete for cubic graphs (see \cite[problem GT39]{GJ-CubicHamilton}).

Let $G$ be a cubic graph and $\mathcal S$ the instance of the \textsc{Hamiltonian Path} problem for $G$.
We build the graph $G'$ from $G$ as follows. For each $u\in V(G)$ we build a triangle $u_1u_2u_3$ of $G'$, and $\{u_i\,|\, 1\leqslant i\leqslant 3,\,u\in V(G)\}$ is the set of vertices of $G'$. Next, for each edge $uv$ of $G$ we build exactly one edge $u_iv_j$, for some $i,j\in\{1,2,3\}$, such that each vertex of $G'$ is of degree 3 (so, for $u\in V(G)$, $i\in\{1,2,3\}$, $u_i$ is adjacent to $u_j$, for $j\in\{1,2,3\}\setminus\{i\}$, and to a vertex $v_s$ for some $v\in N(u)$ and $s\in\{1,2,3\}$).
It is proved in  \cite{IDP-clawPoly} that $\mathcal S$ has a solution if and only if $G'$ has a path that passes through all edges from $\{u_1u_2\,|\,u\in V(G)\}$. We call the latter problem $\mathcal S'$. Now, let $G''$ be the graph obtained from $G'$ by subdividing each edge of $G'$, that is, each edge $xy$ of $G'$ is replaced with the path $xv_{xy}y$ (where $v_{xy}$ is of degree 2 in $G''$). Note that $G''$ is triangle-free and chordless. Then it is clear that $\mathcal S'$ has a solution if and only if $G''$ has a path that passes through all edges from $\{u_1v_{u_1u_2},u_2v_{u_1u_2}\,|\,u\in V(G)\}$. We call the later problem $\mathcal S''$. Since the paths of $G''$ are in one-to-one correspondence with induced paths of $L(G'')$, the problem $\mathcal S''$ has a solution if and only if there is an induced path of $L(G'')$ that passes through all vertices from $\{u_1v_{u_1u_2},u_2v_{u_1u_2}\,|\,u\in V(G)\}\subseteq V(L(G''))$.  This completes our reduction, since $L(G'')$ is the line graph of a triangle-free chordless graph.

For the $k$-\textsc{in-a-Cycle} problem analogous reduction is made from the \textsc{Hamiltonian Cycle} problem, which is again $\mathrm{NP}$-complete for cubic graphs (see \cite{GJ-CubicHamilton}).
\end{proof}


\begin{thebibliography}{99}

\bibitem{DBLP:journals/algorithmica/BelmonteGHHKP14}
R. Belmonte, P.A. Golovach, P. Heggernes, P. van~'t Hof, M.
  Kaminski, and D. Paulusma.
\newblock Detecting fixed patterns in chordal graphs in polynomial time.
\newblock {\em Algorithmica}, 69(3):501--521, 2014.
 
%
\bibitem{bienstock}
D. Bienstock.
\newblock On the complexity of testing for odd holes and induced odd paths.
\newblock{\em Discrete Mathematics}, 90:85--92, 1991.
\newblock See also Corrigendum by B. Reed, {\em Discrete Mathematics},
102:109, 1992.

\bibitem{BruhnS12}
H. Bruhn and A.~Saito.
\newblock Clique or hole in claw-free graphs.
\newblock {\em Journal of Combinatorial Theory, Series {B}}, 102(1):1--13,
  2012.



\bibitem{fast2j}
P. Charbit, M. Habib, N. Trotignon, K. Vu\v{s}kovi\'c.
\newblock Detecting 2-joins faster.
\newblock {\em Journal of Discrete Algorithms}, 17: 60-66, 2012.

\bibitem{chudnovsky.seymour:theta}
M.~Chudnovsky and P.D. Seymour.
\newblock The three-in-a-tree problem.
\newblock {\em Combinatorica}, 30(4):387--417, 2010.

\bibitem{nicolas.d.p:fourTree}
N.~Dehry, C.~Picouleau, and N.~Trotignon.
\newblock The four-in-a-tree problem in triangle-free graphs.
\newblock {\em Graphs and Combinatorics}, 25:489--502, 2009.



\bibitem{twf-p1}
E. Diot, M. Radovanovi\'c, N. Trotignon, K. Vu\v{s}kovi\'c.
\newblock The (theta,wheel)-free graphs Part I: ony-prism and only pyramid graphs. To
appear in {\it Journal of Combinatorial Theory, Series B}. arXiv:1504.01862



\bibitem{DBLP:journals/algorithmica/FellowsKMP95}
M.R. Fellows, J. Kratochv{\'{\i}}l, M. Middendorf, F. Pfeiffer.
\newblock The complexity of induced minors and related problems.
\newblock {\em Algorithmica}, 13(3):266--282, 1995.


\bibitem{IDP-clawPoly}
J. Fiala, M. Kami\'nski, B. Lidick\'y, D. Paulusma.
\newblock The $k$-in-a-Path Problem for Claw-free Graphs.
\newblock {\em Algorithmica}, 62: 499-519, 2012.


\bibitem{DBLP:conf/swat/2012}
F.V. Fomin and P. Kaski, editors.
\newblock {\em Algorithm Theory - {SWAT} 2012 - 13th Scandinavian Symposium and
  Workshops, Helsinki, Finland, July 4-6, 2012. Proceedings}, volume 7357 of
  {\em Lecture Notes in Computer Science}. Springer, 2012.


\bibitem{GJ-CubicHamilton}
M.R. Garey, D.S. Johnson.
\newblock Computers and Intractability.
\newblock Freeman, New York ,1979.


\bibitem{DBLP:conf/swat/GolovachPL12}
P.A. Golovach, D. Paulusma, and E.J. van Leeuwen.
\newblock Induced disjoint paths in {AT}-free graphs.
\newblock In Fomin and Kaski \cite{DBLP:conf/swat/2012}, pages 153--164.


\bibitem{IDP-clawFPT}
P.A. Golovach, D. Paulusma, E.J. van Leeuewen.
\newblock Induced disjoint paths in claw-free graphs.
\newblock {\em SIAM Journal of Discrete Mathematics}, 29 (1): 348-375, 2015.


\bibitem{DBLP:journals/tcs/GolovachPL16}
P.A. Golovach, D. Paulusma, and E.J. van Leeuwen.
\newblock Induced disjoint paths in circular-arc graphs in linear time.
\newblock {\em Theor. Comput. Sci.}, 640:70--83, 2016.


\bibitem{Karp-DisjointNP}
R.M. Karp.
\newblock On the complexity of combinatorial problems.
\newblock {\em Networks}, 5: 45-68, 1975.


\bibitem{DBLP:journals/jcss/KawarabayashiK12}
K. Kawarabayashi and Y. Kobayashi.
\newblock A linear time algorithm for the induced disjoint paths problem in
  planar graphs.
\newblock {\em J. Comput. Syst. Sci.}, 78(2):670--680, 2012.


\bibitem{lehot}
P.G.H. Lehot.
\newblock An optimal algorithm to detect a line graph and output its root graph.
\newblock {\em Journal of the Association for Computing Machinery}, 21(4):569-575, 1974.

%


\bibitem{leveque.lmt:detect}
B.~L{\'e}v{\^e}que, D.~Lin, F.~Maffray, and N.~Trotignon.
\newblock Detecting induced subgraphs.
\newblock {\em Discrete Applied Mathematics}, 157:3540--3551, 2009.

\bibitem{nicolas.wei:kTree}
W.~Liu and N.~Trotignon.
\newblock The $k$-in-a-tree problem for graphs of girth at least~$k$.
\newblock {\em Discrete Applied Mathematics}, 158:1644--1649, 2010.


\bibitem{twf-p2}
M. Radovanovi\'c, N. Trotignon, K. Vu\v{s}kovi\'c.
\newblock The (theta,wheel)-free graphs Part II: structure theorem. To
appear in {\it Journal of Combinatorial Theory, Series B}. arXiv:1703.08675

\bibitem{twf-p3}
M. Radovanovi\'c, N. Trotignon, K. Vu\v{s}kovi\'c.
\newblock The (theta,wheel)-free graphs Part III: cliques, stable sets and coloring. To
appear in {\it Journal of Combinatorial Theory, Series B}. arXiv:1707.04205


\bibitem{RS:GraphMinor13}
N. Robertson and P.D. Seymour.
\newblock Graph minors. XIII. The disjoint paths problem.
\newblock {\em Journal of Combinatorial Theory, Series B}, 63: 65-110, 1995.

\bibitem{rous}
N.D. Roussopoulos.
\newblock A max $\{ m,n\}$ algorithm for determining the graph $H$ from its line graph $G$.
\newblock {\em Information Processing Letters}, 2(4):108-112, 1973.

\bibitem{tarjan}
R.E. Tarjan.
\newblock Decomposition by clique separators.
\newblock {\em Discrete Mathematics}, 55: 221-232, 1985.


\bibitem{nicolas.kristina:two}
N.~Trotignon, K.~Vu{\v s}kovi{\'c}.
\newblock Combinatorial optimization with 2-joins.
\newblock {\em Journal of Combinatorial Theory, Series B}, 102:153-185, 2012.


\end{thebibliography}
\end{document}